\documentclass[reqno]{amsart}
\usepackage{amssymb}
\usepackage{amsmath}
\usepackage{amsthm}
\usepackage[usenames]{color}

\newtheorem{theorem}{Theorem}[section]

\newtheorem{lemma}[theorem]{Lemma}
\newtheorem{proposition}[theorem]{Proposition}
\newtheorem{remark}[theorem]{Remark}

\numberwithin{equation}{section}
\usepackage{hyperref}
\usepackage{marginnote}

\usepackage{color}

\begin{document}

\title[Euler-Maxwell limit for VMB system]{Compressible Euler-Maxwell limit for global smooth solutions to the Vlasov-Maxwell-Boltzmann system}

\author[R.-J. Duan]{Renjun Duan}
\address[R.-J. Duan]{Department of Mathematics, The Chinese University of Hong Kong, Shatin, Hong Kong,
        People's Republic of China}
\email{rjduan@math.cuhk.edu.hk}

\author[D.-C. Yang]{Dongcheng Yang}
\address[D.-C. Yang]{Department of Mathematics, The Chinese University of Hong Kong, Shatin, Hong Kong,
        People's Republic of China}
\email{dcyang@math.cuhk.edu.hk}

\author[H.-J. Yu]{Hongjun Yu}
\address[H.-J. Yu]{School of Mathematical Sciences,
         South China Normal University, Guangzhou 510631, People's Republic of China}
\email{yuhj2002@sina.com}

%\date{\today}

\begin{abstract}
Two fundamental models in plasma physics are given by the Vlasov-Maxwell-Boltzmann system and the compressible Euler-Maxwell system which both capture the complex dynamics of plasmas under the self-consistent electromagnetic interactions at the kinetic and fluid levels, respectively. It has remained a long-standing open problem to rigorously justify the hydrodynamic limit from the former to the latter as the Knudsen number $\varepsilon$ tends to zero. In this paper we give an affirmative answer to the problem for smooth solutions to both systems near constant equilibrium in the whole space in case when only the dynamics of electrons is taken into account. The explicit rate of convergence in $\varepsilon$ over an almost global time interval is also obtained for well-prepared data. For the proof, one of main difficulties occurs to the cubic growth of large velocities due to the action of the classical transport operator on local Maxwellians and we develop new $\varepsilon$-dependent energy estimates basing on the macro-micro decomposition to characterize the asymptotic limit in the compressible setting.
\end{abstract}


\maketitle

\setcounter{tocdepth}{1}
\tableofcontents
\thispagestyle{empty}

\section{Introduction}
The Vlasov-Maxwell-Boltzmann (called VMB in the sequel for simplicity) system is an important  model in plasma 
physics for describing the time evolution of dilute charged particles, such as electrons and ions, 
under the influence of their self-consistent internally generated Lorentz forces, cf. \cite{Chapman,Saint-Raymond}. In the physical situation, the ion mass is usually
much heavier than the electron mass so that the electrons move faster than the ions. Hence, the ions are often described by a fixed ion background $n_b(x)$, and only
the motion of electrons need to be determined. In such one-species case, the scaled VMB system takes the form of
\begin{equation}
\label{1.1}
\left\{
\begin{array}{rl}
&\partial_{t}F+v\cdot\nabla_{x}F-(E+v\times B)\cdot\nabla_{v}F=\frac{1}{\varepsilon}Q(F,F),
\\
&\partial_{t}E-\nabla_{x}\times B=\int_{\mathbb{R}^{3}} vF\,dv,
\\
&\partial_{t}B+\nabla_{x}\times E=0,
\\
&\nabla_{x}\cdot E=n_b-\int_{\mathbb{R}^{3}} F\,dv,\quad \nabla_{x}\cdot B=0.
\end{array} \right.
\end{equation}
Here, the parameter $\varepsilon>0$ is the Knudsen number which is proportional to the mean free path and is assumed
to be small. The unknowns are $F=F(t,x,v)\geq0$ standing for the density
distribution function for the electron particles with position $x=(x_{1},x_{2},x_{3})\in\mathbb{R}^{3}$ and velocity $v=(v_{1},v_{2},v_{3})\in\mathbb{R}^{3}$
at time $t\geq0$ and $E=E(t,x)=(E_1,E_2,E_3)(t,x)$ and $B=B(t,x)=(B_1,B_2,B_3)(t,x)$ denoting the self-consistent electric field and magnetic field, respectively.
We have omitted the explicit dependence of the solution $(F,E,B)$ on $\varepsilon$ for brevity whenever there is no confusion.

Throughout this paper, $n_{b}>0$ is assumed to be a constant denoting
the spatially uniform density of the ionic background. Take $n_{b}=1$ without loss of generality. Note that
the other physical constants such as the charge and mass of electrons and the speed of light in system \eqref{1.1} have been normalized to be one
since they do not cause essential mathematical difficulties in our one-fluid problem; cf. \cite{Saint-Raymond,Guo-2003}.

The Boltzmann collision operator $Q(\cdot,\cdot)$ in \eqref{1.1} is assumed to be for the hard-sphere model, taking the following non-symmetric bilinear
form of
\begin{equation}
\label{1.2}
Q(F_1,F_2)(v)=\int_{\mathbb{R}^{3}}\int_{\mathbb{S}^{2}}|(v-v_{*})\cdot\omega|\{F_1(v')F_2(v'_{*})
-F_1(v)F_2(v_{*})\}\,d\omega\,dv_{*},
\end{equation}
where $\omega\in \mathbb{S}^{2}$ is a unit vector in $\mathbb{R}^{3}$, and the velocity pairs $(v,v_{*})$ before collisions and $(v',v'_{*})$ after collisions are given by
$$
v'=v-[(v-v_{*})\cdot\omega]\omega, \quad  v_{*}'=v_{*}+[(v-v_{*})\cdot\omega]\omega,
$$
in terms of the conservation of momentum and kinetic energy:
%\begin{align*}
$v+v_{*}=v'+v'_{*}$ and %\quad 
$|v|^{2}+|v_{*}|^{2}=|v'|^{2}+|v'_{*}|^{2}$.
%\end{align*}
We remark that one can consider more complex long-range interaction potentials in the Boltzmann case or even the Landau operator with Coulomb potentials; see \cite{Duan-Yang-Yu}, for instance. The reason why we only focus on the hard sphere model in the current work is to present the essentials of our problem in the non-relativistic case in an easier way and interested readers could work on the other harder situations. 

Corresponding to \eqref{1.1}, the hydrodynamic description for the motion of electrons at the fluid
level is also given by the following compressible Euler-Maxwell system %\eqref{1.3}
which is an important fluid model in plasma physics:
\begin{equation}
\label{1.3}
\left\{
\begin{array}{rl}
\partial_{t}\bar{\rho}+\nabla_{x}\cdot(\bar{\rho}\bar{u})&=0,
\\
\partial_{t}(\bar{\rho}\bar{u})+\nabla_{x}\cdot(\bar{\rho}\bar{u}\otimes\bar{u})
+\nabla_{x}\bar{p}&=-\bar{\rho}(\bar{E}+\bar{u}\times \bar{B}),
\\
\partial_{t}\bar{E}-\nabla_{x}\times \bar{B}&=\bar{\rho}\bar{u},
\quad
\partial_{t}\bar{B}+\nabla_{x}\times \bar{E}=0,
\\
\nabla_{x}\cdot \bar{E}=n_b-\bar{\rho},\quad \nabla_{x}\cdot \bar{B}&=0.
\end{array} \right.
\end{equation}
Here the unknowns  are  the electron density $\bar{\rho}=\bar{\rho}(t,x)>0$, the electron velocity $\bar{u}=(\bar{u}_1,\bar{u}_2,\bar{u}_3)(t,x)$, and the electromagnetic field $(\bar{E},\bar{B})=(\bar{E},\bar{B})(t,x)$. Moreover,  
$\bar{p}=K\bar{\rho}^{5/3}$ is the pressure satisfying the power law with the adiabatic exponent $\gamma=\frac{5}{3}$.
We take the  physical constant $K=1$ without loss of generality. 

Both the VMB system \eqref{1.1} and the Euler-Maxwell system \eqref{1.3} can be viewed as two important models in plasma physics which capture the complex dynamics of charged particles due to the self-consistent electromagnetic interactions at the kinetic and fluid levels, respectively.
It has remained a long-standing open problem to justify whether or not the  compressible Euler-Maxwell system \eqref{1.3} can be derived rigorously from
the VMB system \eqref{1.1}, as the Knudsen number $\varepsilon$ goes to zero.
In this paper, we will give an affirmative answer to the issue for smooth solutions of the Cauchy problems near constant equilibrium in the whole space.

There have been extensive mathematical studies on the global well-posedness of the VMB system.
Under the angular cutoff assumption, Guo \cite{Guo-2003}  gave the first proof for constructing the global classical solutions
near global Maxwellians in the torus, and the corresponding
existence result in the whole space was proved by Strain \cite{Strain-2006}.
The  optimal large-time behavior of those global solutions in the whole space
were studied by Duan-Strain \cite{Duan-Strain}. Since then, 
the spatially inhomogeneous perturbation theory of the VMB system with or without angular
cutoff around global Maxwellians was further developed in different settings,
see  \cite{Duan-2010,Duan-Lei,Duan-Liu,Li-Yang} and references therein.

Correspondingly, the global well-posedness of the Euler-Maxwell system has also been well-studied in \cite{Germain-Masmoudi,Ionescu,Ionescu-2014,Ionescu-2016,Deng}.
Precisely, Germain-Masmoudi \cite{Germain-Masmoudi} made use of the method of space-time resonance  to  construct the 
global smooth solutions of the one-fluid Euler-Maxwell model for electrons in the whole space under additional generic conditions on the parameters.
The generic conditions  were removed and a stronger decay was obtained later by Ionescu-Pausader \cite{Ionescu} via
a robust approach, and see also \cite{Deng} for two spatial dimensions. A great progress was made by Guo-Ionescu-Pausader \cite{Ionescu-2016} for constructing the small-amplitude global smooth solutions of the two-fluid Euler-Maxwell system for both ions and electrons with disparate masses in the whole space; see also \cite{Ionescu-2014} for the relativistic case. Interested readers may refer to the introduction of \cite{Ionescu-2016}
for more discussions on the Euler-Maxwell system and its 
connections to many other well-known dispersive PDEs, such as  Euler-Poisson, Zakharov, Zakharov-Kuznetsov, KdV and NLS, etc.
Here we would like to mention \cite{Duan-2011,Feng,Xu}  and references therein for
the global well-posedness of the Euler-Maxwell system with relaxation.

In the absence of  the electromagnetic field, the VMB system is reduced to the Boltzmann equation. There is an enormous literature on the fluid dynamic limits of the Boltzmann equation, so it is hard to give an exhausting list. We mention some of them. For the approach based on the Hilbert or Chapman-Enskog expansions, we refer to \cite{Caflisch,Grad,Guo-Jang-Jiang,Guo-Jang-Jiang-2010,Guo-Huang-Wang,Jiang-Luo,Jiang-Luo-1,Liu-2014}. For the method employing the abstract Cauchy-Kovalevskaya theorem and the spectral analysis of the semigroup generated by the linearized Boltzmann equation,
we refer to  \cite{Nishida,Ukai-Asano,Bardos-1991}.  For the DiPerna-Lions \cite{Diperna-Lions} renormalized weak solutions,
we refer to  \cite{Bardos,Golse-2002,Golse-2004,Masmoudi-2003,Jiang-Masmoudi-1}.
 
There are also some literatures on the hydrodynamic limits of the more complex VMB system in the presence of the electromagnetic field.
In the incompressible regime,  a diffusive expansion to the VMB system in the framework of classical solutions was considered in Jang \cite{Jang-2009}
where the magnetic effect appears only as a higher order. Moreover, a comprehensive and great study was made by  Ars\'enio and Saint-Raymond  \cite{Saint-Raymond} to justify the  various limits towards the incompressible viscous electromagneto-hydrodynamics for the VMB system under the different scalings. In particular, they obtained the limit to the Navier-Stokes-Fourier-Maxwell 
system with the Ohm's law for weak solutions to the VMB system; see also a recent progress \cite{Jiang-Luo-2021} for classical solutions. We also mention  \cite{Jang-2012} for some formal derivations of the fluid equations in the compressible regime.

In what follows we review some specific results that are most related to the topic of this paper in connection
with the compressible Euler-Maxwell limit from the VMB system. For the study on the compressible Euler limit of the
Boltzmann equation, under the angular cutoff assumption, employing the truncated Hilbert expansion, Caflisch \cite{Caflisch} first showed that, if the compressible Euler system has a smooth solution which exists up to a finite time, then there exist corresponding
solutions to the Boltzmann equation with a zero initial condition for the remainder term
in the same time interval such that the Boltzmann solution converges to a smooth local Maxwellian whose fluid dynamical parameters  satisfy the compressible Euler system as $\varepsilon\to 0$. Due to such zero initial condition, the obtained Boltzmann solution may be negative.  

Following the Caflisch's strategy together with a new $L^{2}-L^{\infty}$ interplay technique initiated by Guo \cite{Guo-2008}, Guo-Jang-Jiang \cite{Guo-Jang-Jiang,Guo-Jang-Jiang-2010} removed the restriction to zero initial condition so that the positivity of the Boltzmann solution can be guaranteed. The $L^{2}-L^{\infty}$ interplay technique combined with the Hilbert expansion was employed later to treat the Vlasov-Poisson-Boltzmann system by Guo-Jang \cite{Guo-Jang} which seems a great progress since the Euler-Poisson system can be first rigorously derived from a kinetic plasma model with the self-consistent potential force. Unfortunately, such $L^{2}-L^{\infty}$ interplay technique
fails for the VMB system, essentially due to the complex structure of the system, for instance, the Glassey-Strauss representation of the electromagnetic field no longer holds true for the non-relativistic case; see the detailed explanations of this point at the beginning of Section \ref{sec.3.3}. 

In the relativistic situation, more results on the compressible fluid limit have been obtained, for instance, for the relativistic Boltzmann equation by Speck-Strain \cite{Speck-Strain} and for the relativistic VMB system by Guo-Xiao \cite{Guo-Xiao}; see also two very recent preprints \cite{OWX1,OWX2} for the case when the Boltzmann operator is replaced by the Landau operator. For the proof, the main benefit in the relativistic situation is that  the relativistic transport velocity is bounded in contrast with the fact that the classical velocity is unbounded. In particular, the application of Glassey-Strauss representation is a necessary tool and also \cite{OWX1,OWX2} finds that they can disregard the use of the $L^{2}-L^{\infty}$ technique and rather develop an energy method in smooth Sobolev spaces weighted with a special time-velocity function that can create the extra large-velocity dissipation; similar mechanism was observed in some earlier works \cite{Duan-Zhao-h, Duan-Zhao} as well as in a recent work \cite{DYY-CMP} in a more general way.  
%However, it seems an impossible task to adopt the methods in either \cite{Guo-Xiao} or \cite{OWX1,OWX2}  to treat the compressible fluid limit for the VMB or VML system in the classical case because the cubic growth of large velocities due to the action of the classical transport operator on local Maxwellians remains an essential obstacle.  

We notice that the $L^{2}-L^{\infty}$ techinque was further developed to deal with the fluid limit problems in the half-space for the Boltzmann equation
with different types of boundary conditions including specular reflection, diffuse reflection and even the mixed Maxwell reflection
by Guo-Huang-Wang \cite{Guo-Huang-Wang}  and Jiang-Luo-Tang \cite{Jiang-Luo,Jiang-Luo-1}. 

It is worth noting that the $L^{2}-L^{\infty}$ framework in \cite{Guo-Jang-Jiang,Guo-Jang-Jiang-2010} does not apply to both the Boltzmann equation without cutoff 
and the Landau equation, since those long-range collision operators expose the velocity diffusion property so that  $L^\infty$ estimates are hard to obtain without using Sobolev embeddings. Recently, the same authors of this article \cite{Duan-Yang-Yu} succeeded in developing a new method in terms of the pure high order energy estimates
instead of the $L^{2}-L^{\infty}$ interplay technique  to deal with the compressible Euler limit for smooth solutions to the Landau equation
with Coulomb potentials in the whole space. 

In this paper, we will rigorously justify
the hydrodynamic limit from the VMB system \eqref{1.1} to the compressible Euler-Maxwell system \eqref{1.3} in the whole space.
Specifically, we can prove that there
exists the smooth solutions of the VMB system near a smooth local Maxwellian converging global-in-time to the corresponding smooth solutions of the compressible Euler-Maxwell system, as the Knudsen number $\varepsilon$ goes to zero. The result is proved by elaborate energy estimates.  More details will be specified later on.

The rest of this paper is organized as follows. In Section \ref{sec.2}, we present the macro-micro decomposition of the solution for the 
VMB system in order to study the compressible Euler-Maxwell limit for global smooth solutions to the VMB system.
In Section \ref{sec.3}, we present the main results
Theorem  \ref{thm3.1} for the compressible Euler-Maxwell limit.
For conveniences of readers we also list key points for the strategy of the proof through the paper.
To proceed the proof of main result, we first prepare in Section \ref{seca.4} some basic estimates. Section \ref{seca.5}
is the main part of the proof for establishing the a priori estimates on both the fluid part and the kinetic part. 
 In the end, we prove our main Theorem \ref{thm3.1}
in Section \ref{seca.6}.
\medskip

\noindent{\it Notations.} 
For convenience of presentation in the paper, we shall use $\langle \cdot , \cdot \rangle$  to denote the standard $L^{2}$ inner product in $\mathbb{R}_{v}^{3}$
with its corresponding $L^{2}$ norm $|\cdot|_2$, and $( \cdot , \cdot )$ to denote $L^{2}$ inner product in either
$\mathbb{R}^3_{x}$ or $\mathbb{R}^3_{x}\times \mathbb{R}_{v}^{3}$  with its corresponding $L^{2}$ norm $\|\cdot\|$. 
We also use $\|\cdot\|_{W^{k,p}}$ to denote standard spatial Sobolev norm of order $k$.
The norm of $\nabla^{k}_x f$ means the sum of the norms of functions $\partial^{\alpha}f$ with $|\alpha|=k$.
Let  $\alpha$ and $\beta$ be multi indices $\alpha=(\alpha_{1},\alpha_{2},\alpha_{3})$ and $\beta=(\beta_{1},\beta_{2},\beta_{3})$,
respectively. Denote $\partial_{\beta}^{\alpha}=\partial_{x_{1}}^{\alpha_{1}}\partial_{x_{2}}^{\alpha_{2}}\partial_{x_{3}}^{\alpha_{3}}
\partial_{v_{1}}^{\beta_{1}}\partial_{v_{2}}^{\beta_{2}}\partial_{v_{3}}^{\beta_{3}}$.
If each component of $\beta$ is not greater than the corresponding one  of
$\overline{\beta}$, we use the standard notation
$\beta\leq\overline{\beta}$. And $\beta<\overline{\beta}$ means that
$\beta\leq\overline{\beta}$ and $|\beta|<|\overline{\beta}|$.
$C^{\bar\beta}_{\beta}$ is the usual  binomial coefficient.
Throughout the paper, generic positive constants are denoted  by either $c$ or $C$ which may change line by line.
The notation  $A\approx B$ means that there exists $C>1$ such that $C^{-1}B\leq A\leq CB$.
We also use $\langle v\rangle=\sqrt{1+|v|^2}$.

\section{Macro-Micro decomposition}\label{sec.2}
In order to study the hydrodynamic limit from the VMB system  \eqref{1.1} to the compressible Euler-Maxwell system \eqref{1.3}, 
we will present in this section the macro-micro decomposition of the solution for the 
VMB system with respect to the local Maxwellian, that was initiated by
Liu-Yu \cite{Liu-Yu} and developed by Liu-Yang-Yu \cite{Liu-Yang-Yu} for the Boltzmann equation.

Indeed, associated with a solution $F(t,x,v)$ to the VMB system \eqref{1.1}, there are five macroscopic (fluid) quantities: 
the mass density $\rho(t,x)>0$, momentum $\rho(t,x)u(t,x)$, and
energy density $e(t,x)+\frac 12|u(t,x)|^2$ defined by
\begin{equation}
\label{2.1}
\left\{
\begin{array}{rl}
\rho(t,x)&\equiv\int_{\mathbb{R}^{3}}\psi_{0}(v)F(t,x,v)\,dv,
\\
\rho(t,x) u_{i}(t,x)&\equiv\int_{\mathbb{R}^{3}}\psi_{i}(v)F(t,x,v)\,dv, \quad \mbox{for $i=1,2,3$,}
\\
\rho(t,x)[e(t,x)+\frac{1}{2}|u(t,x)|^{2}]&\equiv\int_{\mathbb{R}^{3}}\psi_{4}(v)F(t,x,v)\,dv,
\end{array} \right.
\end{equation}
where $\psi_{i}(v)$ $(i=0,1,2,3,4)$ are five collision invariants given by
$$
\psi_{0}(v)=1, \quad \psi_{i}(v)=v_{i}~(i=1,2,3),\quad \psi_{4}(v)=\frac{1}{2}|v|^{2},
$$
and satisfy
\begin{equation}
\label{2.2}
\int_{\mathbb{R}^{3}}\psi_{i}(v)Q(F,F)\,dv=0,\quad \mbox{for $i=0,1,2,3,4$.}
\end{equation}
We define the local Maxwellian $M$ associated to the solution $F$ to the VMB system \eqref{1.1}
in terms of the fluid quantities by
\begin{equation}
\label{2.3}
M\equiv M_{[\rho,u,\theta]}(t,x,v):=\frac{\rho(t,x)}{\sqrt{(2\pi R\theta(t,x))^{3}}}\exp\big\{-\frac{|v-u(t,x)|^{2}}{2R\theta(t,x)}\big\}.
\end{equation}
Here $\theta(t,x)>0$ is the temperature which is related to the internal energy $e(t,x)$ by
$e=\frac{3}{2}R\theta=\theta$ with the gas constant $R=2/3$ taken for convenience, and $u(t,x)=(u_{1}(t,x),u_{2}(t,x),u_{3}(t,x))$ is the bulk velocity.

Since the $L^{2}$ inner product in $v\in\mathbb{R}^{3}$ is denoted by
$
\langle h,g\rangle =\int_{\mathbb{R}^{3}}h(v)g(v)\,d v,
$
the macroscopic kernel space is spanned by the following five pairwise-orthogonal
base \begin{equation}
\label{2.4}
\left\{
\begin{array}{rl}
&\chi_{0}(v)=\frac{1}{\sqrt{\rho}}M,
\quad \chi_{i}(v)=\frac{v_{i}-u_{i}}{\sqrt{R\rho\theta}}M, \quad \mbox{for $i=1,2,3$,}
\\
&\chi_{4}(v)=\frac{1}{\sqrt{6\rho}}(\frac{|v-u|^{2}}{R\theta}-3)M,
\\
&\langle \chi_{i},\frac{\chi_{j}}{M}\rangle=\delta_{ij},\quad \mbox{for ~~$i,j=0,1,2,3,4$},
\end{array} \right.
\end{equation}
where $\delta_{ij}$ being the Kronecker delta.
In view of \eqref{2.4}, we define
the macroscopic projection  $P_{0}$ and the 
microscopic projection $P_{1}$ as
\begin{equation}
\label{2.5}
P_{0}h\equiv\sum_{i=0}^{4}\langle h,\frac{\chi_{i}}{M}\rangle\chi_{i},\quad \mbox{and} \quad P_{1}h\equiv h-P_{0}h,
\end{equation}
respectively, where the operators  $P_{0}$ and  $P_{1}$  are  orthogonal projections, namely, it holds that
$$
P_{0}P_{0}=P_{0},\quad
P_{1}P_{1}=P_{1},\quad
P_{1}P_{0}=P_{0}P_{1}=0.
$$
A function $h(v)$  is said to be microscopic or non-fluid if
\begin{equation}
\label{2.6}
\langle h(v),\psi_{i}(v)\rangle=0, \quad \mbox{for $i=0,1,2,3,4$}.
\end{equation}

Using the notations above,  we decompose the solution $F$ of
\eqref{1.1} into the combination of the local Maxwellian $M$ defined in \eqref{2.3} and the microscopic (non-fluid)
component $G=G(t,x,v)$ as
\begin{equation}
\label{2.7}
F=M+G, \quad P_{0}F=M, \quad P_{1}F=G.
\end{equation}
Thanks to the fact that $Q(M,M)=0$, the first equation of \eqref{1.1} can be written as
\begin{equation}
\label{2.8}
\partial_{t}(M+G)+v\cdot\nabla_{x}(M+G)-(E+v\times B)\cdot\nabla_{v}(M+G)
=\frac{1}{\varepsilon}L_{M}G+\frac{1}{\varepsilon}Q(G,G),
\end{equation}
where the linearized Boltzmann operator $L_{M}$ around the local Maxwellian $M$ is defined as
\begin{equation}
\label{2.9}
L_{M}h:=Q(h,M)+Q(M,h).
\end{equation}
Its null space $\mathcal{N}_M$ is spanned by the macroscopic variables $\chi_{i}(v)~(i=0,1,2,3,4)$.

Let's now decompose the VMB system \eqref{1.1} into the microscopic system and the macroscopic system
in term of \eqref{2.7}.
First note that $P_{1}L_{M}G=L_{M}G$, $P_{1}Q(G,G)=Q(G,G)$,
$P_{1}\partial_{t}M=0$ and  $P_{1}\partial_{t}G=\partial_{t}G$ due to \eqref{2.6} and \eqref{2.5}. 
On the other hand, by tedious and careful calculations, we obtain
\begin{equation*}
P_1[(E+v\times B)\cdot\nabla_{v}F]=(E+v\times B)\cdot\nabla_{v}G.
\end{equation*}
With these facts, the microscopic (non-fluid) system for $G$ is obtained by applying the microscopic projection $P_1$ to \eqref{2.8},
namely
\begin{equation}
\label{2.10}
\partial_{t}G+P_{1}(v\cdot\nabla_{x}G)+P_{1}(v\cdot\nabla_{x}M)-(E+v\times B)\cdot\nabla_{v}G
=\frac{1}{\varepsilon}L_{M}G+\frac{1}{\varepsilon}Q(G,G).
\end{equation}
Since $L_{M}$ is invertible on $\mathcal{N}_{M}^\perp$, we have from \eqref{2.10} that
\begin{equation}
\label{2.11}
G=\varepsilon L^{-1}_{M}[P_{1}(v\cdot\nabla_{x}M)]+L^{-1}_{M}\Theta,
\end{equation}
and
\begin{equation}
\label{2.12}
\Theta:=\varepsilon \partial_{t}G+\varepsilon P_{1}(v\cdot\nabla_{x}G)-\varepsilon(E+v\times B)\cdot\nabla_{v}G-Q(G,G).
\end{equation}

By integration by parts, \eqref{2.1} and the facts that $\nabla_{v}\cdot(E+v\times B)=0$ and $v\cdot(v\times B)=0$, 
we have the following two identities:
%\begin{align*}
%-\int_{\mathbb{R}^{3}}(E+v\times B)\cdot\nabla_{v}F\,dv=0,
%\end{align*}
\begin{equation*}
-\int_{\mathbb{R}^{3}}v[(E+v\times B)\cdot\nabla_{v}F]\,dv=\int_{\mathbb{R}^{3}}(E+v\times B)F\,dv
=\rho(E+u\times B),
\end{equation*}
and
\begin{equation*}
-\int_{\mathbb{R}^{3}}\frac{1}{2}|v|^2(E+v\times B)\cdot\nabla_{v}F\,dv=\int_{\mathbb{R}^{3}}v\cdot(E+v\times B)F\,dv
=\rho u\cdot E.
\end{equation*}
Multiplying \eqref{2.8} by the collision invariants $1,~v,~\frac{1}{2}|v|^2$ respectively,
then integrating the resulting equations with respect to $v$ over $\mathbb{R}^{3}$ and using two identities above,
we have the following macroscopic (fluid) system:
\begin{equation}
\left\{
\begin{array}{rl}
\label{2.13}
\partial_{t}\rho+\nabla_{x}\cdot(\rho u)&=0,
\\
\partial_{t}(\rho u)+\nabla_{x}\cdot(\rho u\otimes u)+\nabla_{x}p+\rho(E+u\times B)&=-\int_{\mathbb{R}^{3}} v\otimes v\cdot\nabla_{x}G\,dv,
\\
\partial_{t}[\rho(\theta+\frac{1}{2}|u|^{2})]+\nabla_{x}\cdot[\rho u(\theta+\frac{1}{2}|u|^{2})+pu]+\rho u\cdot E
&=-\int_{\mathbb{R}^{3}} \frac{1}{2}|v|^{2} v\cdot\nabla_{x}G\,dv,
\end{array} \right.
\end{equation}
where the pressure $p=R\rho\theta=\frac{2}{3}\rho\theta$.

Substituting \eqref{2.11} into \eqref{2.13} and using the following two identities
\begin{align*}
-\int_{\mathbb{R}^{3}} v_{i} v\cdot\nabla_{x}
L^{-1}_{M}[P_{1}(v\cdot\nabla_{x}M)]\,dv
&\equiv\sum^{3}_{j=1}\partial_{x_{j}}[\mu(\theta)D_{ij}],\quad i=1,2,3,
\\
-\int_{\mathbb{R}^{3}} \frac{1}{2}|v|^{2} v\cdot\nabla_{x}
L^{-1}_{M}[P_{1}(v\cdot\nabla_{x}M)]\,dv
&\equiv\sum^{3}_{j=1}\partial_{x_{j}}(\kappa(\theta)\partial_{x_{j}}\theta)+
\sum^{3}_{i,j=1}\partial_{x_{j}}[\mu(\theta) u_{i}D_{ij}],
\end{align*}
with the viscous stress tensor $D=[D_{ij}]_{1\leq i,j\leq j}$ given by
\begin{equation}
\label{2.14}
D_{ij}=\partial_{x_{j}}u_{i}
+\partial_{x_{i}}u_{j}-\frac{2}{3}\delta_{ij}\nabla_{x}\cdot u,
\end{equation}
we further obtain the following fluid-type system
\begin{equation}
\label{2.15}
\left\{
\begin{array}{rl}
\partial_{t}\rho+\nabla_{x}\cdot(\rho u)&=0,
\\
\partial_{t}(\rho u_{i})+\nabla_{x}\cdot(\rho u_{i}u)+\partial_{x_{i}}p+\rho(E+u\times B)_{i}
&=\varepsilon\sum^{3}_{j=1}\partial_{x_{j}}(\mu(\theta)D_{ij})
\\
-\int_{\mathbb{R}^{3}}&v_{i}(v\cdot\nabla_{x}L^{-1}_{M}\Theta)\,dv, \quad i=1,2,3,
\\
\partial_{t}[\rho(\theta+\frac{1}{2}|u|^{2})]+\nabla_x\cdot[\rho u(\theta+\frac{1}{2}|u|^{2})+pu]+\rho u\cdot E
&=\varepsilon\sum^{3}_{j=1}\partial_{x_{j}}(\kappa(\theta)\partial_{x_{j}}\theta)
\\
+\varepsilon
\sum^{3}_{i,j=1}\partial_{x_{j}}(\mu(\theta) u_{i}D_{ij})&-\int_{\mathbb{R}^{3}} \frac{1}{2}|v|^{2} v\cdot\nabla_{x}L^{-1}_{M}\Theta\,dv,
\end{array} \right.
\end{equation}
where $(E+u\times B)_{i}$ represents the $i$-th component of $(E+u\times B)$.
Moreover, from \eqref{1.1} and \eqref{2.1}, the electromagnetic field $(E,B)(t,x)$ is governed by
\begin{equation}
\label{2.16}
\left\{
\begin{array}{rl}
\partial_{t}E-\nabla_{x}\times B=\rho u,
\quad
\partial_{t}B+\nabla_{x}\times E&=0,
\\
\nabla_{x}\cdot E=1-\rho,\quad \nabla_{x}\cdot B&=0.
\end{array} \right.
\end{equation}
Here the viscosity coefficient $\mu(\theta)>0$ and the heat conductivity coefficient $\kappa(\theta)>0$ in \eqref{2.15} are smooth functions depending only on the temperature $\theta$, represented by
\begin{align*}
\mu(\theta)=&- R\theta\int_{\mathbb{R}^{3}}\hat{B}_{ij}(\frac{v-u}{\sqrt{R\theta}})
B_{ij}(\frac{v-u}{\sqrt{R\theta}})\,dv>0,\quad i\neq j,
\nonumber\\
\kappa(\theta)=&-R^{2}\theta\int_{\mathbb{R}^{3}}\hat{A}_{j}(\frac{v-u}{\sqrt{R\theta}})
A_{j}(\frac{v-u}{\sqrt{R\theta}})\,dv>0,
\end{align*}
where $\hat{A}_{j}(\cdot)$ and $\hat{B}_{ij}(\cdot)$ are Burnett functions, defined by
\begin{equation}
\label{2.17}
\hat{A}_{j}(v)=\frac{|v|^{2}-5}{2}v_{j}\quad \mbox{and} \quad \hat{B}_{ij}(v)=v_{i}v_{j}-\frac{1}{3}\delta_{ij}|v|^{2}, \quad \mbox{for} \quad i,j=1,2,3.
\end{equation}
And $A_{j}(\cdot)$ and $B_{ij}(\cdot)$ satisfy  $P_{0}A_{j}(v)=0$ and $P_{0}B_{ij}(v)=0$, given by
\begin{equation}
\label{2.18}
A_{j}(v)=L^{-1}_{M}[\hat{A}_{j}(v)M]\quad
\mbox{and} \quad B_{ij}(v)=L^{-1}_{M}[\hat{B}_{ij}(v)M].
\end{equation}
For the convenience of readers, we list some elementary but important properties of the Burnett functions summarized in the following lemma, cf.
\cite{Ukai-Yang}.
\begin{lemma}\label{lem2.1}
The Burnett functions have the following properties:
\begin{itemize}
\item{$-\langle \hat{A}_{i}, A_{i}\rangle$ ~~is positive and independent of i;}
\item{$\langle \hat{A}_{i}, A_{j}\rangle=0$ ~~for ~any ~$i\neq j$;\quad $\langle
\hat{A}_{i}, B_{jk}\rangle=0$~~for ~any ~i,~j,~k;}
\item{$\langle\hat{B}_{ij},B_{kj}\rangle=\langle\hat{B}_{kl},B_{ij}\rangle=\langle\hat{B}_{ji},B_{kj}\rangle$,~~
which is independent of ~i,~j, for fixed~~k,~l;}
\item{$-\langle \hat{B}_{ij}, B_{ij}\rangle$ ~~is positive and independent of i,~j when $i\neq j$;}	\item{$\langle \hat{B}_{ii}, B_{jj}\rangle$ ~~is positive and independent of i,~j when $i\neq j$;}\item{$-\langle \hat{B}_{ii}, B_{ii}\rangle$ ~~is positive and independent of i;}
\item{$\langle \hat{B}_{ij}, B_{kl}\rangle=0$ ~~unless~either~$(i,j)=(k,l)$~or~$(l,k)$,~or~i=j~and~k=l;}	\item{$\langle \hat{B}_{ii}, B_{ii}\rangle-\langle \hat{B}_{ii}, B_{jj}\rangle=2\langle \hat{B}_{ij},B_{ij}\rangle$ ~~holds for any~ $i\neq j$.}
\end{itemize}
\end{lemma}
To make a conclusion, we have decomposed the VMB system \eqref{1.1} as the coupling of the
viscous compressible fluid-type system \eqref{2.15}-\eqref{2.16} and the microscopic equation \eqref{2.10}.
When the Knudsen number $\varepsilon$ and the microscopic part $G$ are set to
be zero, the system \eqref{2.15}-\eqref{2.16} becomes the full compressible Euler-Maxwell system.
These fluid systems can also be derived from the VMB system through
the Hilbert expansion. This means that the macro-micro
decomposition \eqref{2.7} in some sense can be viewed as a unification of the classical Hilbert
expansion up to the first order approximation. Therefore, this
decomposition gives a good framework for deriving rigorously the compressible Euler-Maxwell system
from the VMB system.

\section{Main results}\label{sec.3}
In this section, we make use of the macro-micro decomposition introduced in the previous
section to establish our main result that for well-prepared data the smooth solutions of the VMB system \eqref{1.1} exist and converge to smooth solutions of the 
Euler-Maxwell system \eqref{1.3} globally in time, as the Knudsen number $\varepsilon$ goes to zero.
\subsection{Smooth solutions for Euler-Maxwell system}\label{sec.3.1}
Now we construct the  global smooth solutions for the compressible Euler-Maxwell system \eqref{1.3}. For this,
we supplement it with initial data
\begin{equation}
\label{3.2}
(\bar{\rho},\bar{u},\bar{E},\bar{B})(0,x)=(\bar{\rho}_{0},\bar{u}_{0},\bar{E}_{0},\bar{B}_{0})(x).
\end{equation}
Throughout this paper, we assume that $(\bar{\rho}_{0},\bar{u}_{0},\bar{E}_{0},\bar{B}_{0})(x)$ 
satisfies the same assumptions as in \cite[Theorem 1.3]{Ionescu}. Then the Cauchy problem on the compressible
Euler-Maxwell system \eqref{1.3} and \eqref{3.2} admits a unique global-in-time  smooth solution  $(\bar{\rho},\bar{u},\bar{E},\bar{B})(t,x)$.
We note that the system \eqref{2.15}-\eqref{2.16} becomes the full compressible Euler-Maxwell system 
when $\varepsilon$ and $G$ are set to be zero. In order to recover  the isentropic compressible Euler-Maxwell system \eqref{1.3} 
to the full compressible Euler-Maxwell system, we take $\bar{\theta}(t,x)=\frac{3}{2}\bar{\rho}^{2/3}(t,x)$, then
$(\bar{\rho},\bar{u},\bar{\theta},\bar{E},\bar{B})(t,x)$ satisfies the following full compressible Euler-Maxwell system
\begin{align}
\label{3.1}
\left\{
\begin{array}{rl}
\partial_{t}\bar{\rho}+\nabla_{x}\cdot(\bar{\rho}\bar{u})&=0,
\\
\bar{\rho}\bar{u}_{t}+\bar{\rho}(\bar{u}\cdot\nabla_{x}\bar{u})+\nabla_{x}\bar{p}&=-\bar{\rho}(\bar{E}+\bar{u}\times \bar{B}),
\\
\bar{\theta}_{t}+(\bar{u}\cdot\nabla_{x})\bar{\theta}+\frac{2}{3}\bar{\theta}\nabla_{x}\cdot\bar{u}&=0,
\\
\partial_{t}\bar{E}-\nabla_{x}\times \bar{B}&=\bar{\rho}\bar{u},
\quad
\partial_{t}\bar{B}+\nabla_{x}\times \bar{E}=0,
\\
\nabla_{x}\cdot \bar{E}=1-\bar{\rho},\quad \nabla_{x}\cdot \bar{B}&=0,
\end{array} \right.
\end{align}
where the pressure $\bar{p}=\bar{\rho}^{5/3}=\frac{2}{3}\bar{\rho}\bar{\theta}$.
Moreover, we have the following properties of $(\bar{\rho},\bar{u},\bar{\theta},\bar{E},\bar{B})(t,x)$, which can be found in 
\cite[Theorem 1.3]{Ionescu}.
\begin{proposition}\label{prop.3.1}
Let  $(\bar{\rho},\bar{u},\bar{E},\bar{B})(t,x)$ be a global-in-time  smooth solution to
the compressible Euler-Maxwell system \eqref{1.3} and \eqref{3.2}, and let $\bar{\theta}(t,x)=\frac{3}{2}\bar{\rho}^{2/3}(t,x)$,
then the following estimate holds for all $t\geq 0$:
\begin{align}
\label{3.3}
&\|(\bar{\rho}-1,\bar{u},\bar{\theta}-\frac{3}{2},\bar{E},\bar{B})\|_{W^{N_0,2}}
\nonumber\\
&+(1+t)^{\vartheta}\big\{\|(\bar{\rho}-1,\bar{\theta}-\frac{3}{2},\bar{B})\|_{W^{N,\infty}}
+\|(\bar{u},\bar{E})\|_{W^{N+1,\infty}}\big\}\leq C\eta_{0}.
\end{align}
Here $\vartheta=101/100$, $\eta_{0}>0$ is a sufficiently small constant and
$N_0>0$ is a large integer, where integer $N$ satisfies $3\leq N< N_0$.
\end{proposition}
We recall that the main goal of this paper is to show that the smooth solution $F^{\varepsilon}(t,x,v)$ in \eqref{1.1} converges to a local Maxwellian
\begin{equation}
\label{3.4}
\overline{M}\equiv M_{[\bar{\rho},\bar{u},\bar{\theta}]}(t,x,v):
=\frac{\bar{\rho}(t,x)}{\sqrt{(2\pi R\overline{\theta}(t,x))^{3}}}\exp\big\{-\frac{|v-\overline{u}(t,x)|^2}{2R\overline{\theta}(t,x)}\big\},
\end{equation}
as $\varepsilon\to 0$, where $(\bar{\rho},\bar{u},\bar{\theta})(t,x)$ is given in Proposition \ref{prop.3.1}. For this, 
we shall supplement \eqref{1.1} with a specific initial valve
\begin{equation}
\label{3.6}
F^{\varepsilon}(0,x,v)\equiv M_{[\bar{\rho},\bar{u},\bar{\theta}]}(0,x,v)\geq 0,\quad (E^{\varepsilon},B^{\varepsilon})(0,x)\equiv(\bar{E},\bar{B})(0,x),
\end{equation}
where $(\bar{\rho},\bar{u},\bar{E},\bar{B})(0,x)$ is given in \eqref{3.2} and $\bar{\theta}(0,x)=\frac{3}{2}\bar{\rho}^{2/3}(0,x)$.
\subsection{Perturbation equations}\label{seca.5.1}
As in \cite{Duan-Yang-Yu}, associated with the constant state $(1,0,3/2)$,
we introduce  a normalized global Maxwellian 
\begin{equation}
\label{3.5}
\mu\equiv M_{[1,0,\frac{3}{2}]}(v):=(2\pi)^{-\frac{3}{2}}\exp\big(-\frac{|v|^{2}}{2}\big).
\end{equation}
Then we define the microscopic perturbation
\begin{equation}
\label{4.5}
\sqrt{\mu}f(t,x,v)=G(t,x,v)-\overline{G}(t,x,v),
\end{equation}
and the macroscopic perturbation
\begin{equation}
\label{4.1}
{(\widetilde{\rho},\widetilde{u},\widetilde{\theta},\widetilde{E},\widetilde{B})(t,x)
=(\rho-\bar{\rho},u-\bar{u},\theta-\bar{\theta},E-\overline{E},B-\bar{B})(t,x),}
\end{equation}
where  $\overline{G}(t,x,v)$ is given by
\begin{equation}
\label{4.6}
\overline{G}(t,x,v)\equiv\varepsilon L_{M}^{-1}P_{1}\big\{v\cdot(\frac{|v-u|^{2}
\nabla_{x}\bar{\theta}}{2R\theta^{2}}+\frac{(v-u)\cdot\nabla_{x}\bar{u}}{R\theta})M\big\}.
\end{equation}
Note that the similar correction function $\overline{G}(t,x,v)$ was first introduced by Liu-Yang-Yu-Zhao \cite{Liu-Yang-Yu-Zhao} for the stability of the rarefaction wave to the one-dimensional Boltzmann equation.

Let's now derive the equations of the microscopic perturbation and the macroscopic perturbation. Note form \eqref{4.5} that $L_{M}G=L_{M}\overline{G}+L_{M}(\sqrt{\mu}f)$, we can rewrite \eqref{2.10} as
\begin{multline}
\label{4.7}
\partial_{t}G+P_{1}(v\cdot\nabla_{x}G)-(E+v\times B)\cdot\nabla_{v}G=\frac{1}{\varepsilon}L_{M}(\sqrt{\mu}f)+\frac{1}{\varepsilon}Q(G,G)
\\
-P_{1}\big\{v\cdot(\frac{|v-u|^{2}\nabla_{x}\widetilde{\theta}}{2R\theta^{2}}
+\frac{(v-u)\cdot\nabla_{x}\widetilde{u}}{R\theta})M\big\},
\end{multline}
where we have used \eqref{4.6}, \eqref{2.5} and 
\begin{equation}
\label{4.8}
P_{1}(v\cdot\nabla_{x}M)=P_{1}\big\{v\cdot(\frac{|v-u|^{2}
\nabla_{x}\widetilde{\theta}}{2R\theta^{2}}+\frac{(v-u)\cdot\nabla_{x}\widetilde{u}}{R\theta})M\big\}
+\frac{1}{\varepsilon}L_{M}\overline{G}.
\end{equation}
Inspired by \cite{Guo-Indiana}, denoting
\begin{equation}
\label{4.9}
\Gamma(h,g):=\frac{1}{\sqrt{\mu}}Q(\sqrt{\mu}h,\sqrt{\mu}g),
\quad \mathcal{L}h:=\Gamma(h,\sqrt{\mu})+\Gamma(\sqrt{\mu},h),
\end{equation}
then we can derive from this and \eqref{2.9} that
\begin{equation}
\label{4.10}
\frac{1}{\sqrt{\mu}}L_{M}(\sqrt{\mu}f)=\mathcal{L}f+\Gamma(\frac{M-\mu}{\sqrt{\mu}},f)
+\Gamma(f,\frac{M-\mu}{\sqrt{\mu}}).
\end{equation}
Hence, using \eqref{4.7}, \eqref{4.9}, \eqref{4.10}, \eqref{4.5} and the fact $P_{1}=I-P_{0}$ in \eqref{2.5}, 
we obtain the system for $f(t,x,v)$ as 
\begin{align}
\label{4.11}
&\partial_{t}f+v\cdot\nabla_{x}f-\frac{(E+v\times B)\cdot\nabla_{v}(\sqrt{\mu}f+\overline{G})}{\sqrt{\mu}}
-\frac{1}{\varepsilon}\mathcal{L}f
\nonumber\\
&=\frac{1}{\varepsilon}\Gamma(\frac{M-\mu}{\sqrt{\mu}},f)
+\frac{1}{\varepsilon}\Gamma(f,\frac{M-\mu}{\sqrt{\mu}})
+\frac{1}{\varepsilon}\Gamma(\frac{G}{\sqrt{\mu}},\frac{G}{\sqrt{\mu}})+\frac{P_{0}[v\cdot\nabla_{x}(\sqrt{\mu}f)]}{\sqrt{\mu}}
\nonumber\\
&\hspace{0.5cm}-\frac{P_{1}(v\cdot\nabla_{x}\overline{G})}{\sqrt{\mu}}-\frac{\partial_{t}\overline{G}}{\sqrt{\mu}}	-\frac{1}{\sqrt{\mu}}P_{1}\big\{v\cdot(\frac{|v-u|^{2}\nabla_{x}\widetilde{\theta}}{2R\theta^{2}}+\frac{(v-u)\cdot\nabla_{x}\widetilde{u}}{R\theta})M\big\}.
\end{align}
Similarly, we can rewrite  the first equation of \eqref{1.1} as
\begin{multline}
\label{4.12}
\frac{\partial_{t}F}{\sqrt{\mu}}+\frac{v\cdot\nabla_{x}F}{\sqrt{\mu}}-\frac{(E+v\times B)\cdot\nabla_{v}F}{\sqrt{\mu}}
=\frac{1}{\varepsilon}\mathcal{L}f+\frac{1}{\varepsilon}\Gamma(\frac{M-\mu}{\sqrt{\mu}},f)
\\
+\frac{1}{\varepsilon}\Gamma(f,\frac{M-\mu}{\sqrt{\mu}})+\frac{1}{\varepsilon}\Gamma(\frac{G}{\sqrt{\mu}},\frac{G}{\sqrt{\mu}})+\frac{1}{\varepsilon}\frac{L_{M}\overline{G}}{\sqrt{\mu}}.
\end{multline}
We recall the properties of the linearized operator $\mathcal{L}$ defined as \eqref{4.9} in \cite{Guo-Indiana}.
Note that $\mathcal{L}$ is self-adjoint and non-positive definite, and its null space $\ker\mathcal{L}$ is spanned by the basis $\{\sqrt{\mu},v\sqrt{\mu},|v|^{2}\sqrt{\mu}\}$. 
Moreover, for any $g\in (\ker\mathcal{L})^{\perp}$, there exists a generic constant $c_1>0$ such that
\begin{equation}
\label{4.28}
-\langle\mathcal{L}g, g \rangle\geq c_1|g|^{2}_{\nu}=c_1\int_{\mathbb{R}^{3}}\nu(v)|g|^{2}\,dv,
\end{equation}
where $\nu(v)\approx(1+|v|)$ is the collision frequency for the hard sphere model. We also define the weighted dissipation norms as
\begin{equation}
\label{1.4}
\|g\|_{\nu}^{2}\equiv\int_{\mathbb{R}^{3}}\int_{\mathbb{R}^{3}}\nu(v)|g|^{2}\,dv\,dx.
\end{equation}
\begin{remark}
\label{rem3.2}
Since $G$ and $\overline{G}$ are purely microscopic, then the unknown $f(t,x,v)$ in \eqref{4.5}  is purely	microscopic, namely $f(t,x,v)\in (\ker\mathcal{L})^{\perp}$. This is essentially different from the one in \cite{Guo-Indiana}
since $f(t,x,v)$  used  in \cite{Guo-Indiana} involves the macroscopic part.
\end{remark}
By \eqref{4.1}, \eqref{2.15}, \eqref{2.16} and \eqref{3.1},
we obtain the system for $(\widetilde{\rho},\widetilde{u},\widetilde{\theta},\widetilde{E},\widetilde{B})$ as 
\begin{align}
\label{4.2}
\left\{
\begin{array}{rl}
&\partial_{t}\widetilde{\rho}
+u\cdot\nabla_{x}\widetilde{\rho}+\bar{\rho}\nabla_{x}\cdot\widetilde{u}+\widetilde{u}\cdot\nabla_{x}\bar{\rho}+\widetilde{\rho}\nabla_{x}\cdot u=0,
\\
&\partial_{t}\widetilde{u}_{i}+u\cdot\nabla_{x}\widetilde{u}_{i}+\widetilde{u}\cdot\nabla_{x}\bar{u}_{i}+\frac{2\bar{\theta}}{3\bar{\rho}}
\partial_{x_{i}}\widetilde{\rho}
+\frac{2}{3}\partial_{x_{i}}\widetilde{\theta}+\frac{2}{3}(\frac{\theta}{\rho}-\frac{\bar{\theta}}{\bar{\rho}})\partial_{x_{i}}\rho
\\
&\hspace{3cm}=-(\widetilde{E}+u\times \widetilde{B}+\widetilde{u}\times\bar{B})_{i}+
\varepsilon\frac{1}{\rho}\sum^{3}_{j=1}\partial_{x_{j}}(\mu(\theta)D_{ij})
\\
&\hspace{3.5cm}-\frac{1}{\rho}\int_{\mathbb{R}^{3}} v_{i}(v\cdot\nabla_{x}L^{-1}_{M}\Theta)\,dv, \quad i=1,2,3,
\\
&\partial_{t}\widetilde{\theta}+u\cdot\nabla_{x}\widetilde{\theta}+\frac{2}{3}\bar{\theta}\nabla_{x}\cdot \widetilde{u}
+\widetilde{u}\cdot\nabla_{x}\bar{\theta}+\frac{2}{3}\widetilde{\theta}\nabla_{x}\cdot u
\\
&\hspace{3cm}=\varepsilon\frac{1}{\rho}\sum^{3}_{j=1}\partial_{x_{j}}(\kappa(\theta)\partial_{x_{j}}\theta)+\varepsilon\frac{1}{\rho}\sum^{3}_{i,j=1}\mu(\theta)\partial_{x_{j}}u_{i}D_{ij}
\\
&\hspace{3.5cm}-\frac{1}{\rho}\int_{\mathbb{R}^{3}} \frac{1}{2}|v|^{2} v\cdot\nabla_{x}L^{-1}_{M}\Theta\,dv
+\frac{1}{\rho}u\cdot\int_{\mathbb{R}^{3}} v\otimes v\cdot\nabla_{x}L^{-1}_{M}\Theta\,dv,
\end{array} \right.
\end{align}
where the electromagnetic field $(\widetilde{E},\widetilde{B})$ satisfies
\begin{align}
\label{4.3}
\left\{
\begin{array}{rl}
&\partial_{t}\widetilde{E}-\nabla_{x}\times \widetilde{B}=\rho u-\bar{\rho}\bar{u},
\\
&\partial_{t}\widetilde{B}+\nabla_{x}\times \widetilde{E}=0,
\\
&\nabla_{x}\cdot \widetilde{E}=-\widetilde{\rho},\quad \nabla_{x}\cdot \widetilde{B}=0.
\end{array} \right.
\end{align}
In the late energy analysis, we will focus on these perturbation systems.

In order to prove the almost global-in-time solutions for the VMB system \eqref{1.1} and \eqref{3.6}, a key point is
to establish uniform energy estimates on the macroscopic part $(\widetilde{\rho},\widetilde{u},\widetilde{\theta},\widetilde{E},\widetilde{B})$
and the microscopic part $f$. For this, we define the  instant energy as
\begin{align}
\label{4.14}
\mathcal{E}_{N}(t)\equiv&\sum_{|\alpha|\leq N-1}\{\|\partial^{\alpha}(\widetilde{\rho},\widetilde{u},\widetilde{\theta},\widetilde{E},\widetilde{B})(t)\|^{2}+\|\partial^{\alpha}f(t)\|^{2}\}+\sum_{|\alpha|+|\beta|\leq N,|\beta|\geq1}\|\partial^{\alpha}_{\beta}f(t)\|^{2}
\nonumber\\
&+\varepsilon^{2}\sum_{|\alpha|=N}\{\|\partial^{\alpha}(\widetilde{\rho},\widetilde{u},\widetilde{\theta},\widetilde{E},\widetilde{B})(t)\|^{2}+\|\partial^{\alpha}f(t)\|^{2}\},
\end{align}
and the dissipation rate as
\begin{align}
\label{4.15}
\mathcal{D}_{N}(t)&\equiv\varepsilon\sum_{1\leq|\alpha|\leq N}
\|\partial^{\alpha}(\widetilde{\rho},\widetilde{u},\widetilde{\theta})(t)\|^{2}
+\varepsilon\sum_{|\alpha|=N}\|\partial^{\alpha}f(t)\|_{\nu}^{2}
\nonumber\\
&\quad+\frac{1}{\varepsilon}\sum_{|\alpha|\leq N-1}\|\partial^{\alpha}f(t)\|_{\nu}^{2}
+\frac{1}{\varepsilon}\sum_{|\alpha|+|\beta|\leq N,|\beta|\geq1}\|\partial^{\alpha}_{\beta}f(t)\|_{\nu}^{2}.
\end{align}
Throughout this article, we assume $N\geq 3$ corresponding to the one in \eqref{3.3}.  
%It is important to note that the highest $N$th-order space derivatives are much more singular with respect to $\varepsilon$ than those derivatives of order up to $N-1$
%in the instant energy $\mathcal{E}_N(t)$, and it occurs similarly to  the dissipation rate $\mathcal{D}_N(t)$ for the non-fluid component.

\subsection{Main theorem}\label{sec.3.2}
Our main result is given as follows.

\begin{theorem}\label{thm3.1}
Let $(\bar{\rho},\bar{u},\bar{\theta},\bar{E},\bar{B})(t,x)$  
be a global smooth solution to the compressible Euler-Maxwell system given in Proposition \ref{prop.3.1}.
Construct a  local Maxwellian $M_{[\bar{\rho},\bar{u},\bar{\theta}]}(t,x,v)$  as in \eqref{3.4}. Then there exists a small constant $\varepsilon_{0}>0$
such that for each $\varepsilon\in(0,\varepsilon_{0}]$,  the Cauchy problem on the Vlasov-Maxwell-Boltzmann system \eqref{1.1} and \eqref{3.6} admits a unique smooth solution $(F^{\varepsilon}(t,x,v),E^{\varepsilon}(t,x),B^{\varepsilon}(t,x))$ for all $t\in[0,T_{max}]$ with
\begin{equation}
\label{3.7a}
T_{max}=\frac{1}{4C_1}\frac{1}{\eta_{0}\varepsilon^a+\varepsilon^{\frac{1}{2}-a}}, \quad \mbox{for} \quad a\in[0,\frac{1}{2}),
\end{equation}
where generic constant $C_1>1$ and small constant $\eta_{0}>0$  are independent of $\varepsilon$.
Moreover, it holds that $F^{\varepsilon}(t,x,v)\geq 0$ and
\begin{equation}
\label{3.8}
\mathcal{E}_{N}(t)+\frac{1}{2}\int^{t}_{0}\mathcal{D}_{N}(s)\,ds\leq \frac{1}{2}\varepsilon^{2-2a},
\end{equation}
for any $t\in[0,T_{max}]$.
In particular, there exists a constant $C>0$ independent of $\varepsilon$ and $T_{max}$ such that
\begin{align}
\label{3.7}
&\sup_{t\in[0,T_{max}]}\{\|\frac{F^{\varepsilon}(t,x,v)-M_{[\bar{\rho},\bar{u},\bar{\theta}]}(t,x,v)}{\sqrt{\mu}}\|_{L_{x}^{2}L_{v}^{2}}
+\|\frac{F^{\varepsilon}(t,x,v)-M_{[\bar{\rho},\bar{u},\bar{\theta}]}(t,x,v)}{\sqrt{\mu}}
\|_{L_{x}^{\infty}L_{v}^{2}}\}
\nonumber\\
&+\sup_{t\in[0,T_{max}]}\{\|(E^{\varepsilon}-\bar{E},B^{\varepsilon}-\bar{B})(t,x)\|_{L_{x}^{2}}
+\|(E^{\varepsilon}-\bar{E},B^{\varepsilon}-\bar{B})(t,x)\|_{L_{x}^{\infty}}\}\leq C\varepsilon^{1-a}.
\end{align}
\end{theorem}
\begin{remark}
Theorem \ref{thm3.1} shows that both the validity time $T_{max}$ and the uniform convergence rate in small Knudsen number $\varepsilon$ can be variable with
respect to the  parameter $a\in[0,\frac{1}{2})$.
%If we take $a=0$, the validity time $T_{max}$ in \eqref{3.7a} is finite, and the uniform convergence rate is $\varepsilon$ corresponding to \eqref{3.7}. 
If we take $a\in(0,\frac{1}{2})$,
$T_{max}\to +\infty$ as $\varepsilon\to 0$. This implies  the global-in-time convergence from the VMB system to the Euler-Maxwell system.
\end{remark}

\begin{remark}	
Although only the hard sphere model is considered in this paper, we believe that our method can be applied to the Boltzmann operator for general soft potentials with or without angular cutoff or even to the Landau operator with Coulomb potentials, cf.~\cite{Duan-Yang-Yu}.
\end{remark}

\subsection{Key points of the proof}\label{sec.3.3}
We are now in a position to make some comments on the analysis of this paper. 
We shall point out main difficulties to be overcome as well as main differences of the proof compared to the previous results
\cite{Guo-Xiao}. Recall that by using the $L^{2}-L^{\infty}$ framework developed in \cite{Guo-Jang-Jiang,Guo-Jang-Jiang-2010,Guo-Jang}, together with 
the Glassey-Strauss representation of the electromagnetic field introduced in \cite{Glassey,Glassey-1996}, 
Guo-Xiao \cite{Guo-Xiao} proved the hydrodynamic limit from the relativistic VMB system to the relativistic compressible Euler-Maxwell system.
However, the approach in \cite{Guo-Xiao} is again not available for the non-relativistic VMB system 
because their analysis depends crucially on the Glassey-Strauss representation of the electromagnetic field, which fails for the
non-relativistic case. Indeed, in view of the Glassey-Strauss representation, the electric field $E$
in the relativistic case can be expressed as follows:
\begin{equation}
\label{3.16}
4\pi E(t,x)=-\int_{|y-x|\leq t}\int_{{\mathbb R}^3}\frac{(\omega+\hat{v})(1-|\hat{v}|^2)}{(1+\hat{v}\cdot \omega)^2}
F(t-|y-x|,y,v)\,dv\frac{dy}{|y-x|^2}+\mbox{other terms},
\end{equation}
with $\hat{v}=\frac{v}{\sqrt{1+|v|^2}}$ and $\omega=\frac{y-x}{|y-x|}$. A similar representation also holds for the magnetic field $B$.
Interested readers may refer to \cite{Glassey} and \cite[Chapter 5]{Glassey-1996} for more detailed representation.
The relativistic velocity $\hat{v}$ is bounded, so the expression $1+\hat{v}\cdot \omega$ is bounded away from $0$.
This fact is crucial in the analysis in \cite{Guo-Xiao}. However, the corresponding expression in the non-relativistic case is $1+v\cdot \omega$.
The expression $1+v\cdot \omega$ could be $0$ due to unbounded velocity $v$. This creates singularities in \eqref{3.16}, such as
$(1+v\cdot \omega)^{-2}$. Hence, the strategy in \cite{Guo-Xiao} cannot be applied directly to the non-relativistic case, and ones 
have to develop  new ideas.

Our strategy is based on nonlinear energy method in the high order Sobolev space instead of the $L^{2}-L^{\infty}$ method.
This  method is based on the macro-micro decomposition in \cite{Liu-Yu,Liu-Yang-Yu} 
with respect to the local Maxwellian $M$  determined by the solution of kinetic equations,
namely
\begin{equation}
\label{3.26B}
F=M+G.
\end{equation}
Under this decomposition, we can rewrite the VMB system as a compressible Navier-Stokes-Maxwell-type system with the non-fluid component
appearing in the conservative source terms, cf. \eqref{2.15}, coupled with an equation for the
non-fluid component, cf. \eqref{2.10}. In this way, one advantage is that the viscosity and heat conductivity coefficients can be expressed explicitly so that the energy analysis in the context of the viscous compressible fluid can be applied to capture the dissipation of the fluid part. The other advantage is that the  nonlinear term $Q(G,G)$ in \eqref{2.10} depends only on the microscopic part $G$. 

Similar for showing the compressible fluid limit for smooth solutions to the Landau equation in \cite{Duan-Yang-Yu},
we need to  consider the subtraction
of $G$ by $\overline{G}$ as \eqref{4.6} to remove the linear terms, because those terms induce
the energy term  $\varepsilon\|(\nabla_{x}\bar{\theta},\nabla_{x}\bar{u})\|^2$ in the $L^2$ estimate
that is  out of control by $O(\varepsilon^{2-2a})$ corresponding to \eqref{4.13}.
Hence we set $\sqrt{\mu}f=G-\overline{G}$ and then improve the decomposition \eqref{3.26B} to the decomposition
\begin{equation}
\label{3.26a}
F=M+\overline{G}+\sqrt{\mu}f.
\end{equation}
One key observation is that the term $\overline{G}$ can be represented precisely by using the
Burnett functions so that its estimates can be made by using the properties of
Burnett functions, cf. Lemma \ref{lem5.3}. Another key observation is that 
$f\in (\ker\mathcal{L})^{\perp}$ in \eqref{4.5} is purely
microscopic such that the estimate
\begin{equation*}
-\frac{1}{\varepsilon}(\mathcal{L}\partial^{\alpha}f,\partial^{\alpha}f)\geq c_1\frac{1}{\varepsilon}\|\partial^{\alpha}f\|^{2}_{\nu}
\end{equation*}
holds ture, which is again not true for the one in \cite{Guo-Indiana}. 
This is convenient for obtaining the  $\varepsilon$-singularity trilinear estimate  because the macroscopic part would not longer appear.

%Indeed, \cite{Guo-Indiana} used the decomposition $F=\mu+\sqrt{\mu}f$ with the perturbation
%$f$ involving the macroscopic part, namely $f\notin(\ker\mathcal{L})^{\perp}$,  and as such 
%one has to decompose $f=\mathbf{P}f+\{\mathbf{I-P}\}f$ with $\mathbf{P}$ the projection on the kernel space of $\mathcal{L}$, so that one can only obtain
%\begin{equation*}
%-\frac{1}{\varepsilon}(\mathcal{L}f,f)\geq c_1\frac{1}{\varepsilon}\|\{\mathbf{I-P}\}f\|^{2}_{\nu}.
%\end{equation*}
%This results in the appearance of a difficult term $\frac{1}{\varepsilon}\|\mathbf{P}f\|_\nu^2$ that
%would not be able to be controlled  since it involves a strong singularity about $\frac{1}{\varepsilon}$.
%The main reason is that the $\varepsilon$-dependent coefficient of the hydrodynamic part $\mathbf{P}f$ in the energy dissipation functional is just $\varepsilon$ instead of $1/\varepsilon$. Moreover, under the decomposition \eqref{3.26a},
%the trilinear estimate $\frac{1}{\varepsilon}(\Gamma(f,f),f)$ can be obtained easily because the difficult term
%$\frac{1}{\varepsilon}(\Gamma(\mathbf{P}f,\mathbf{P}f),\{\mathbf{I-P}\}f)$ would not longer appear. 
%Note that the time decay rate $(1+t)^{-\vartheta}$ with $\vartheta>1$ of the solution in the $L^{\infty}$ norm given in Proposition \ref{prop.3.1} plays a crucial role in our analysis, and it ensure that the estimates \eqref{5.3a} and \eqref{6.36} 
%can be controlled by $O(\varepsilon^{2-2a})$.

For the integral terms involving the inverse of the linearized operator $L_M$ around the local Maxwellian,
we make use of the Burnett functions and velocity-decay properties to deal with them,
so that the estimates can be obtained in a clear way, see \eqref{6.25}-\eqref{5.27A} for details.
For the $N$-order space derivative estimate on the term
$\frac{1}{\varepsilon}(\Gamma(\frac{M-\mu}{\sqrt{\mu}},\partial^{\alpha}f),\frac{I_{2}}{\sqrt{\mu}})$ in \eqref{5.29},
we must move one derivative from the  $N$-order derivative $\partial^{\alpha}f$ to the other component of the inner product by integration by part, see \eqref{5.32}.
This technique is also used for the electromagnetic field terms, see \eqref{5.60}. To overcome the difficulties 
caused by the electromagnetic field, we employ some useful techniques developed by Guo \cite{Guo-2003} for a global
Maxwellian. Because of the singular factor $\varepsilon^{-1}$ in front of the fluid part $\|\partial^{\alpha}(\widetilde{\rho},\widetilde{u},\widetilde{\theta})\|^{2}$ in \eqref{4.44}, one has to multiply the estimate \eqref{4.44} by $\varepsilon^2$ so that the fluid term can be controlled by $\mathcal{D}_N(t)$ in \eqref{4.15}. 
Hence, this motives how we include the Knudsen number $\varepsilon$ to the highest-order derivative $|\alpha|=N$ for the instant energy  $\mathcal{E}_N(t)$ in \eqref{4.14}.

Another important ingredient of the proof is that the an $\varepsilon$-dependent priori assumption \eqref{4.13}
involving a parameter $a\in[0,\frac{1}{2})$. 
%This parameter is crucial for obtaining the almost global-in-time solution.
%Indeed, we have established the result in Theorem \ref{thm3.1} showing that both the validity time and the uniform convergence rate in small Knudsen number $\varepsilon$ can be variable with respect to the  parameter $a\in[0,\frac{1}{2})$. 
%If we take $a=0$, the validity time $T_{max}$ in \eqref{3.7a} is finite, and the
%uniform convergence is $\varepsilon$ corresponding to \eqref{3.7}. In particular, if we take $a\in(0,\frac{1}{2})$,
%$T_{max}\to +\infty$ as $\varepsilon\to 0$. This implies global-in-time convergence from the VMB system \eqref{1.1} to the Euler-Maxwell system \eqref{1.3}.
Under this a priori assumption, we have to handle energy estimates carefully. 
The desired goal is to obtain the uniform a priori estimate \eqref{3.8} and then derive the convergence rate \eqref{3.7}.

\section{Preliminaries}\label{seca.4}
In  this section, we give some basic estimates which will be used frequently in the later energy
analysis.
%we first make the a priori assumption in order to perform energy
%analysis conveniently. Then we compute the complicated collision terms and the nonlinear terms involving
%with the electromagnetic fields. Finally we  derive some estimates on the fluid quantities and the characterization of low bound of the $N$-order energy norm on solution.
We should emphasize that in all estimates below, all constants $C>0$ 
are independent of $t$ and sufficiently small parameters both $\varepsilon$ and $\eta_0$.

\subsection{A priori assumption}\label{seca.4.2}
To prove the global existence of the solution for the VMB system in Theorem \ref{thm3.1}, a key point is to establish the a priori energy estimates
on the macroscopic part $(\widetilde{\rho},\widetilde{u},\widetilde{\theta},\widetilde{E},\widetilde{B})$
and the microscopic part $f$ based on the following {\it a priori} assumption:
\begin{equation}
\label{4.13}
\sup_{0\leq t\leq T}\mathcal{E}_{N}(t)\leq  \varepsilon^{2-2a}, \quad a\in[0,\frac{1}{2}),
\end{equation}
for arbitrary time $T\in(0,T_{max}]$ with $T_{max}$ given by \eqref{3.7a}. Here $\mathcal{E}_{N}(t)$ is given by \eqref{4.14}.
\begin{remark}
The range of the parameter $a$ in \eqref{4.13} will be determined by the estimate such as \eqref{6.3a}
in the later proof. This parameter is crucial for obtaining the almost global-in-time solution.
\end{remark}
\begin{remark}	
It should be pointed out that the energy inequality \eqref{4.13} holds true at $t=0$ under the choice of 
the initial data \eqref{3.6}. Indeed, we can claim that the initial data \eqref{3.6} automatically satisfies
\begin{equation}
\label{4.17}
\mathcal{E}_{N}(t)\mid_{t=0}\leq C\eta^{2}_{0}\varepsilon^{2}.
\end{equation}	
In fact, by the decomposition $F(t,x,v)=M_{[\rho,u,\theta]}(t,x,v)+G(t,x,v)$ as well as \eqref{3.6},
we have $M_{[\rho,u,\theta]}(0,x,v)=M_{[\bar{\rho},\bar{u},\bar{\theta}]}(0,x,v)$ and $G(0,x,v)=0$.
This implies that 
$$
(\rho,u,\theta)(0,x)=(\bar{\rho},\bar{u},\bar{\theta})(0,x),\quad \overline{G}(0,x,v)+\sqrt{\mu}f(0,x,v)=0.
$$
Then it holds that 
$$[\widetilde{\rho},\widetilde{u},\widetilde{\theta}](0,x)
=[{\rho}-\overline{\rho},u-\overline{u},\theta-\overline{\theta}](0,x)=0,\quad f(0,x,v)=-\frac{\overline{G}(0,x,v)}{\sqrt{\mu}}.$$
Moreover, it follows from \eqref{3.6} that
$$[\widetilde{E},\widetilde{B}](0,x)=[E-\bar{E},B-\bar{B}](0,x)=0.$$
By the above observations together with \eqref{4.14} and Lemma \ref{lem5.3}, one can see that \eqref{4.17} holds true. Therefore, letting $\eta_{0}>0$ be small enough, \eqref{4.13} is satisfied for $t=0$.
\end{remark}

\begin{remark}
In view of \eqref{4.13} and \eqref{3.3}, it is straightforward to verify that
\begin{equation}
\label{4.16}
|\rho(t,x)-1|+|u(t,x)|+|\theta(t,x)-\frac{3}{2}|\leq C(\eta_{0}+\varepsilon^{1-a}),\quad \frac{1}{2}<\rho(t,x)<\frac{3}{2},\quad
1<\theta(t,x)<2,
\end{equation}
holds for all $(t,x)\in[0,T]\times\mathbb{R}^{3}$. 
%Moreover, under the condition of \eqref{4.16}, there exist constants $C>0$ and $1<\sigma_1<\frac{3}{2}$ such that $M^{\sigma_1}\leq C\mu$ for all $(t,x,v)$.
\end{remark}

In all lemmas below, all derived estimates are satisfied for any $0\leq t\leq T$.
\subsection{Estimates on correction term $\overline{G}$}\label{seca.5.2}
We shall derive some estimates on $\overline{G}$ in \eqref{4.6} frequently used throughout this paper.
We first list several basic inequalities.
\begin{lemma}\label{lem5.1}
For any function $h=h(x)\in H^{1}(\mathbb{R}^{3})$, it holds that
\begin{align*}
\|h\|_{L^{6}(\mathbb{R}^{3})}\leq C\|\nabla_{x} h\|,
\quad
\|h\|_{L^{3}(\mathbb{R}^{3})}\leq C\| h\|^{\frac{1}{2}}\|\nabla_{x} h\|^{\frac{1}{2}}.
\end{align*}
For any function $h=h(x)\in H^{2}(\mathbb{R}^{3})$, it holds that
\begin{equation*}
\|h\|_{L^{\infty}(\mathbb{R}^{3})}
\leq C\|\nabla_{x} h\|^{\frac{1}{2}}\|\nabla_{x}^{2} h\|^{\frac{1}{2}}.
\end{equation*}
\end{lemma}
The following lemma is similar to the one in \cite[Lemma 6.1]{Duan-Yu1}.
\begin{lemma}\label{lem5.2}
Let $L_{\widehat{M}}$ defined in \eqref{2.9} for any Maxwellian $\widehat{M}=M_{[\widehat{\rho},\widehat{u},\widehat{\theta}]}(v)$.
Suppose that $U(v)$ is any polynomial of $\frac{v-\hat{u}}{\sqrt{R}\hat{\theta}}$ such that
$U(v)\widehat{M}\in(\ker{L_{\widehat{M}}})^{\perp}$. For any $\epsilon\in(0,1)$ and any multi-index $\beta$, there exists constant $C_{\beta}>0$ such that
$$
|\partial_{\beta}L^{-1}_{\widehat{M}}(U(v)\widehat{M})|\leq C_{\beta}(\widehat{\rho},\widehat{u},\widehat{\theta})\widehat{M}^{1-\epsilon}.
$$
Moreover, if \eqref{4.16} holds, there exists constant $C_{\beta}>0$ such that
\begin{equation}
\label{4.18}
|\partial_{\beta}A_{j}(\frac{v-u}{\sqrt{R\theta}})|+|\partial_{\beta}B_{ij}(\frac{v-u}{\sqrt{R\theta}})|
\leq C_{\beta}M^{1-\epsilon},
\end{equation}
where $A_{j}(\cdot)$ and $B_{ij}(\cdot)$ are defined in \eqref{2.18}.
\end{lemma}
%\begin{remark}
%Note that the upper bound in \eqref{4.18} is not optimal, but it enough for the use of the later proof. Interested readers may refer to \cite[Proposition 3.1]{Caflisch}.
%\end{remark}	
Based on Lemma \ref{lem5.2} above, we can deduce that
\begin{lemma}\label{lem5.3}
Let $\overline{G}$ given in \eqref{4.6} and $\langle v\rangle=\sqrt{1+|v|^2}$.
Assume that the a priori assumption  \eqref{4.13} and \eqref{3.3} hold. For any $b\geq 0$ and $|\beta|\geq 0$, if
$|\alpha|\leq N-1$, we have
\begin{equation}
\label{4.19}
\|\langle v\rangle^{b}\partial^{\alpha}_{\beta}(\frac{\overline{G}}{\sqrt{\mu}})\|
\leq C\eta_{0}\varepsilon.
\end{equation}
If $|\alpha|=N$, we have
\begin{equation}
\label{4.20}
\|\langle v\rangle^{b}\partial^{\alpha}_{\beta}(\frac{\overline{G}}{\sqrt{\mu}})\|\leq C\eta_{0}\varepsilon^{1-a}.
\end{equation}
\end{lemma}
\begin{proof}
Thanks to \eqref{2.17} and \eqref{2.18}, we can rewrite $\overline{G}$ in \eqref{4.6} as 
\begin{equation}
\label{4.21}
\overline{G}=\varepsilon\frac{\sqrt{R}}{\sqrt{\theta}}\sum^{3}_{j=1}
\frac{\partial\bar{\theta}}{\partial x_{j}}A_{j}(\frac{v-u}{\sqrt{R\theta}})
+\varepsilon\sum^{3}_{j=1}\sum^{3}_{i=1}\frac{\partial\bar{u}_{j}}{\partial x_{i}}B_{ij}(\frac{v-u}{\sqrt{R\theta}}).
\end{equation}	
Then, for $k=1,2,3$, a direct calculation shows that
\begin{equation}
\label{4.22}
\frac{\partial\overline{G}}{\partial v_{k}}=\varepsilon\frac{\sqrt{R}}{\sqrt{\theta}}\sum^{3}_{j=1}
\frac{\partial\bar{\theta}}{\partial x_{j}}\partial_{v_{k}}A_{j}(\frac{v-u}{\sqrt{R\theta}})\frac{1}{\sqrt{R\theta}} +\varepsilon\sum^{3}_{i,j=1}\frac{\partial\bar{u}_{j}}{\partial x_{i}}
\partial_{v_{k}}B_{ij}(\frac{v-u}{\sqrt{R\theta}})\frac{1}{\sqrt{R\theta}},
\end{equation}
and
\begin{align}
\label{4.23}
\frac{\partial\overline{G}}{\partial x_{k}}&=\varepsilon\Big\{
\frac{\sqrt{R}}{\sqrt{\theta}}\sum^{3}_{j=1}\frac{\partial^{2}\bar{\theta}}{\partial x_{j}\partial x_{k}}A_{j}(\frac{v-u}{\sqrt{R\theta}})
-\frac{\sqrt{R}}{2\sqrt{\theta^{3}}}\sum^{3}_{j=1}\frac{\partial\bar{\theta}}{\partial x_{j}}\frac{\partial\theta}{\partial x_{k}}A_{j}(\frac{v-u}{\sqrt{R\theta}})
\nonumber\\
&\quad-\frac{\sqrt{R}}{\sqrt{\theta}}\sum^{3}_{j=1}\frac{\partial\bar{\theta}}{\partial x_{j}}\frac{\partial u}{\partial x_{k}}\cdot
\nabla_{v}A_{j}(\frac{v-u}{\sqrt{R\theta}})\frac{1}{\sqrt{R\theta}}
-\frac{\sqrt{R}}{\sqrt{\theta}}\sum^{3}_{j=1}
\frac{\partial\bar{\theta}}{\partial x_{j}}\frac{\partial\theta}{\partial x_{k}}\nabla_{v}A_{j}(\frac{v-u}{\sqrt{R\theta}})
\cdot\frac{v-u}{2\sqrt{R\theta^{3}}}
\nonumber\\
&\quad+\sum^{3}_{i,j=1}\frac{\partial^{2}\bar{u}_{j}}{\partial x_{i}\partial x_{k}}B_{ij}(\frac{v-u}{\sqrt{R\theta}})
-\sum^{3}_{i,j=1}\frac{\partial\bar{u}_{j}}{\partial x_{i}}\frac{\partial u}{\partial x_{k}}\cdot\nabla_{v}B_{ij}(\frac{v-u}{\sqrt{R\theta}})\frac{1}{\sqrt{R\theta}}
\nonumber\\
&\quad-\sum^{3}_{i,j=1}\frac{\partial\bar{u}_{j}}{\partial x_{i}}\frac{\partial\theta}{\partial x_{k}}\nabla_{v}B_{ij}(\frac{v-u}{\sqrt{R\theta}})\cdot\frac{v-u}{2\sqrt{R\theta^{3}}}\Big\}.
\end{align}
For any $|\beta|\geq 0$ and  $b\geq 0$, using the similar expression as \eqref{4.22}, \eqref{4.18} and \eqref{3.3}, one gets
\begin{equation}
\label{4.26}
\|\langle v\rangle^{b}\partial_{\beta}(\frac{\overline{G}}{\sqrt{\mu}})\|^2
\leq C\varepsilon^2\|(\nabla_{x}\bar{u},\nabla_{x}\bar{\theta})\|^2\leq C\eta^2_{0}\varepsilon^2,
\end{equation}
where we used \eqref{4.16} and sufficiently  small $\epsilon>0$ to claim that
\begin{equation}
\label{4.25}
|\langle v\rangle^{b}\mu^{-\frac{1}{2}}M^{1-\epsilon}|_{2}\leq C.
\end{equation}
Similarly, if $1\leq|\alpha|\leq N$, it is straightforward to get that
\begin{align}
\label{4.12A}
\|\langle v\rangle^{b}\partial^{\alpha}_{\beta}(\frac{\overline{G}}{\sqrt{\mu}})\|^2
\leq& C\varepsilon^2\int_{\mathbb{R}^3}\big\{|\partial^{\alpha}(\nabla_{x}\bar{u},\nabla_{x}\bar{\theta})|+|(\nabla_{x}\bar{u},\nabla_{x}\bar{\theta})|
|\nabla_{x}(u,\theta)|^{|\alpha|}
\nonumber\\
&\hspace{1.5cm}+\cdot\cdot\cdot+|(\nabla_{x}\bar{u},\nabla_{x}\bar{\theta})|
|\partial^{\alpha}(u,\theta)|\big\}^2\,dx,
\end{align}
where $\cdot\cdot\cdot$ denoting the low order nonlinear terms.
To bound \eqref{4.12A}, we only estimate the second term and last term since the other terms can be treated similarly.
By virtue of \eqref{4.13} and \eqref{4.14}, one gets
\begin{equation}
\label{4.13A}
\sum_{|\alpha|=N}\{\|\partial^{\alpha}(\widetilde{\rho},\widetilde{u},\widetilde{\theta},\widetilde{E},\widetilde{B})\|^{2}+\|\partial^{\alpha}f\|^{2}\}
\leq \varepsilon^{-2a}.
\end{equation}
Using the embedding inequality, \eqref{3.3} and \eqref{4.13}, we get
\begin{align}
\label{4.27}
\|\nabla_{x}(u,\theta)\|_{L^{\infty}}&\leq \|\nabla_{x}(\bar{u},\bar{\theta})\|_{L^{\infty}}
+\|\nabla_{x}(\widetilde{u},\widetilde{\theta})\|_{L^{\infty}}
\nonumber\\
&\leq C\eta_{0}(1+t)^{-\vartheta}+C\|\nabla^2_{x}(\widetilde{u},\widetilde{\theta})\|^{\frac{1}{2}}
\|\nabla^3_{x}(\widetilde{u},\widetilde{\theta})\|^{\frac{1}{2}}
\nonumber\\
&\leq C\eta_{0}+C\varepsilon^{\frac{1}{2}-\frac{1}{2}a}\varepsilon^{-\frac{1}{2}a}\leq C(\eta_{0}+\varepsilon^{\frac{1}{2}-a}),
\end{align}
where in the third inequality we have used \eqref{4.13A} when $N=3$.
Hence, the second term in \eqref{4.12A} can be estimated as
\begin{align*}
&\varepsilon^2\int_{\mathbb{R}^3}|(\nabla_{x}\bar{u},\nabla_{x}\bar{\theta})|^2
|\nabla_{x}(u,\theta)|^{2|\alpha|}\,dx
\\
&\leq C\varepsilon^2\|\nabla_{x}(u,\theta)\|^{2|\alpha|}_{L^{\infty}}\|(\nabla_{x}\bar{u},\nabla_{x}\bar{\theta})\|^2
\leq C\eta^2_{0}\varepsilon^2.
\end{align*}
Here we used the smallness of both $\eta_{0}$ and $\varepsilon$ and chosen $a\in[0,1/2)$ such that
\begin{equation}
\label{4.22a}
\eta_{0}+\varepsilon^{\frac{1}{2}-a}<1.
\end{equation}	 
The requirement \eqref{4.22a} will be frequently used later on. For the last term in \eqref{4.12A}, one has
\begin{equation*}
\varepsilon^2\int_{\mathbb{R}^3}|(\nabla_{x}\bar{u},\nabla_{x}\bar{\theta})|^2|\partial^{\alpha}(u,\theta)|^2\,dx
\leq  C\eta^2_{0}(\eta^2_{0}+\|\partial^{\alpha}(\widetilde{u},\widetilde{\theta})\|^2)\varepsilon^2.
\end{equation*}
Therefore, for $1\leq|\alpha|\leq N$, we obtain
\begin{equation}
\label{4.23a}
\|\langle v\rangle^{b}\partial^{\alpha}_{\beta}(\frac{\overline{G}}{\sqrt{\mu}})\|^2
\leq C\eta^2_{0}(1+\|\partial^{\alpha}(\widetilde{u},\widetilde{\theta})\|^2)\varepsilon^2.
\end{equation}
If $|\alpha|=N$ in \eqref{4.23a}, we immediately get \eqref{4.20}  by using \eqref{4.13A}. Likewise, the desired estimate \eqref{4.19}
follows from \eqref{4.26}, \eqref{4.23a} and \eqref{4.13}. We thus finish the proof of Lemma \ref{lem5.3}.
\end{proof}
\subsection{Estimates on collision terms}\label{seca.5.3}
To perform the energy estimates for the equations \eqref{4.11} and \eqref{4.12}, we need to treat those collision terms.
\subsubsection{Properties of $\mathcal{L}$ and $\Gamma$}\label{seca.5.3.1}
We first list the following results on the
linearized Boltzmann  operator $\mathcal{L}$ and the nonlinear collision terms $\Gamma(g_1,g_2)$
defined in \eqref{4.9}. Interested readers may refer to \cite{Guo-Indiana,Guo-2003}
for more details.
\begin{lemma}\label{lem5.4}
For $|\beta|>0$, there exist the constants $c_2> 0$ and $C>0$ such that
\begin{equation}
\label{4.29}
-\langle\partial_\beta\mathcal{L}g,\partial_\beta g\rangle\geq c_2|\partial_\beta g|_\nu^2
-C|g|_\nu^2.
\end{equation}
\end{lemma}
\begin{lemma}
\label{lem5.5}
For $|\beta|\geq0$, there exists the constant $C>0$ such that
\begin{equation}
\label{4.30}
|\langle\partial_\beta \Gamma(g_1,g_2), g_3\rangle|\leq
C\sum_{\beta_1+\beta_2\leq\beta}\{|\partial_{\beta_1}g_1|_2| \partial_{\beta_2}g_2|_{\nu}
+|\partial_{\beta_1}g_1|_\nu| \partial_{\beta_2}g_2|_{2}\}|g_3|_{\nu}.
\end{equation}
\end{lemma}
\subsubsection{Estimates on linear collision terms}\label{seca.5.3.2}
For later use, we now consider the linear  terms $\Gamma(f,\frac{M-\mu}{\sqrt{\mu}})$ and $\Gamma(\frac{M-\mu}{\sqrt{\mu}},f)$.
We first deduce the mixed derivative estimates with respect to both the space variable $x$
and the velocity variable $v$.
\begin{lemma}\label{lem5.6}
Assume that \eqref{3.3} and \eqref{4.13} hold.
For $|\alpha|+|\beta|\leq N$ and $|\beta|\geq1$, one has
\begin{equation}
\label{5.20}
\frac{1}{\varepsilon}|(\partial^\alpha_\beta \Gamma(f,\frac{M-\mu}{\sqrt{\mu}}),\partial^\alpha_\beta f)|
+\frac{1}{\varepsilon}|(\partial^\alpha_\beta\Gamma(\frac{M-\mu}{\sqrt{\mu}},f),\partial^\alpha_\beta f)|
\leq C(\eta_{0}+\varepsilon^{\frac{1}{2}-a})\mathcal{D}_{N}(t).
\end{equation}
\end{lemma}
\begin{proof}
Let $|\alpha|+|\beta|\leq N$ and $|\beta|\geq1$, we have from \eqref{4.30} that
\begin{align}
\label{5.21}
&\frac{1}{\varepsilon}|(\partial^\alpha_\beta \Gamma(f,\frac{M-\mu}{\sqrt{\mu}}),\partial^\alpha_\beta f)|
\nonumber\\
&\leq C\sum_{\alpha_1\leq\alpha}\sum_{\beta_1+\beta_2\leq\beta}
\underbrace{\frac{1}{\varepsilon}\int_{\mathbb {R}^3}\{|\partial^{\alpha_1}_{\beta_1}f|_2| \partial^{\alpha-\alpha_1}_{\beta_2}(\frac{M-\mu}{\sqrt{\mu}})|_{\nu}
+|\partial^{\alpha_1}_{\beta_1}f|_\nu| \partial^{\alpha-\alpha_1}_{\beta_2}(\frac{M-\mu}{\sqrt{\mu}})|_{2}\}
|\partial^\alpha_\beta f|_{\nu}\,dx}_{J_1}.
\end{align}
For any $|\beta'|\geq0$ and any $b\geq0$, by \eqref{4.16}, there exist $ R_1>0$ large enough such that
$$
\int_{|v|\geq R_1}\langle v\rangle^{2b}|\partial _{\beta'}(\frac{M-\mu}{\sqrt{\mu}})|^2\,dv\leq C(\eta_0+\varepsilon^{1-a})^{2},
$$
and
$$
\int_{|v|\leq R_1}\langle v\rangle^{2b}|\partial _{\beta'}(\frac{M-\mu}{\sqrt{\mu}})|^2\,dv\leq C(|\rho-1|+|u|+|\theta-\frac{3}{2}|)^2\leq C
(\eta_0+\varepsilon^{1-a})^{2},
$$
which yields immediately that
\begin{equation}
\label{5.22}
| \langle v\rangle^{b}\partial _{\beta'}(\frac{M-\mu}{\sqrt{\mu}})|_{2}\leq C(\eta_0+\varepsilon^{1-a}).
\end{equation}
It is clear that $|\alpha|\leq N-1$ in \eqref{5.21} since $|\alpha|+|\beta|\leq N$ and $|\beta|\geq1$.
Note that the norm $|\cdot|_\nu$ is stronger than $|\cdot|_2$. If $|\alpha-\alpha_{1}|=0$, then we get from this, \eqref{5.22} and \eqref{4.15} that
\begin{align*}
J_1&\leq C\frac{1}{\varepsilon}\int_{\mathbb {R}^3}
|\partial^{\alpha_1}_{\beta_1}f|_\nu| \partial^{\alpha-\alpha_1}_{\beta_2}(\frac{M-\mu}{\sqrt{\mu}})|_{\nu}
|\partial^\alpha_\beta f|_{\nu}\,dx
\nonumber\\
&\leq C\frac{1}{\varepsilon}\||\partial^{\alpha-\alpha_1}_{\beta_2}(\frac{M-\mu}{\sqrt{\mu}})|_{\nu}\|_{L^{\infty}}
\|\partial^{\alpha_1}_{\beta_1}f\|_\nu\|\partial^{\alpha}_{\beta}f\|_{\nu}
\nonumber\\
&\leq C(\eta_0+\varepsilon^{1-a})\frac{1}{\varepsilon}\|\partial^{\alpha_1}_{\beta_1}f\|_\nu\|\partial^{\alpha}_{\beta}f\|_{\nu}
\nonumber\\
&\leq C(\eta_0+\varepsilon^{1-a})\mathcal{D}_{N}(t).
\end{align*}
By a simple computation, one has the following identity
\begin{equation}
\label{4.22A}
\partial_{x_i}M=M\big\{\frac{\partial_{x_i}\rho}{\rho}+\frac{(v-u)\cdot\partial_{x_i}u}{R\theta}
+(\frac{|v-u|^{2}}{2R\theta}-\frac{3}{2})\frac{\partial_{x_i}\theta}{\theta} \big\}.
\end{equation}
Then for $|\alpha|\geq 2$ and $\partial^{\alpha}=\partial^{\alpha'}\partial_{x_i}$, we easily obtain
\begin{align}
\label{4.23A}
\partial^{\alpha}M=&M\big\{\frac{\partial^{\alpha'}\partial_{x_i}\rho}{\rho}+\frac{(v-u)\cdot\partial^{\alpha'}\partial_{x_i}u}{R\theta}+(\frac{|v-u|^{2}}{2R\theta}-\frac{3}{2})\frac{\partial^{\alpha'}\partial_{x_i}\theta}{\theta} \big\}
\nonumber\\
&+\sum_{1\leq\alpha_{1}\leq \alpha'}C^{\alpha_1}_{\alpha'}\big\{\partial^{\alpha_{1}}(M\frac{1}{\rho})\partial^{\alpha'-\alpha_{1}}\partial_{x_i}\rho+\partial^{\alpha_{1}}(M\frac{v-u}{R\theta})\cdot\partial^{\alpha'-\alpha_{1}}\partial_{x_i}u
\nonumber\\
&\hspace{2cm}+\partial^{\alpha_{1}}(M\frac{|v-u|^{2}}{2R\theta^{2}}-M\frac{3}{2\theta})\partial^{\alpha'-\alpha_{1}}\partial_{x_i}\theta\big\}.
\end{align}
If $1\leq|\alpha-\alpha_{1}|\leq N-1$, 
using the Young's inequality, the embedding inequality, \eqref{3.3}, \eqref{4.13}, \eqref{4.27} and \eqref{4.13A}, one has
\begin{align}
\label{4.34}
\|\partial^{\alpha-\alpha_{1}}(\rho,u,\theta)\|_{L^3}
&\leq C\|\partial^{\alpha-\alpha_{1}}(\bar{\rho},\bar{u},\bar{\theta})\|_{L^3}
+C\|\partial^{\alpha-\alpha_{1}}(\widetilde{\rho},\widetilde{u},\widetilde{\theta})\|_{L^3}
\nonumber\\
&\leq C\eta_{0}+C\|\partial^{\alpha-\alpha_{1}}(\widetilde{\rho},\widetilde{u},\widetilde{\theta})\|^{\frac{1}{2}}
\|\nabla_{x}\partial^{\alpha-\alpha_{1}}(\widetilde{\rho},\widetilde{u},\widetilde{\theta})\|^{\frac{1}{2}}
\nonumber\\
&\leq C\eta_{0}+C\varepsilon^{\frac{1}{2}(1-a)}\varepsilon^{-\frac{1}{2}a}\leq C(\eta_0+\varepsilon^{\frac{1}{2}-a}),
\end{align}
and
\begin{equation}
\label{4.35}
\||\nabla_{x}(\rho,u,\theta)|^{|\alpha-\alpha_{1}|}\|_{L^3}
\leq C(\eta_0+\varepsilon^{\frac{1}{2}-a}).
\end{equation}
Then, for $1\leq|\alpha-\alpha_{1}|\leq N-1$, we get by \eqref{4.25}, \eqref{4.22A}, \eqref{4.23A}, \eqref{4.34}
and \eqref{4.35} that
\begin{equation}
\label{4.36}
\|| \partial^{\alpha-\alpha_1}_{\beta_2}(\frac{M-\mu}{\sqrt{\mu}})|_{\nu}\|_{L^3}\leq C(\eta_0+\varepsilon^{\frac{1}{2}-a}).
\end{equation}
Hence, for the cases $1\leq|\alpha-\alpha_{1}|\leq N-1$ in \eqref{5.21}, then $|\alpha_1|\leq |\alpha|-1$ and
$|\alpha_1|+|\beta_1|\leq|\alpha|+|\beta|-1 \leq N-1$, we use $L^{6}-L^{3}-L^{2}$ {H\"{o}lder} inequality, \eqref{4.36} and \eqref{4.15} to get
\begin{align*}
J_1&\leq C\frac{1}{\varepsilon}\int_{\mathbb {R}^3}
|\partial^{\alpha_1}_{\beta_1}f|_\nu| \partial^{\alpha-\alpha_1}_{\beta_2}(\frac{M-\mu}{\sqrt{\mu}})|_{\nu}
|\partial^\alpha_\beta f|_{\nu}\,dx
\nonumber\\
&\leq 
C\frac{1}{\varepsilon}\|| \partial^{\alpha-\alpha_1}_{\beta_2}(\frac{M-\mu}{\sqrt{\mu}})|_{\nu}\|_{L^3}
\||\partial^{\alpha_1}_{\beta_1}f|_\nu\|_{L^6}\||\partial^{\alpha}_{\beta}f|_{\nu}\|_{L^2}
\nonumber\\
%&\leq C\frac{1}{\varepsilon}(\eta_0+\varepsilon^{\frac{1}{2}-a})
%\|\nabla_x\partial^{\alpha_1}_{\beta_1}f\|_{\nu}\|\partial^{\alpha}_{\beta}f\|_{\nu}^{2}
%\\
&\leq C(\eta_0+\varepsilon^{\frac{1}{2}-a})\mathcal{D}_{N}(t).
\end{align*}
Putting the estimates on $J_1$ into \eqref{5.21}, one gets
\begin{equation}
\label{5.26}
\frac{1}{\varepsilon}|(\partial^\alpha_\beta \Gamma(f,\frac{M-\mu}{\sqrt{\mu}}),\partial^\alpha_\beta f)|
\leq C(\eta_0+\varepsilon^{\frac{1}{2}-a})\mathcal{D}_{N}(t).
\end{equation}
The second term on the LHS of \eqref{5.20} has the same bound as \eqref{5.26}. This ends the proof of Lemma \ref{lem5.6}.
\end{proof}
Similar arguments as in \eqref{5.20}, we have the following estimates.
\begin{lemma}\label{lem5.7}
For $|\alpha|\leq N-1$, it holds that
\begin{equation}
\label{5.27}
\frac{1}{\varepsilon}|(\partial^\alpha\Gamma(\frac{M-\mu}{\sqrt{\mu}},f), \partial^\alpha f)|
+\frac{1}{\varepsilon}|(\partial^\alpha \Gamma(f,\frac{M-\mu}{\sqrt{\mu}}),\partial^\alpha f)|
\leq C(\eta_{0}+\varepsilon^{\frac{1}{2}-a})\mathcal{D}_{N}(t).
\end{equation}
\end{lemma}
Note that  Lemma \ref{lem5.6} and Lemma \ref{lem5.7} do not include the $N$th-order space derivatives.
Since the $N$th-order space derivatives in $\mathcal{E}_{N}(t)$ 
are much more singular, we need to derive the $N$-order space derivative estimates separately.
\begin{lemma}\label{lem5.8}
For any $\alpha$ with $|\alpha|=N$, it holds that
\begin{align}
\label{5.28}	
&\frac{1}{\varepsilon}|(\partial^\alpha\Gamma(\frac{M-\mu}{\sqrt{\mu}},f),\frac{\partial^\alpha F}{\sqrt{\mu}})|+
\frac{1}{\varepsilon}|(\partial^\alpha\Gamma(f,\frac{M-\mu}{\sqrt{\mu}}),\frac{\partial^\alpha F}{\sqrt{\mu}})|
\nonumber\\
&\leq C(\eta_{0}+\varepsilon^{\frac{1}{2}-a})\frac{1}{\varepsilon^{2}}\mathcal{D}_{N}(t)+C(\eta_{0}+\varepsilon^{1-a})\varepsilon^{\frac{1}{2}-a}.
\end{align}
\end{lemma}
\begin{proof}
Let $|\alpha|=N$, we easily see that
\begin{align}
\label{4.40}
\frac{1}{\varepsilon}(\partial^{\alpha}\Gamma(\frac{M-\mu}{\sqrt{\mu}},f),\frac{\partial^{\alpha}F}{\sqrt{\mu}})=&
\frac{1}{\varepsilon}(\Gamma(\frac{M-\mu}{\sqrt{\mu}},\partial^{\alpha}f),\frac{\partial^{\alpha}F}{\sqrt{\mu}})
\nonumber\\
&+\sum_{1\leq\alpha_{1}\leq\alpha}C^{\alpha_{1}}_{\alpha}\frac{1}{\varepsilon}(\Gamma(
\frac{\partial^{\alpha_{1}}(M-\mu)}{\sqrt{\mu}},\partial^{\alpha-\alpha_{1}}f),\frac{\partial^{\alpha}F}{\sqrt{\mu}}).
\end{align}
In view of the decomposition $F=M+\overline{G}+\sqrt{\mu}f$, one has
\begin{align}
\label{5.29}
\frac{1}{\varepsilon}(\Gamma(\frac{M-\mu}{\sqrt{\mu}},\partial^{\alpha}f),\frac{\partial^{\alpha}F}{\sqrt{\mu}})
=&\frac{1}{\varepsilon}(\Gamma(\frac{M-\mu}{\sqrt{\mu}},\partial^{\alpha}f),\frac{\partial^{\alpha}M}{\sqrt{\mu}})
\nonumber\\
&+\frac{1}{\varepsilon}(\Gamma(\frac{M-\mu}{\sqrt{\mu}},\partial^{\alpha}f),\frac{\partial^{\alpha}\overline{G}}{\sqrt{\mu}}+\partial^{\alpha}f).
\end{align}
Let us carefully deal with the first term on the RHS of \eqref{5.29}.
For $|\alpha|=N$, denoting $\partial^{\alpha}=\partial^{\alpha'}\partial_{x_i}$ with $|\alpha'|=N-1$, we have from \eqref{4.23A} that
\begin{align}
\label{5.30}
\partial^{\alpha}M=&M\{\frac{\partial^{\alpha}\widetilde{\rho}}{\rho}+\frac{(v-u)\cdot\partial^{\alpha}\widetilde{u}}{R\theta}
+(\frac{|v-u|^{2}}{2R\theta}-\frac{3}{2})\frac{\partial^{\alpha}\widetilde{\theta}}{\theta}\}
\nonumber\\
&+M\{\frac{\partial^{\alpha}\bar{\rho}}{\rho}+\frac{(v-u)\cdot\partial^{\alpha}\bar{u}}{R\theta}
+(\frac{|v-u|^{2}}{2R\theta}-\frac{3}{2})\frac{\partial^{\alpha}\bar{\theta}}{\theta}\}
\nonumber\\
&+\sum_{1\leq\alpha_{1}\leq \alpha'}C^{\alpha_1}_{\alpha'}\big\{\partial^{\alpha_{1}}(M\frac{1}{\rho})\partial^{\alpha'-\alpha_{1}}\partial_{x_i}\rho+\partial^{\alpha_{1}}(M\frac{v-u}{R\theta})\cdot\partial^{\alpha'-\alpha_{1}}\partial_{x_i}u
\nonumber\\
&\hspace{1cm}+\partial^{\alpha_{1}}(M\frac{|v-u|^{2}}{2R\theta^{2}}-M\frac{3}{2\theta})\partial^{\alpha'-\alpha_{1}}\partial_{x_i}\theta\big\}
\nonumber\\
:=&I_{1}+I_{2}+I_{3}.
\end{align}
By the definition of $I_1$ in \eqref{5.30}, \eqref{4.30}, \eqref{4.25}, \eqref{5.22} and \eqref{4.15}, one has
\begin{align}
\label{4.44}
\frac{1}{\varepsilon}|(\Gamma(\frac{M-\mu}{\sqrt{\mu}},\partial^{\alpha}f),\frac{I_{1}}{\sqrt{\mu}})|
&\leq
C\frac{1}{\varepsilon}\int_{\mathbb{R}^{3}}\{|(\frac{M-\mu}{\sqrt{\mu}})|_2|\partial^{\alpha}f|_{\nu}
+|(\frac{M-\mu}{\sqrt{\mu}})|_\nu|\partial^{\alpha}f|_{2}\}
|\frac{I_{1}}{\sqrt{\mu}}|_\nu\,dx
\nonumber\\
&\leq
C\frac{1}{\varepsilon}\int_{\mathbb{R}^{3}}|(\frac{M-\mu}{\sqrt{\mu}})|_\nu|\partial^{\alpha}f|_{\nu}
|\partial^{\alpha}(\widetilde{\rho},\widetilde{u},\widetilde{\theta})|\,dx
\nonumber\\
&\leq C(\eta_{0}+\varepsilon^{1-a})\frac{1}{\varepsilon}(\|\partial^{\alpha}f\|^{2}_{\nu}+\|\partial^{\alpha}(\widetilde{\rho},\widetilde{u},\widetilde{\theta})\|^{2})
\nonumber\\
&\leq C(\eta_{0}+\varepsilon^{1-a})\frac{1}{\varepsilon^{2}}\mathcal{D}_{N}(t).
\end{align}
Recalling that $\partial^{\alpha}=\partial^{\alpha'}\partial_{x_i}$ with $|\alpha'|=N-1$, we have from the integration by parts that
\begin{align}
\label{5.32}
\frac{1}{\varepsilon}&(\Gamma(\frac{M-\mu}{\sqrt{\mu}},\partial^{\alpha}f),\frac{I_{2}}{\sqrt{\mu}})
\nonumber\\
=&-\frac{1}{\varepsilon}\big(\Gamma(\frac{M-\mu}{\sqrt{\mu}},\partial^{\alpha'}f)
,\partial_{x_i}[\frac{M}{\sqrt{\mu}}\{\frac{\partial^{\alpha}\bar{\rho}}{\rho}+\frac{(v-u)\cdot\partial^{\alpha}\bar{u}}{R\theta}
+(\frac{|v-u|^{2}}{2R\theta}-\frac{3}{2})\frac{\partial^{\alpha}\bar{\theta}}{\theta}\}]\big)
\nonumber\\
&-\frac{1}{\varepsilon}\big(\Gamma(\partial_{x_i}[\frac{M-\mu}{\sqrt{\mu}}],\partial^{\alpha'}f),\frac{M}{\sqrt{\mu}}\{\frac{\partial^{\alpha}\bar{\rho}}{\rho}+\frac{(v-u)\cdot\partial^{\alpha}\bar{u}}{R\theta}
+(\frac{|v-u|^{2}}{2R\theta}-\frac{3}{2})\frac{\partial^{\alpha}\bar{\theta}}{\theta}\}\big),
\end{align}
which can be further bounded by
\begin{align}
\label{5.33}
&C\frac{1}{\varepsilon}\int_{\mathbb{R}^{3}}|\partial^{\alpha'}f|_{\nu}
\big\{|\partial_{x_i}\partial^{\alpha}(\bar{\rho},\bar{u},\bar{\theta})|
+|\partial^{\alpha}(\bar{\rho},\bar{u},\bar{\theta})||\partial_{x_i}(\rho,u,\theta)|\big\}\,dx
\nonumber\\
&\leq C(\eta_{0}+\varepsilon^{1-a})\frac{1}{\varepsilon}(\frac{1}{\varepsilon^{2}}\|\partial^{\alpha'}f\|^{2}_{\nu}+\varepsilon^2)
\nonumber\\
&\leq C(\eta_{0}+\varepsilon^{1-a})\frac{1}{\varepsilon^{2}}\mathcal{D}_{N}(t)+C(\eta_{0}+\varepsilon^{1-a})\varepsilon,
\end{align}
according to \eqref{4.30}, \eqref{4.25}, \eqref{5.22}, \eqref{3.3}, \eqref{4.13} and \eqref{4.15}.
As for the term $I_3$  in \eqref{5.30}, we write $I_3=I_3^1+I_3^2$ with
\begin{align*}
I^1_{3}=\sum_{1\leq\alpha_{1}\leq \alpha'}C^{\alpha_1}_{\alpha'}&\big(\partial^{\alpha_{1}}(M\frac{1}{\rho})\partial^{\alpha'-\alpha_{1}}\partial_{x_i}\widetilde{\rho}+\partial^{\alpha_{1}}(M\frac{v-u}{R\theta})\cdot\partial^{\alpha'-\alpha_{1}}\partial_{x_i}\widetilde{u}
\\
&+\partial^{\alpha_{1}}(M\frac{|v-u|^{2}}{2R\theta^{2}}-M\frac{3}{2\theta})\partial^{\alpha'-\alpha_{1}}\partial_{x_i}\widetilde{\theta}\big),
\end{align*}
and
\begin{align*}
I_3^2=\sum_{1\leq\alpha_{1}\leq \alpha'}C^{\alpha_1}_{\alpha'}&\big(\partial^{\alpha_{1}}(M\frac{1}{\rho})\partial^{\alpha'-\alpha_{1}}\partial_{x_i}\bar{\rho}+\partial^{\alpha_{1}}(M\frac{v-u}{R\theta})\cdot\partial^{\alpha'-\alpha_{1}}\partial_{x_i}\bar{u}
\\
&+\partial^{\alpha_{1}}(M\frac{|v-u|^{2}}{2R\theta^{2}}-M\frac{3}{2\theta})\partial^{\alpha'-\alpha_{1}}\partial_{x_i}\bar{\theta}\big).
\end{align*}
For $1\leq\alpha_{1}\leq \alpha'$ with $|\alpha'|=N-1$, we get by \eqref{4.30} and \eqref{5.22} that
\begin{align*}
&\frac{1}{\varepsilon}|(\Gamma(\frac{M-\mu}{\sqrt{\mu}},\partial^{\alpha}f),\frac{1}{\sqrt{\mu}}
\partial^{\alpha_{1}}(M\frac{1}{\rho})\partial^{\alpha'-\alpha_{1}}\partial_{x_i}\widetilde{\rho})|
\\
\leq& C\frac{1}{\varepsilon}\int_{\mathbb{R}^{3}}|\frac{M-\mu}{\sqrt{\mu}}|_\nu|\partial^{\alpha}f|_\nu
|\frac{1}{\sqrt{\mu}}\partial^{\alpha_{1}}(M\frac{1}{\rho})\partial^{\alpha'-\alpha_{1}}\partial_{x_i}\widetilde{\rho}|_\nu\,dx
\\
\leq& C(\eta_{0}+\varepsilon^{1-a})\frac{1}{\varepsilon}
\||\partial^{\alpha}f|_{\nu}\|_{L^2}\|\partial^{\alpha'-\alpha_{1}}\partial_{x_i}\widetilde{\rho}\|_{L^6}
\||\frac{1}{\sqrt{\mu}}\partial^{\alpha_{1}}(M\frac{1}{\rho})|_\nu\|_{L^3}
\\
\leq& C(\eta_{0}+\varepsilon^{1-a})\frac{1}{\varepsilon^{2}}\mathcal{D}_{N}(t),
\end{align*}
where in the last inequality we have used the similar argument as \eqref{4.36}, \eqref{4.15} and \eqref{4.22a}. 
The other terms in $I_3^1$ can be handed similarly, it follows that
\begin{equation*}
\frac{1}{\varepsilon}|(\Gamma(\frac{M-\mu}{\sqrt{\mu}},\partial^{\alpha}f),\frac{I^1_{3}}{\sqrt{\mu}})|
\leq C(\eta_{0}+\varepsilon^{1-a})\frac{1}{\varepsilon^{2}}\mathcal{D}_{N}(t).
\end{equation*}
Similar arguments as \eqref{5.33}, one gets 
\begin{equation*}
\frac{1}{\varepsilon}|(\Gamma(\frac{M-\mu}{\sqrt{\mu}},\partial^{\alpha}f),\frac{I^2_{3}}{\sqrt{\mu}})|
\leq C(\eta_{0}+\varepsilon^{1-a})\frac{1}{\varepsilon^{2}}\mathcal{D}_{N}(t)+C(\eta_{0}+\varepsilon^{1-a})\varepsilon.
\end{equation*}
Thanks to $I_3=I_3^1+I_3^2$, we deduce from the aforementioned two estimates  that
\begin{equation}
\label{5.34}
\frac{1}{\varepsilon}|(\Gamma(\frac{M-\mu}{\sqrt{\mu}},\partial^{\alpha}f),\frac{I_{3}}{\sqrt{\mu}})|
\leq C(\eta_{0}+\varepsilon^{1-a})\frac{1}{\varepsilon^{2}}\mathcal{D}_{N}(t)+C(\eta_{0}+\varepsilon^{1-a})\varepsilon.
\end{equation} 
Combining \eqref{5.34}, \eqref{4.44} and \eqref{5.33} yields that
\begin{equation}
\label{5.35}
\frac{1}{\varepsilon}|(\Gamma(\frac{M-\mu}{\sqrt{\mu}},\partial^{\alpha}f),\frac{\partial^{\alpha}M}{\sqrt{\mu}})| 
\leq C(\eta_{0}+\varepsilon^{1-a})\frac{1}{\varepsilon^{2}}\mathcal{D}_{N}(t)+C(\eta_{0}+\varepsilon^{1-a})\varepsilon.
\end{equation}
The second term on the RHS of \eqref{5.29} has the same bound as \eqref{5.35}.
%\begin{align*}
%&\frac{1}{\varepsilon}|(\Gamma(\frac{M-\mu}{\sqrt{\mu}},\partial^{\alpha}f),\frac{\partial^{\alpha}\overline{G}}{\sqrt{\mu}}+\partial^{\alpha}f)|
%\\&\leq C\frac{1}{\varepsilon}\int_{\mathbb{R}^{3}}|\frac{M-\mu}{\sqrt{\mu}}|_{\nu}|\partial^{\alpha}f|_{\nu}
%(|\frac{\partial^{\alpha}\overline{G}}{\sqrt{\mu}}|_{\nu}+|\partial^{\alpha}f|_{\nu})\,dx
%\\
%&\leq C(\eta_{0}+\varepsilon^{1-a})\frac{1}{\varepsilon}\|\partial^{\alpha}f\|^2_{\nu}
%+C(\eta_{0}+\varepsilon^{1-a})\frac{1}{\varepsilon}\|\frac{\partial^{\alpha}\overline{G}}{\sqrt{\mu}}\|^2_{\nu}
%\\&\leq C(\eta_{0}+\varepsilon^{1-a})\frac{1}{\varepsilon^{2}}\mathcal{D}_{N}(t)+C(\eta_{0}+\varepsilon^{1-a})\varepsilon^{1-2a}.
%\end{align*}
Hence, one gets
\begin{equation}
\label{5.36}
\frac{1}{\varepsilon}|(\Gamma(\frac{M-\mu}{\sqrt{\mu}},\partial^{\alpha}f),\frac{\partial^{\alpha}F}{\sqrt{\mu}})|
\leq C(\eta_{0}+\varepsilon^{1-a})\frac{1}{\varepsilon^{2}}\mathcal{D}_{N}(t)+C(\eta_{0}+\varepsilon^{1-a})\varepsilon^{1-2a}.
\end{equation}
It remains to compute the last term in \eqref{4.40}. Note that the norm $|\cdot|_\nu$ is stronger than $|\cdot|_2$, we have 
from \eqref{4.30} that
\begin{equation*}
\sum_{1\leq\alpha_{1}\leq\alpha}\frac{1}{\varepsilon}
|(\Gamma(\frac{\partial^{\alpha_{1}}(M-\mu)}{\sqrt{\mu}},\partial^{\alpha-\alpha_{1}}f),\frac{\partial^{\alpha}F}{\sqrt{\mu}})|
\leq C\sum_{1\leq\alpha_{1}\leq\alpha}\underbrace{\frac{1}{\varepsilon}\int_{\mathbb{R}^{3}}
|\frac{\partial^{\alpha_{1}}M}{\sqrt{\mu}}|_{\nu}|\partial^{\alpha-\alpha_{1}}f|_{\nu}
|\frac{\partial^{\alpha}F}{\sqrt{\mu}}|_{\nu}\,dx}_{J_2}.
\end{equation*}
To bound $J_2$, we first give the following two estimates for $|\alpha|=N$, 
\begin{align}
\label{4.49}
\|\frac{\partial^{\alpha}F}{\sqrt{\mu}}\|_{\nu}
&\leq \|\partial^{\alpha}f\|_{\nu}+\|\frac{\partial^{\alpha}\overline{G}}{\sqrt{\mu}}\|_{\nu}
+\|\frac{\partial^{\alpha}M}{\sqrt{\mu}}\|_{\nu}
\nonumber\\
&\leq C(\|\partial^{\alpha}f\|_{\nu}+\|\partial^{\alpha}(\widetilde{\rho},\widetilde{u},\widetilde{\theta})\|+\eta_{0}+\varepsilon^{1-a}),
\end{align}
and
\begin{align}
\label{4.47a}
\|\frac{\partial^{\alpha}M}{\sqrt{\mu}}\|_\nu
&\leq C(\|\partial^{\alpha}(\widetilde{\rho},\widetilde{u},\widetilde{\theta})\|+\eta_{0}+\varepsilon^{1-a})
\nonumber\\
&\leq C(\varepsilon^{-a}+\eta_{0}+\varepsilon^{1-a})\leq C\varepsilon^{-a}.
\end{align}
Here we have used \eqref{4.23A}, \eqref{4.20}, \eqref{4.25}, \eqref{3.3}, \eqref{4.13} and \eqref{4.22a}.
If $|\alpha_{1}|=N$, we have from \eqref{4.47a}, \eqref{4.49} and \eqref{4.15} that
\begin{align*}
J_2\leq& C\frac{1}{\varepsilon}\|\frac{\partial^{\alpha_{1}}M}{\sqrt{\mu}}\|_\nu\||\partial^{\alpha-\alpha_{1}}f|_{\nu}\|_{L^{\infty}}
\|\frac{\partial^{\alpha}F}{\sqrt{\mu}}\|_{\nu}
\\
\leq&C\frac{1}{\varepsilon}\varepsilon^{-a}(\varepsilon^{-\frac{3}{2}}\|\nabla_xf\|_{\nu}\|\nabla_x^{2}f\|_{\nu}
+\varepsilon^{\frac{3}{2}}\|\frac{\partial^{\alpha}F}{\sqrt{\mu}}\|^2_{\nu})
\\
\leq& C\varepsilon^{\frac{1}{2}-a}\frac{1}{\varepsilon^{2}}\mathcal{D}_{N}(t)+C(\eta_{0}+\varepsilon^{1-a})\varepsilon^{\frac{1}{2}-a}.
\end{align*}
If $N/2<|\alpha_{1}|\leq N-1$, {then by} $L^{6}-L^{3}-L^{2}$ H\"{o}lder inequality, \eqref{4.36}, \eqref{4.49}, \eqref{4.22a} and \eqref{4.15}, we arrive at
\begin{align*}	
J_2\leq& C\frac{1}{\varepsilon}\||\frac{\partial^{\alpha_{1}}M}{\sqrt{\mu}}|_\nu\|_{L^{3}}\||\partial^{\alpha-\alpha_{1}}f|_{\nu}\|_{L^{6}}
\|\frac{\partial^{\alpha}F}{\sqrt{\mu}}\|_{\nu}
\\
\leq& C(\eta_{0}+\varepsilon^{\frac{1}{2}-a})\frac{1}{\varepsilon}
(\frac{1}{\varepsilon^2}\|\nabla_x\partial^{\alpha-\alpha_{1}}f\|^2_{\nu}+\varepsilon^2\|\frac{\partial^{\alpha}F}{\sqrt{\mu}}\|^2_{\nu})
\\
\leq& C(\eta_{0}+\varepsilon^{\frac{1}{2}-a})\frac{1}{\varepsilon^{2}}\mathcal{D}_{N}(t)+C(\eta_{0}+\varepsilon^{1-a})\varepsilon^{\frac{1}{2}-a}.
\end{align*}
If $1\leq|\alpha_{1}|\leq N/2$, then by the similar calculations as \eqref{4.36}, one has
\begin{equation}
\label{5.38}
\||\frac{\partial^{\alpha_{1}}M}{\sqrt{\mu}}|_{\nu}\|_{L^{\infty}}
\leq C(\eta_0+\varepsilon^{\frac{1}{2}-a}),
\end{equation}
which together with \eqref{4.49}, \eqref{4.22a} and \eqref{4.15}, yields
\begin{align*}	
J_2\leq& C\frac{1}{\varepsilon}\||\frac{\partial^{\alpha_{1}}M}{\sqrt{\mu}}|_{\nu}\|_{L^{\infty}}
\|\partial^{\alpha-\alpha_{1}}f\|_{\nu}\|\frac{\partial^{\alpha}F}{\sqrt{\mu}}\|_{\nu}
\\	
\leq& 
C(\eta_{0}+\varepsilon^{\frac{1}{2}-a})\frac{1}{\varepsilon^{2}}\mathcal{D}_{N}(t)+C(\eta_{0}+\varepsilon^{1-a})\varepsilon^{\frac{1}{2}-a}.
\end{align*}
Collecting the above estimates on $J_2$, we conclude that
\begin{equation*}
\sum_{1\leq\alpha_{1}\leq\alpha}\frac{1}{\varepsilon}|(\Gamma(\frac{\partial^{\alpha_{1}}(M-\mu)}{\sqrt{\mu}},\partial^{\alpha-\alpha_{1}}f),\frac{\partial^{\alpha}F}{\sqrt{\mu}})|
\leq C(\eta_{0}+\varepsilon^{\frac{1}{2}-a})\frac{1}{\varepsilon^{2}}\mathcal{D}_{N}(t)+C(\eta_{0}+\varepsilon^{1-a})\varepsilon^{\frac{1}{2}-a}.
\end{equation*}
Hence, putting this and \eqref{5.36} into \eqref{4.40} and taking $a\in[0,1/2)$, we obtain
\begin{equation}
\label{5.40}
\frac{1}{\varepsilon}|(\partial^{\alpha}\Gamma(\frac{M-\mu}{\sqrt{\mu}},f),\frac{\partial^{\alpha}F}{\sqrt{\mu}})|\leq  C(\eta_{0}+\varepsilon^{\frac{1}{2}-a})\frac{1}{\varepsilon^{2}}\mathcal{D}_{N}(t)+C(\eta_{0}+\varepsilon^{1-a})\varepsilon^{\frac{1}{2}-a}.
\end{equation}
The second term on the LHS of \eqref{5.28}  has same bound as 
\eqref{5.40}. So the desired estimate \eqref{5.28} follows. This completes
the proof of Lemma \ref{lem5.8}.
\end{proof}
\subsubsection{Estimates on non-linear collision terms}\label{seca.5.3.3}
Next we consider the nonlinear collision terms  $\Gamma(\frac{G}{\sqrt{\mu}},\frac{G}{\sqrt{\mu}})$. 
As before, we first derive the mixed derivative estimates with respect to both the space variable $x$
and the velocity variable $v$.
\begin{lemma}
\label{lem5.9}
For  $|\alpha|+|\beta|\leq N$ and $|\beta|\geq1$, it holds that
\begin{equation}
\label{5.41}
\frac{1}{\varepsilon}|(\partial^\alpha_\beta\Gamma(\frac{G}{\sqrt{\mu}},\frac{G}{\sqrt{\mu}}), \partial^\alpha_\beta f)|
\leq  C(\eta_{0}+\varepsilon^{\frac{1}{2}-a})\mathcal{D}_{N}(t)+C\varepsilon\varepsilon^{2-2a}.
\end{equation}
\end{lemma}
\begin{proof}
Let $|\alpha|+|\beta|\leq N$ and $|\beta|\geq1$, we use $G=\overline{G}+\sqrt{\mu}f$ to write
\begin{equation*}
\partial^\alpha_\beta\Gamma(\frac{G}{\sqrt{\mu}},\frac{G}{\sqrt{\mu}})
=\partial^\alpha_\beta\Gamma(\frac{\overline{G}}{\sqrt{\mu}},\frac{\overline{G}}{\sqrt{\mu}})
+\partial^\alpha_\beta\Gamma(\frac{\overline{G}}{\sqrt{\mu}},f)+\partial^\alpha_\beta\Gamma(f,\frac{\overline{G}}{\sqrt{\mu}})
+\partial^\alpha_\beta\Gamma(f,f).
\end{equation*}
We take the inner product of the above equality with $\frac{1}{\varepsilon}\partial^\alpha_\beta f$
and then compute each term.  By \eqref{4.30}, Lemma \ref{lem5.3}, the embedding inequality, we arrive at
\begin{align}
\label{5.42}
&\frac{1}{\varepsilon}|(\partial^\alpha_\beta \Gamma(\frac{\overline{G}}{\sqrt{\mu}},\frac{\overline{G}}{\sqrt{\mu}}),\partial^\alpha_\beta f)|
\nonumber\\
&\leq C\frac{1}{\varepsilon}\sum_{\alpha_1\leq\alpha}\sum_{\beta_1+\beta_2\leq\beta}\int_{{\mathbb R}^3}|\partial^{\alpha_1}_{\beta_1}(\frac{\overline{G}}{\sqrt{\mu}})|_\nu| \partial^{\alpha-\alpha_1}_{\beta_2}(\frac{\overline{G}}{\sqrt{\mu}})|_{\nu}|\partial^\alpha_\beta f|_{\nu}\,dx
\nonumber\\
&\leq C\eta_{0}\mathcal{D}_{N}(t)+C\varepsilon\varepsilon^{2-2a}.
\end{align}
Using \eqref{4.30} again, it is clear  that
\begin{equation*}
\frac{1}{\varepsilon}|(\partial^\alpha_\beta \Gamma(\frac{\overline{G}}{\sqrt{\mu}},f),\partial^\alpha_\beta f)|
\leq C\sum_{\alpha_1\leq\alpha}\sum_{\beta_1+\beta_2\leq\beta}
\underbrace{\frac{1}{\varepsilon}\int_{{\mathbb R}^3}|\partial^{\alpha_1}_{\beta_1}(\frac{\overline{G}}{\sqrt{\mu}})|_\nu| \partial^{\alpha-\alpha_1}_{\beta_2}f|_{\nu}|\partial^\alpha_\beta f|_{\nu}\,dx}_{J_3}.
\end{equation*}
To compute $J_3$, first note that $|\alpha_1|\leq|\alpha|\leq N-1$ in $J_3$ since we only consider the case $|\alpha|+|\beta|\leq N$ and $|\beta|\geq1$.
If $|\alpha_1|+|\beta_1|\leq N/2$, one gets from Lemma \ref{lem5.3} and \eqref{4.15} that
\begin{equation*}
J_3\leq C\frac{1}{\varepsilon}\||\partial^{\alpha_{1}}_{\beta_1}(\frac{\overline{G}}{\sqrt{\mu}})|_{\nu}\|_{L^\infty}
\|\partial^{\alpha-\alpha_{1}}_{\beta_{2}}f\|_{\nu}\|\partial^{\alpha}_{\beta}f\|_{\nu}
\leq C\eta_{0}\varepsilon^{1-a}\mathcal{D}_{N}(t).
\end{equation*}
If $N/2<|\alpha_1|+|\beta_1|\leq N$, we get by \eqref{4.19}  and \eqref{4.15} that
\begin{align*}
J_3&\leq C\frac{1}{\varepsilon}\|\partial^{\alpha_{1}}_{\beta_1}(\frac{\overline{G}}{\sqrt{\mu}})\|_{\nu}
\||\partial^{\alpha-\alpha_{1}}_{\beta_{2}}f|_{\nu}\|_{L^\infty}\|\partial^{\alpha}_{\beta}f\|_{\nu}
\\
&\leq C\eta_{0}(\varepsilon\|\partial^{\alpha-\alpha_{1}}_{\beta_{2}}\nabla_xf\|_{\nu}
\|\partial^{\alpha-\alpha_{1}}_{\beta_{2}}\nabla^2_xf\|_{\nu}+\frac{1}{\varepsilon}\|\partial^{\alpha}_{\beta}f\|^2_{\nu})
\\
&\leq C\eta_{0}\mathcal{D}_{N}(t).
\end{align*}
Therefore, we conclude from the above estimates on $J_3$ that
\begin{equation}
\label{5.44}
\frac{1}{\varepsilon}|(\partial^\alpha_\beta \Gamma(\frac{\overline{G}}{\sqrt{\mu}},f),\partial^\alpha_\beta f)|
\leq C(\eta_{0}+\varepsilon^{1-a})\mathcal{D}_{N}(t).
\end{equation}
Similar to \eqref{5.44}, one has
\begin{equation}
\label{5.45}
\frac{1}{\varepsilon}|(\partial^\alpha_\beta \Gamma(f,\frac{\overline{G}}{\sqrt{\mu}}),\partial^\alpha_\beta f)|
\leq C(\eta_{0}+\varepsilon^{1-a})\mathcal{D}_{N}(t).
\end{equation}
From \eqref{4.30}, we can easily see that
\begin{equation*}
\frac{1}{\varepsilon}|(\partial^\alpha_\beta \Gamma(f,f),\partial^\alpha_\beta f)|
\leq  C\sum_{\alpha_1\leq\alpha}\sum_{\beta_1+\beta_2\leq\beta}
\underbrace{\frac{1}{\varepsilon}\int_{{\mathbb R}^3}
\{|\partial^{\alpha_1}_{\beta_1}f|_2|\partial^{\alpha-\alpha_1}_{\beta_2}f|_{\nu}
+|\partial^{\alpha_1}_{\beta_1}f|_\nu|\partial^{\alpha-\alpha_1}_{\beta_2}f|_{2}\}|\partial^\alpha_\beta f|_{\nu}\,dx}_{J_{4}}.
\end{equation*}
The term $J_4$ can be treated in the similar way as $J_3$.
We consider two cases  $|\alpha_1|+|\beta_1|\leq N/2$  and $N/2<|\alpha_1|+|\beta_1|\leq N$.
For the first case, one has from \eqref{4.13}  and \eqref{4.15} that
\begin{align*}
J_{4}&\leq C\frac{1}{\varepsilon}\{\||\partial^{\alpha_{1}}_{\beta_1}f|_{2}\|_{L^\infty}
\|\partial^{\alpha-\alpha_{1}}_{\beta_{2}}f\|_{\nu}
+\||\partial^{\alpha_{1}}_{\beta_1}f|_{\nu}\|_{L^\infty}
\|\partial^{\alpha-\alpha_{1}}_{\beta_{2}}f\|\}\|\partial^{\alpha}_{\beta}f\|_{\nu}
\nonumber\\
&\leq C\frac{1}{\varepsilon}\{\varepsilon^{\frac{1}{2}-a}
\|\partial^{\alpha-\alpha_{1}}_{\beta_{2}}f\|_{\nu}
+\varepsilon^{1-a}\||\partial^{\alpha_{1}}_{\beta_1}f|_{\nu}\|_{L^\infty}\}
\|\partial^{\alpha}_{\beta}f\|_{\nu}
\nonumber\\
&\leq C\frac{1}{\varepsilon}\varepsilon^{\frac{1}{2}-a}\{
\|\partial^{\alpha-\alpha_{1}}_{\beta_{2}}f\|^2_{\nu}
+\varepsilon\|\nabla_x\partial^{\alpha_{1}}_{\beta_1}f\|_{\nu}\|\nabla^2_x\partial^{\alpha_{1}}_{\beta_1}f\|_{\nu}
+\|\partial^{\alpha}_{\beta}f\|^2_{\nu}\}
\nonumber\\
&\leq C\varepsilon^{\frac{1}{2}-a}\mathcal{D}_{N}(t).
\end{align*}
The second case has the same bound as above. It follows that
\begin{align*}
\frac{1}{\varepsilon}|(\partial^\alpha_\beta \Gamma(f,f),\partial^\alpha_\beta f)|
\leq C\varepsilon^{\frac{1}{2}-a}\mathcal{D}_{N}(t).
\end{align*}
This, \eqref{5.42}, \eqref{5.44} and \eqref{5.45} together, gives
the desired estimate \eqref{5.41}. We thus complete the proof of Lemma \ref{lem5.9}.
\end{proof}
The rest parts, we only consider the derivative estimates with respect to the space variable $x$.
\begin{lemma}
\label{lem5.10}
For $|\alpha|\leq N-1$, it holds that
\begin{align}
\label{5.46}
\frac{1}{\varepsilon}|(\partial^\alpha\Gamma(\frac{G}{\sqrt{\mu}},\frac{G}{\sqrt{\mu}}),\partial^\alpha f)|
\leq  C(\eta_{0}+\varepsilon^{\frac{1}{2}-a})\mathcal{D}_{N}(t)+C\varepsilon\varepsilon^{2-2a}.
\end{align}
\end{lemma}
\begin{proof}
The proof of \eqref{5.46} follows along the same lines as the proof of \eqref{5.41},
and we omit its proof for brevity.
\end{proof}
\begin{lemma}
\label{lem5.11}
For $|\alpha|= N$, it holds that
\begin{align}
\label{5.47}	
\frac{1}{\varepsilon}|\partial^{\alpha}(\Gamma(\frac{G}{\sqrt{\mu}},\frac{G}{\sqrt{\mu}}),\frac{\partial^{\alpha}F}{\sqrt{\mu}})|\leq C(\eta_{0}+\varepsilon^{\frac{1}{2}-a})\frac{1}{\varepsilon^{2}}\mathcal{D}_{N}(t)+C(\eta_{0}+\varepsilon^{1-a})\varepsilon^{\frac{1}{2}-a}.
\end{align}
\end{lemma}
\begin{proof}
For $|\alpha|=N$, since $G=\overline{G}+\sqrt{\mu}f$, the term on the LHS of \eqref{5.47} is equivalent to
\begin{align}
\label{5.48}
\frac{1}{\varepsilon}|(\partial^{\alpha}\Gamma(\frac{\overline{G}}{\sqrt{\mu}},\frac{\overline{G}}{\sqrt{\mu}})+\partial^{\alpha}\Gamma(\frac{\overline{G}}{\sqrt{\mu}},f)+\partial^{\alpha}\Gamma(f,\frac{\overline{G}}{\sqrt{\mu}})+\partial^{\alpha}\Gamma(f,f),\frac{\partial^{\alpha}F}{\sqrt{\mu}})|.
\end{align}
Let's carefully deal with the terms in \eqref{5.48}. Note that the norm $|\cdot|_\nu$ is stronger than $|\cdot|_2$.
By this, \eqref{4.30}, Lemma \ref{lem5.3}, \eqref{4.49} and \eqref{4.22a}, we arrive at
\begin{align*}
\frac{1}{\varepsilon}|(\partial^{\alpha}\Gamma(\frac{\overline{G}}{\sqrt{\mu}},
\frac{\overline{G}}{\sqrt{\mu}}),\frac{\partial^{\alpha}F}{\sqrt{\mu}})|&\leq C\frac{1}{\varepsilon}\sum_{\alpha_{1}\leq\alpha}\int_{\mathbb{R}^{3}}|\partial^{\alpha_{1}}(\frac{\overline{G}}{\sqrt{\mu}})|_{\nu}
|\partial^{\alpha-\alpha_{1}}(\frac{\overline{G}}{\sqrt{\mu}})|_{\nu}|\frac{\partial^{\alpha}F}{\sqrt{\mu}}|_{\nu}\,dx
\\
&\leq C\eta_{0}\frac{1}{\varepsilon}\mathcal{D}_{N}(t)+C(\eta_{0}+\varepsilon^{1-a})\varepsilon^{1-a}.
\end{align*}
For the second in \eqref{5.48}, from \eqref{4.30}, we easily see that
\begin{align*}
\frac{1}{\varepsilon}|(\partial^{\alpha}\Gamma(\frac{\overline{G}}{\sqrt{\mu}},f),\frac{\partial^{\alpha}F}{\sqrt{\mu}})|
\leq C\sum_{\alpha_{1}\leq\alpha}\underbrace{\frac{1}{\varepsilon}\int_{\mathbb{R}^{3}}|\partial^{\alpha_{1}}(\frac{\overline{G}}{\sqrt{\mu}})|_{\nu}
|\partial^{\alpha-\alpha_{1}}f|_{\nu}|\frac{\partial^{\alpha}F}{\sqrt{\mu}}|_{\nu}\,dx}_{J_5}.
\end{align*}
To bound $J_5$, we consider two cases $|\alpha_{1}|\leq N/2$ and $N/2<|\alpha_{1}|\leq N$.
For the first case, we deduce from  Lemma \ref{lem5.3}, \eqref{4.49}, \eqref{4.22a} and \eqref{4.15} that
\begin{align*}
J_5&\leq C\frac{1}{\varepsilon}\||\partial^{\alpha_{1}}(\frac{\overline{G}}{\sqrt{\mu}})|_\nu\|_{L^{\infty}}
\|\partial^{\alpha-\alpha_{1}}f\|_{\nu}\|\frac{\partial^{\alpha}F}{\sqrt{\mu}}\|_{\nu}
\\
&\leq \eta_{0}\frac{1}{\varepsilon}\|\partial^{\alpha-\alpha_{1}}f\|^2_{\nu}
+C\frac{1}{\eta_{0}}\frac{1}{\varepsilon}\||\partial^{\alpha_{1}}(\frac{\overline{G}}{\sqrt{\mu}})|_\nu\|^2_{L^{\infty}}
\|\frac{\partial^{\alpha}F}{\sqrt{\mu}}\|^2_{\nu}
\\
&\leq C(\eta_{0}+\varepsilon^{1-a})\frac{1}{\varepsilon^{2}}\mathcal{D}_{N}(t)+C(\eta_{0}+\varepsilon^{1-a})\varepsilon^{\frac{1}{2}-a}.
\end{align*}
For the second case, $J_5$ has the same bound as above. We thereby obtain
\begin{align}
\label{4.50a}
\frac{1}{\varepsilon}|(\partial^{\alpha}\Gamma(\frac{\overline{G}}{\sqrt{\mu}},f),\frac{\partial^{\alpha}F}{\sqrt{\mu}})|
\leq C(\eta_{0}+\varepsilon^{1-a})\frac{1}{\varepsilon^{2}}\mathcal{D}_{N}(t)+C(\eta_{0}+\varepsilon^{1-a})\varepsilon^{\frac{1}{2}-a}.
\end{align}
For the third term in \eqref{5.48}, we use the similar arguments as \eqref{4.50a} to
get the same estimate. For the last term in \eqref{5.48}, the \eqref{4.30} yields
\begin{align*}
\frac{1}{\varepsilon}|(\partial^{\alpha}\Gamma(f,f),\frac{\partial^{\alpha}F}{\sqrt{\mu}})|\leq C\sum_{\alpha_{1}\leq\alpha}\underbrace{\frac{1}{\varepsilon}\int_{\mathbb{R}^{3}}\{|\partial^{\alpha_{1}}f|_{2}
|\partial^{\alpha-\alpha_{1}}f|_{\nu}+|\partial^{\alpha_{1}}f|_{\nu}
|\partial^{\alpha-\alpha_{1}}f|_{2}\}
|\frac{\partial^{\alpha}F}{\sqrt{\mu}}|_{\nu}\,dx}_{J_6}.
\end{align*}
If $|\alpha_{1}|=0$, then by \eqref{4.49}, \eqref{4.13} and \eqref{4.15}, we get
\begin{align*}
J_6&\leq C\frac{1}{\varepsilon}\{\||\partial^{\alpha_{1}}f|_2\|_{L^{\infty}}\|\partial^{\alpha-\alpha_{1}}f\|_{\nu}
+\||\partial^{\alpha_{1}}f|_\nu\|_{L^{\infty}}\|\partial^{\alpha-\alpha_{1}}f\|
\}\|\frac{\partial^{\alpha}F}{\sqrt{\mu}}\|_{\nu}
\\
&\leq C\frac{1}{\varepsilon}\{\varepsilon^{1-a}\|\partial^{\alpha}f\|_{\nu}
+\varepsilon^{-a}\|\nabla_xf\|^{\frac{1}{2}}_{\nu}\|\nabla^2_xf\|^{\frac{1}{2}}_{\nu}\}
\|\frac{\partial^{\alpha}F}{\sqrt{\mu}}\|_{\nu}
\\
&\leq C\frac{1}{\varepsilon}\varepsilon^{\frac{1}{2}-a}\{\|\partial^{\alpha}f\|^2_{\nu}
+\varepsilon^{-2}\|\nabla_xf\|_{\nu}\|\nabla^2_xf\|_{\nu}+
\varepsilon\|\frac{\partial^{\alpha}F}{\sqrt{\mu}}\|^2_{\nu}\}
\\
&\leq C\varepsilon^{\frac{1}{2}-a}\frac{1}{\varepsilon^{2}}\mathcal{D}_{N}(t)+C(\eta_{0}+\varepsilon^{1-a})\varepsilon^{\frac{1}{2}-a}.
\end{align*}
The case $|\alpha_{1}|=|\alpha|=N$ can be treated in the similar way as  above.
If $1\leq|\alpha_{1}|\leq N-1$, then by the $L^{6}-L^{3}-L^{2}$ H\"{o}lder inequality, \eqref{4.49}, \eqref{4.13} and \eqref{4.15}, one has 
\begin{align*}
J_6&\leq C\frac{1}{\varepsilon}\{\||\partial^{\alpha_{1}}f|_2\|_{L^{6}}\||\partial^{\alpha-\alpha_{1}}f|_\nu\|_{_{L^{3}}}
+\||\partial^{\alpha_{1}}f|_\nu\|_{L^{6}}\||\partial^{\alpha-\alpha_{1}}f|_2\|_{L^3}
\}\|\frac{\partial^{\alpha}F}{\sqrt{\mu}}\|_{\nu}
\\
&\leq C\varepsilon^{\frac{1}{2}-a}\frac{1}{\varepsilon^{2}}\mathcal{D}_{N}(t)+C(\eta_{0}+\varepsilon^{1-a})\varepsilon^{\frac{1}{2}-a}.
\end{align*}	
Collecting the above estimates on $J_6$, we conclude that	
\begin{align*}
\frac{1}{\varepsilon}|(\partial^{\alpha}\Gamma(f,f),\frac{\partial^{\alpha}F}{\sqrt{\mu}})|
\leq C\varepsilon^{\frac{1}{2}-a}\frac{1}{\varepsilon^{2}}\mathcal{D}_{N}(t)+C(\eta_{0}+\varepsilon^{1-a})\varepsilon^{\frac{1}{2}-a}.
\end{align*}
Therefore, we can get the desired estimate \eqref{5.47} by substituting the above estimates into \eqref{5.48}. This completes the proof of Lemma \ref{lem5.11}.
\end{proof}

\subsection{Estimates on electromagnetic field terms}\label{seca.5.4}
To make the energy estimates for the equations \eqref{4.11} and \eqref{4.12}, one has to deal with
the complicated  nonlinear  terms involving with the electromagnetic fields.
\begin{lemma}\label{lem5.12}
Assume \eqref{3.3} and \eqref{4.13} hold. For $|\alpha|=N$, it holds that
\begin{align}
\label{5.49}
&(\frac{\partial^{\alpha}[(E+v\times B)\cdot\nabla_{v}F]}{\sqrt{\mu}},\frac{\partial^{\alpha}F}{\sqrt{\mu}})
\nonumber\\
&\leq -\frac{1}{2}\frac{d}{dt}(\frac{1}{R\theta}\partial^{\alpha}\widetilde{E},\partial^{\alpha}\widetilde{E})
-\frac{1}{2}\frac{d}{dt}(\frac{1}{R\theta}\partial^{\alpha}\widetilde{B},\partial^{\alpha}\widetilde{B})
\nonumber\\
&\quad\quad+C(\eta_{0}+\varepsilon^{\frac{1}{2}-a})\frac{1}{\varepsilon^2}\mathcal{D}_{N}(t)+C[\eta_{0}(1+t)^{-\vartheta}
+\eta_{0}\varepsilon^{2a}+\varepsilon^{\frac{1}{2}-a}]\varepsilon^{-2a}.	
\end{align}
\end{lemma}
\begin{proof}
For $|\alpha|=N$, we denote
\begin{align}
\label{5.50}
(\frac{\partial^{\alpha}[(E+v\times B)\cdot\nabla_{v}F]}{\sqrt{\mu}},\frac{\partial^{\alpha}F}{\sqrt{\mu}})
:=I_4+I_5+I_6,
\end{align}
where $I_4$, $I_5$ and $I_6$ are given by
\begin{align}
\label{5.51}
\left\{
\begin{array}{rl}
I_4=&(\frac{(E+v\times B)\cdot\nabla_{v}\partial^{\alpha}F}{\sqrt{\mu}},\frac{\partial^{\alpha}F}{\sqrt{\mu}}),
\\
I_5=&\sum_{1\leq \alpha_{1}\leq \alpha}C^{\alpha_1}_\alpha(\frac{\partial^{\alpha_1}E\cdot\nabla_{v}\partial^{\alpha-\alpha_{1}}F}{\sqrt{\mu}},
\frac{\partial^{\alpha}F}{\sqrt{\mu}}),
\\
I_6=&\sum_{1\leq \alpha_{1}\leq \alpha}C^{\alpha_1}_\alpha(\frac{(v\times \partial^{\alpha_1}B)\cdot\nabla_{v}\partial^{\alpha-\alpha_{1}}F}{\sqrt{\mu}},
\frac{\partial^{\alpha}F}{\sqrt{\mu}}).
\end{array} \right.
\end{align}
We compute each $I_i$ $(i=4,5,6)$ in the following way. 
Since $\nabla_{v}\cdot(E+v\times B)=0$ and $v\cdot(v\times B)=0$, we get by integration by parts that
\begin{align*}
|I_4|
=|\int_{{\mathbb R}^3}\int_{{\mathbb R}^3}(E+v\times B)\cdot v\frac{(\partial^{\alpha}F)^2}{2\mu}\,dv\,dx|\leq C\|E\|_{L^{\infty}}\|\langle v\rangle^{\frac{1}{2}}\frac{\partial^{\alpha}F}{\sqrt{\mu}}\|^2,
\end{align*}
which together with \eqref{4.49}, \eqref{4.22a} and the following simple fact
\begin{align}
\label{5.52}
\|E\|_{L^{\infty}}\leq \|\bar{E}\|_{L^{\infty}}+\|\widetilde{E}\|_{L^{\infty}}
\leq C[\eta_{0}(1+t)^{-\vartheta}+\varepsilon^{1-a}],
\end{align}
yields that
\begin{align*}
|I_4|\leq C(\eta_{0}+\varepsilon^{1-a})\frac{1}{\varepsilon}\mathcal{D}_{N}(t)+C[\eta_{0}(1+t)^{-\vartheta}+\varepsilon^{1-a}].
\end{align*}
In order to estimate $I_5$, we write $I_5=\sum_{1\leq \alpha_{1}\leq \alpha}C^{\alpha_1}_\alpha(I^1_5+I^2_5)$ with
\begin{align}
\label{4.66}
I^1_5=(\frac{\partial^{\alpha_1}\widetilde{E}\cdot\partial^{\alpha-\alpha_{1}}\nabla_{v}F}{\sqrt{\mu}},\frac{\partial^{\alpha}F}{\sqrt{\mu}}),
\quad I^2_5=(\frac{\partial^{\alpha_1}\bar{E}\cdot\partial^{\alpha-\alpha_{1}}\nabla_{v}F}{\sqrt{\mu}},\frac{\partial^{\alpha}F}{\sqrt{\mu}}).
\end{align}
For $1\leq \alpha_{1}\leq \alpha$, we easily see that the term $I^1_5$ is equivalent to
\begin{align}
\label{5.55}
&(\frac{\partial^{\alpha_1}\widetilde{E}\cdot\partial^{\alpha-\alpha_{1}}\nabla_{v}M}{\sqrt{\mu}},\frac{\partial^{\alpha}M}{\sqrt{\mu}})
+(\frac{\partial^{\alpha_1}\widetilde{E}\cdot\partial^{\alpha-\alpha_{1}}\nabla_{v}M}{\sqrt{\mu}},\frac{\partial^{\alpha}(\overline{G}+\sqrt{\mu}f)}{\sqrt{\mu}})
\nonumber\\
+&(\frac{\partial^{\alpha_1}\widetilde{E}\cdot\partial^{\alpha-\alpha_{1}}\nabla_{v}\overline{G}}{\sqrt{\mu}},\frac{\partial^{\alpha}F}{\sqrt{\mu}})
+(\partial^{\alpha_1}\widetilde{E}\cdot\partial^{\alpha-\alpha_{1}}\nabla_{v}f,\frac{\partial^{\alpha}F}{\sqrt{\mu}})
-(\frac{v}{2}\partial^{\alpha_1}\widetilde{E}\cdot\partial^{\alpha-\alpha_{1}}f,\frac{\partial^{\alpha}F}{\sqrt{\mu}}),
\end{align}
where we used  $F=M+\overline{G}+\sqrt{\mu}f$ and
\begin{align}
\label{5.54}
\partial^{\alpha-\alpha_{1}}\nabla_{v}F=\partial^{\alpha-\alpha_{1}}\nabla_{v}M+\partial^{\alpha-\alpha_{1}}\nabla_{v}\overline{G}
+\sqrt{\mu}\partial^{\alpha-\alpha_{1}}\nabla_{v}f-\frac{v}{2}\sqrt{\mu}\partial^{\alpha-\alpha_{1}}f.
\end{align}
To present the calculations for the first term of \eqref{5.55}, we first consider the case $\alpha_{1}=\alpha$ and split it into two parts as follows
\begin{align}
\label{5.56}
(\frac{\partial^{\alpha_1}\widetilde{E}\cdot\partial^{\alpha-\alpha_{1}}\nabla_{v}M}{\sqrt{\mu}},\frac{\partial^{\alpha}M}{\sqrt{\mu}})
=(\partial^{\alpha}\widetilde{E}\cdot\nabla_{v}M,[\frac{1}{\mu}-\frac{1}{M}]\partial^{\alpha}M)
+(\partial^{\alpha}\widetilde{E}\cdot\nabla_{v}M,\frac{1}{M}\partial^{\alpha}M).
\end{align}
We begin to compute the second term on the RHS of \eqref{5.56}, since the first term  is more easier and 
is thus left to the end. In view of \eqref{2.3} and \eqref{2.1}, a direct computation shows that
\begin{align}
\label{5.57}
(\partial^{\alpha}\widetilde{E}\cdot\nabla_{v}M,\frac{1}{M}\partial^{\alpha}M)&=
-\int_{{\mathbb R}^3}\int_{{\mathbb R}^3}\partial^{\alpha}\widetilde{E}\cdot\frac{v-u}{R\theta}\partial^{\alpha}M\,dv\,dx
\nonumber\\
&=-\int_{{\mathbb R}^3}\frac{1}{R\theta}\partial^{\alpha}\widetilde{E}\cdot[\partial^{\alpha}(\rho u)-u\partial^{\alpha}\rho]\,dx
\nonumber\\
&=-(\frac{1}{R\theta}\partial^{\alpha}\widetilde{E},\partial^{\alpha}(\rho u-\bar{\rho}\bar{u}))
-(\frac{1}{R\theta}\partial^{\alpha}\widetilde{E},\partial^{\alpha}(\bar{\rho}\bar{u})-u\partial^{\alpha}\rho).
\end{align}
Since $\partial_{t}\widetilde{E}-\nabla_{x}\times \widetilde{B}=\rho u-\bar{\rho}\bar{u}$ in \eqref{4.3}, the following identities hold true
\begin{align}
\label{5.58}
-&(\frac{1}{R\theta}\partial^{\alpha}\widetilde{E},\partial^{\alpha}(\rho u-\bar{\rho}\bar{u}))
\nonumber\\
=&-(\frac{1}{R\theta}\partial^{\alpha}\widetilde{E},\partial_{t}\partial^{\alpha}\widetilde{E}-\nabla_{x}\times \partial^{\alpha}\widetilde{B})	
\nonumber\\
=&-\frac{1}{2}\frac{d}{dt}(\frac{1}{R\theta}\partial^{\alpha}\widetilde{E},\partial^{\alpha}\widetilde{E})
-\frac{1}{2}\frac{d}{dt}(\frac{1}{R\theta}\partial^{\alpha}\widetilde{B},\partial^{\alpha}\widetilde{B})
\nonumber\\
&+\frac{1}{2}(\partial_{t}[\frac{1}{R\theta}]\partial^{\alpha}\widetilde{E},\partial^{\alpha}\widetilde{E})
+\frac{1}{2}(\partial_{t}[\frac{1}{R\theta}]\partial^{\alpha}\widetilde{B},\partial^{\alpha}\widetilde{B})
+(\nabla_{x}[\frac{1}{R\theta}]\times \partial^{\alpha}\widetilde{E},\partial^{\alpha}\widetilde{B}),
\end{align}
where in the last identity we have used 
\begin{align*}
&(\frac{1}{R\theta}\partial^{\alpha}\widetilde{E},\nabla_{x}\times \partial^{\alpha}\widetilde{B})
=(\nabla_{x}\times [\frac{1}{R\theta}\partial^{\alpha}\widetilde{E}],\partial^{\alpha}\widetilde{B})
\nonumber\\
&=(\frac{1}{R\theta}\nabla_{x}\times \partial^{\alpha}\widetilde{E},\partial^{\alpha}\widetilde{B})
+(\nabla_{x}[\frac{1}{R\theta}]\times \partial^{\alpha}\widetilde{E},\partial^{\alpha}\widetilde{B})
\nonumber\\
&=-(\frac{1}{R\theta}\partial_{t}\partial^{\alpha}\widetilde{B},\partial^{\alpha}\widetilde{B})
+(\nabla_{x}[\frac{1}{R\theta}]\times \partial^{\alpha}\widetilde{E},\partial^{\alpha}\widetilde{B})
\nonumber\\
&=-\frac{1}{2}\frac{d}{dt}(\frac{1}{R\theta}\partial^{\alpha}\widetilde{B},\partial^{\alpha}\widetilde{B})
+\frac{1}{2}(\partial_{t}[\frac{1}{R\theta}]\partial^{\alpha}\widetilde{B},\partial^{\alpha}\widetilde{B})
+(\nabla_{x}[\frac{1}{R\theta}]\times \partial^{\alpha}\widetilde{E},\partial^{\alpha}\widetilde{B}),
\end{align*}
due to the fact $\partial_{t}\widetilde{B}+\nabla_{x}\times \widetilde{E}=0$ in \eqref{4.3}. 
By \eqref{3.3}, \eqref{3.1}, \eqref{4.13A} and Lemma \ref{lem4.18}, we get
\begin{align*}
|(\partial_{t}[\frac{1}{R\theta}]\partial^{\alpha}\widetilde{E},\partial^{\alpha}\widetilde{E})|
&\leq C(\|\partial_{t}\bar{\theta}\|_{L^{\infty}}+\|\partial_{t}\widetilde{\theta}\|_{L^{\infty}})\|\partial^{\alpha}\widetilde{E}\|^2
\nonumber\\
&\leq 
C[\eta_{0}(1+t)^{-\vartheta}+\varepsilon^{\frac{1}{2}-a}]\varepsilon^{-2a}.
\end{align*}
The last two terms of \eqref{5.58} can be treated similarly. It follows from \eqref{5.58} that
\begin{align}
\label{5.59}
&-(\frac{1}{R\theta}\partial^{\alpha}\widetilde{E},\partial^{\alpha}(\rho u-\bar{\rho}\bar{u}))
\nonumber\\
&\leq -\frac{1}{2}\frac{d}{dt}(\frac{1}{R\theta}\partial^{\alpha}\widetilde{E},\partial^{\alpha}\widetilde{E})
-\frac{1}{2}\frac{d}{dt}(\frac{1}{R\theta}\partial^{\alpha}\widetilde{B},\partial^{\alpha}\widetilde{B})
+C[\eta_{0}(1+t)^{-\vartheta}+\varepsilon^{\frac{1}{2}-a}]\varepsilon^{-2a}.
\end{align}
As before, we denote $\partial^{\alpha}=\partial^{\alpha'}\partial_{x_i}$ with $|\alpha'|=N-1$ due to $|\alpha|=N$. By integration by parts,
\eqref{3.3} and \eqref{4.13}, one gets
\begin{align}
\label{5.60}	
|(\frac{1}{R\theta}\partial^{\alpha}\widetilde{E},\partial^{\alpha}(\bar{\rho}\bar{u}))|
&=|(\partial^{\alpha'}\widetilde{E},\partial_{x_i}[\frac{1}{R\theta}\partial^{\alpha}(\bar{\rho}\bar{u})])|
\nonumber\\
&\leq C\|\partial^{\alpha'}\widetilde{E}\|\|\partial_{x_i}[\frac{1}{R\theta}\partial^{\alpha}(\bar{\rho}\bar{u})]\|
\leq C\eta_{0}\varepsilon^{1-a}.
\end{align}
The rest term in \eqref{5.57} can be estimated as
\begin{align*}	
|(\frac{1}{R\theta}\partial^{\alpha}\widetilde{E},u\partial^{\alpha}\rho)|
&\leq C\|u\|_{L^{\infty}}\|\partial^{\alpha}\widetilde{E}\|\|\partial^{\alpha}\rho\|
\\
&\leq  C[\eta_{0}(1+t)^{-\vartheta}+\varepsilon^{1-a}](\varepsilon^{-2a}+\eta_{0}).
\end{align*}
Hence, substituting this, \eqref{5.60} and \eqref{5.59} into \eqref{5.57}, one gets
\begin{align}
\label{5.61}
(\partial^{\alpha}\widetilde{E}\cdot\nabla_{v}M,\frac{1}{M}\partial^{\alpha}M)
\leq& -\frac{1}{2}\frac{d}{dt}(\frac{1}{R\theta}\partial^{\alpha}\widetilde{E},\partial^{\alpha}\widetilde{E})
-\frac{1}{2}\frac{d}{dt}(\frac{1}{R\theta}\partial^{\alpha}\widetilde{B},\partial^{\alpha}\widetilde{B})
\nonumber\\
&+C[\eta_{0}(1+t)^{-\vartheta}+\varepsilon^{\frac{1}{2}-a}]\varepsilon^{-2a}.
\end{align}
Recall $\partial^{\alpha}M=I_1+I_2+I_3$ given in \eqref{5.30}, we easily see that the first term on the RHS of \eqref{5.56} is equivalent to
\begin{align}
\label{5.62}
(\partial^{\alpha}\widetilde{E}\cdot\nabla_{v}M,[\frac{1}{\mu}-\frac{1}{M}]I_1)
+(\partial^{\alpha}\widetilde{E}\cdot\nabla_{v}M,[\frac{1}{\mu}-\frac{1}{M}]I_2)
+(\partial^{\alpha}\widetilde{E}\cdot\nabla_{v}M,[\frac{1}{\mu}-\frac{1}{M}]I_3).
\end{align}
Recall $I_1$ given in \eqref{5.30}, we use \eqref{5.22}, \eqref{4.13}  and \eqref{4.15} to get
\begin{align*}
|(\partial^{\alpha}\widetilde{E}\cdot\nabla_{v}M,[\frac{1}{\mu}-\frac{1}{M}]I_1)|
&\leq C(\eta_{0}+\varepsilon^{1-a})\|\partial^{\alpha}\widetilde{E}\|\|\partial^{\alpha}(\widetilde{\rho},\widetilde{u},\widetilde{\theta})\|
\\
&\leq C(\eta_{0}+\varepsilon^{1-a})(\frac{1}{\varepsilon}\|\partial^{\alpha}(\widetilde{\rho},\widetilde{u},\widetilde{\theta})\|^2
+\varepsilon\|\partial^{\alpha}\widetilde{E}\|^2)
\\
&\leq C(\eta_{0}+\varepsilon^{1-a})(\frac{1}{\varepsilon^2}\mathcal{D}_{N}(t)+\varepsilon^{1-2a}).
\end{align*}
Recall $I_2$ given in \eqref{5.30}, we use the similar arguments as \eqref{5.60} to get
\begin{align*}
|(\partial^{\alpha}\widetilde{E}\cdot\nabla_{v}M,[\frac{1}{\mu}-\frac{1}{M}]I_2)|
\leq |(\partial^{\alpha'}\widetilde{E},\partial_{x_i}\{\nabla_{v}M[\frac{1}{\mu}-\frac{1}{M}]I_2\})|
\leq C(\eta_{0}+\varepsilon^{1-a})\varepsilon^{1-a}.
\end{align*}
The last term of \eqref{5.62} can be handed similarly. We thereby obtain
\begin{align*}
|(\partial^{\alpha}\widetilde{E}\cdot\nabla_{v}M,[\frac{1}{\mu}-\frac{1}{M}]\partial^{\alpha}M)|
\leq C(\eta_{0}+\varepsilon^{1-a})\frac{1}{\varepsilon^2}\mathcal{D}_{N}(t)+C(\eta_{0}+\varepsilon^{1-a})\varepsilon^{1-2a}.
\end{align*}
For the case $\alpha_{1}=\alpha$, we substitute this and \eqref{5.61} into \eqref{5.56} to obtain
\begin{align}
\label{5.63}
(\frac{\partial^{\alpha_1}\widetilde{E}\cdot\partial^{\alpha-\alpha_{1}}\nabla_{v}M}{\sqrt{\mu}},\frac{\partial^{\alpha}M}{\sqrt{\mu}})
\leq& -\frac{1}{2}\frac{d}{dt}(\frac{1}{R\theta}\partial^{\alpha}\widetilde{E},\partial^{\alpha}\widetilde{E})
-\frac{1}{2}\frac{d}{dt}(\frac{1}{R\theta}\partial^{\alpha}\widetilde{B},\partial^{\alpha}\widetilde{B})
\nonumber\\
+C(\eta_{0}&+\varepsilon^{1-a})\frac{1}{\varepsilon^2}\mathcal{D}_{N}(t)+C[\eta_{0}(1+t)^{-\vartheta}+\varepsilon^{\frac{1}{2}-a}]\varepsilon^{-2a},	
\end{align}
where we have used \eqref{4.22a} such that
$$
(\eta_{0}+\varepsilon^{1-a})\varepsilon^{1-2a}\leq\varepsilon^{\frac{1}{2}-a}\varepsilon^{-2a}.
$$
For the case $1\leq\alpha_{1}<\alpha$, we get by \eqref{4.13}, \eqref{4.22a} and \eqref{4.15}
that
\begin{align}
\label{5.64}
|(\frac{\partial^{\alpha_1}\widetilde{E}\cdot\partial^{\alpha-\alpha_{1}}\nabla_{v}M}{\sqrt{\mu}},\frac{\partial^{\alpha}M}{\sqrt{\mu}})|
&\leq C\|\partial^{\alpha_1}\widetilde{E}\|_{L^3}\||\frac{\partial^{\alpha-\alpha_{1}}\nabla_{v}M}{\sqrt{\mu}}|_2\|_{L^6}
\||\frac{\partial^{\alpha}M}{\sqrt{\mu}}|_2\|_{L^2}
\nonumber\\
&\leq C\varepsilon^{\frac{1}{2}-a}(\|\frac{\nabla_x\partial^{\alpha-\alpha_{1}}\nabla_{v}M}{\sqrt{\mu}}\|^2+
\|\frac{\partial^{\alpha}M}{\sqrt{\mu}}\|^2)
\nonumber\\
&\leq C\varepsilon^{\frac{1}{2}-a}\frac{1}{\varepsilon^2}\mathcal{D}_{N}(t)
+ C(\eta_{0}+\varepsilon^{1-a})\varepsilon^{\frac{1}{2}-a},
\end{align}
where in the last inequality we have applied the same argument as in the first line of \eqref{4.47a}. 
Then the above estimate combined with \eqref{5.63} immediately gives
\begin{align}
\label{5.65}
&\sum_{1\leq \alpha_{1}\leq \alpha}C^{\alpha_1}_\alpha(\frac{\partial^{\alpha_1}\widetilde{E}\cdot\partial^{\alpha-\alpha_{1}}\nabla_{v}M}{\sqrt{\mu}},\frac{\partial^{\alpha}M}{\sqrt{\mu}})
\nonumber\\
&\leq -\frac{1}{2}\frac{d}{dt}(\frac{1}{R\theta}\partial^{\alpha}\widetilde{E},\partial^{\alpha}\widetilde{E})
-\frac{1}{2}\frac{d}{dt}(\frac{1}{R\theta}\partial^{\alpha}\widetilde{B},\partial^{\alpha}\widetilde{B})
\nonumber\\
&\quad\quad+C(\eta_{0}+\varepsilon^{\frac{1}{2}-a})\frac{1}{\varepsilon^2}\mathcal{D}_{N}(t)+C[\eta_{0}(1+t)^{-\vartheta}+\varepsilon^{\frac{1}{2}-a}]\varepsilon^{-2a}.	
\end{align}
For the second  term in \eqref{5.55}. If $\alpha/2<\alpha_{1}\leq\alpha$, {then
by \eqref{5.38}, \eqref{4.22a}, \eqref{4.20} and \eqref{4.15}, we get}
\begin{align*}
&|(\frac{\partial^{\alpha_1}\widetilde{E}\cdot\partial^{\alpha-\alpha_{1}}\nabla_{v}M}{\sqrt{\mu}},\frac{\partial^{\alpha}(\overline{G}+\sqrt{\mu}f)}{\sqrt{\mu}})|
\\
&\leq C\varepsilon^{\frac{1}{2}+a}\|\partial^{\alpha_1}\widetilde{E}\|^2+C\varepsilon^{-\frac{1}{2}-a}\|\frac{\partial^{\alpha}(\overline{G}+\sqrt{\mu}f)}{\sqrt{\mu}}\|^2
\\
&\leq C\varepsilon^{\frac{1}{2}-a}\frac{1}{\varepsilon^2}\mathcal{D}_{N}(t)+C\varepsilon^{\frac{1}{2}-a}.
\end{align*}
The case $1\leq\alpha_{1}\leq\alpha/2$ can be treated similarly.
It follows  that
\begin{align}
\label{5.66}
\sum_{1\leq \alpha_{1}\leq \alpha}|(\frac{\partial^{\alpha_1}\widetilde{E}\cdot\partial^{\alpha-\alpha_{1}}\nabla_{v}M}{\sqrt{\mu}},\frac{\partial^{\alpha}(\overline{G}+\sqrt{\mu}f)}{\sqrt{\mu}})|
\leq C\varepsilon^{\frac{1}{2}-a}\frac{1}{\varepsilon^2}\mathcal{D}_{N}(t)+C\varepsilon^{\frac{1}{2}-a}.
\end{align}
The third  term of \eqref{5.55} has the same bound as \eqref{5.66}.
For the fourth  term in \eqref{5.55}. If  $1\leq\alpha_{1}\leq \alpha/2$, then by \eqref{4.13},  \eqref{4.22a}, \eqref{4.49} and \eqref{4.15}, we have
\begin{align*}
|(\partial^{\alpha_1}\widetilde{E}\cdot\partial^{\alpha-\alpha_{1}}\nabla_{v}f,\frac{\partial^{\alpha}F}{\sqrt{\mu}})|
&\leq C\|\partial^{\alpha_1}\widetilde{E}\|_{L^\infty}(\|\partial^{\alpha-\alpha_{1}}\nabla_{v}f\|^2+
\|\frac{\partial^{\alpha}F}{\sqrt{\mu}}\|^2)
\\
&\leq C\varepsilon^{\frac{1}{2}-a}\frac{1}{\varepsilon^2}\mathcal{D}_{N}(t)+C\varepsilon^{\frac{1}{2}-a}.
\end{align*}
If  $|\alpha|/2<\alpha_{1}\leq \alpha-1$, then we take $L^2-L^3-L^6$ H\"{o}lder inequality to get
\begin{align*}
|(\partial^{\alpha_1}\widetilde{E}\cdot\partial^{\alpha-\alpha_{1}}\nabla_{v}f,\frac{\partial^{\alpha}F}{\sqrt{\mu}})|
&\leq C\|\partial^{\alpha_1}\widetilde{E}\|_{L^3}\||\partial^{\alpha-\alpha_{1}}\nabla_{v}f|_2\|_{L^6}
\|\frac{\partial^{\alpha}F}{\sqrt{\mu}}\|
\\
&\leq C\varepsilon^{\frac{1}{2}-a}\frac{1}{\varepsilon^2}\mathcal{D}_{N}(t)+C\varepsilon^{\frac{1}{2}-a}.
\end{align*}
The rest case  $\alpha_{1}=\alpha$ has the same bound as above. So  one gets
\begin{align}
\label{5.67}
\sum_{1\leq \alpha_{1}\leq \alpha}|(\partial^{\alpha_1}\widetilde{E}\cdot\partial^{\alpha-\alpha_{1}}\nabla_{v}f,\frac{\partial^{\alpha}F}{\sqrt{\mu}})|
\leq C\varepsilon^{\frac{1}{2}-a}\frac{1}{\varepsilon^2}\mathcal{D}_{N}(t)+C\varepsilon^{\frac{1}{2}-a}.
\end{align}
The  last term of \eqref{5.55} can be treated in the same way as \eqref{5.67}.
All in all, collecting the estimates above, we can conclude from \eqref{5.55} that
\begin{align}
\label{4.79a}
\sum_{1\leq \alpha_{1}\leq \alpha}C^{\alpha_1}_\alpha I^1_5\leq& -\frac{1}{2}\frac{d}{dt}(\frac{1}{R\theta}\partial^{\alpha}\widetilde{E},\partial^{\alpha}\widetilde{E})
-\frac{1}{2}\frac{d}{dt}(\frac{1}{R\theta}\partial^{\alpha}\widetilde{B},\partial^{\alpha}\widetilde{B})
\nonumber\\
&+C(\eta_{0}+\varepsilon^{\frac{1}{2}-a})\frac{1}{\varepsilon^2}\mathcal{D}_{N}(t)+C[\eta_{0}(1+t)^{-\vartheta}+\varepsilon^{\frac{1}{2}-a}]\varepsilon^{-2a}.		
\end{align}
Next we compute $I^2_5$ given in \eqref{4.66}, we write $I^2_5=I^{21}_5+I^{22}_5$ with
$$
I^{21}_5=(\frac{\partial^{\alpha_1}\bar{E}\cdot\partial^{\alpha-\alpha_{1}}\nabla_{v}M}{\sqrt{\mu}},\frac{\partial^{\alpha}F}{\sqrt{\mu}}),\quad
I^{22}_5=(\frac{\partial^{\alpha_1}\bar{E}\cdot\partial^{\alpha-\alpha_{1}}\nabla_{v}(\overline{G}+\sqrt{\mu}f)}{\sqrt{\mu}},\frac{\partial^{\alpha}F}{\sqrt{\mu}}).
$$
To bound $I^{21}_5$, if $\alpha_{1}=\alpha$, we employ \eqref{4.25}, \eqref{4.49}, \eqref{3.3} and \eqref{4.22a} to get
\begin{align*}
|I^{21}_5|\leq C\|\partial^{\alpha}\bar{E}\|\|\frac{\partial^{\alpha}F}{\sqrt{\mu}}\|
\leq C\eta_{0}\frac{1}{\varepsilon}\mathcal{D}_{N}(t)+C\eta_{0}.
\end{align*}
If $1\leq \alpha_{1}\leq\alpha-1$, making use of  \eqref{4.25},  \eqref{3.3}, \eqref{4.49} and \eqref{4.22a} again, one gets
\begin{align*}
|I^{21}_5|\leq C\|\partial^{\alpha_1}\bar{E}\|_{L^{\infty}}\|\frac{\partial^{\alpha-\alpha_{1}}\nabla_{v}M}{\sqrt{\mu}}\|
\|\frac{\partial^{\alpha}F}{\sqrt{\mu}}\|
\leq C\eta_{0}\frac{1}{\varepsilon}\mathcal{D}_{N}(t)+C\eta_{0}.
\end{align*}
It follows from the aforementioned two estimates that
\begin{align*}
|I^{21}_5|\leq C\eta_{0}\frac{1}{\varepsilon}\mathcal{D}_{N}(t)+C\eta_{0}.
\end{align*}
Similar estimate also holds for $I^{22}_5$. Therefore, we conclude that
\begin{align}
\label{5.68}
\sum_{1\leq \alpha_{1}\leq \alpha}C^{\alpha_1}_\alpha|I^2_5|
\leq C\eta_{0}\frac{1}{\varepsilon}\mathcal{D}_{N}(t)+C\eta_{0}.
\end{align}
Thanks to $I_5=\sum_{1\leq \alpha_{1}\leq \alpha}C^{\alpha_1}_\alpha(I^1_5+I^2_5)$, we deduce from  \eqref{4.79a} and \eqref{5.68} that
\begin{align*}
I_5\leq& -\frac{1}{2}\frac{d}{dt}(\frac{1}{R\theta}\partial^{\alpha}\widetilde{E},\partial^{\alpha}\widetilde{E})
-\frac{1}{2}\frac{d}{dt}(\frac{1}{R\theta}\partial^{\alpha}\widetilde{B},\partial^{\alpha}\widetilde{B})
\nonumber\\
&+C(\eta_{0}+\varepsilon^{\frac{1}{2}-a})\frac{1}{\varepsilon^2}\mathcal{D}_{N}(t)+C[\eta_{0}(1+t)^{-\vartheta}+\eta_{0}\varepsilon^{2a}+\varepsilon^{\frac{1}{2}-a}]\varepsilon^{-2a}.	
\end{align*}
It remains to estimate $I_6$ given in \eqref{5.51}.
By  the similar calculations as \eqref{5.67} and \eqref{5.68}, we can arrive at
\begin{align*}
I_6=\sum_{1\leq \alpha_{1}\leq \alpha}	
C^{\alpha_1}_\alpha&\big\{(\frac{(v\times \partial^{\alpha_1}\widetilde{B})\cdot\partial^{\alpha-\alpha_{1}}\nabla_{v}
			M}{\sqrt{\mu}},\frac{\partial^{\alpha}F}{\sqrt{\mu}})
		\\
&+(\frac{(v\times \partial^{\alpha_1}\widetilde{B})\cdot\partial^{\alpha-\alpha_{1}}
			(\nabla_{v}\overline{G}+\sqrt{\mu}\nabla_{v}f)}{\sqrt{\mu}},\frac{\partial^{\alpha}F}{\sqrt{\mu}})
		\\
	&+(\frac{(v\times \partial^{\alpha_1}\bar{B})\cdot\partial^{\alpha-\alpha_{1}}
			(\nabla_{v}M+\nabla_{v}\overline{G}+\sqrt{\mu}\nabla_{v}f)}{\sqrt{\mu}},
		\frac{\partial^{\alpha}F}{\sqrt{\mu}})\big\}
\\
\leq C(\eta_{0}+\varepsilon^{\frac{1}{2}-a})&\frac{1}{\varepsilon^2}\mathcal{D}_{N}(t)
+C[\eta_{0}(1+t)^{-\vartheta}+\eta_{0}\varepsilon^{2a}+\varepsilon^{\frac{1}{2}-a}]\varepsilon^{-2a}.
\end{align*}

In summary, we can obtain \eqref{5.49} by substituting the estimates on $I_4$, $I_5$ and $I_6$ into \eqref{5.50}. This completes
the proof of Lemma \ref{lem5.12}.
\end{proof}
\begin{lemma}\label{lem5.13}
For  $|\alpha|+|\beta|\leq N$ and $|\beta|\geq 1$, it holds that
\begin{align}
\label{5.69}
|(\partial_{\beta}^{\alpha}[\frac{(E+v\times B)\cdot\nabla_{v}(\sqrt{\mu}f)}{\sqrt{\mu}}],\partial_{\beta}^{\alpha}f)|
\leq C(\eta_{0}+\varepsilon^{\frac{1}{2}-a})\mathcal{D}_{N}(t).
\end{align}	
\end{lemma}
\begin{proof}
Let $|\alpha|+|\beta|\leq N$ and $|\beta|\geq 1$, it is straightforward to get that
\begin{align}
\label{5.70}
&(\partial_{\beta}^{\alpha}[\frac{(E+v\times B)\cdot\nabla_{v}(\sqrt{\mu}f)}{\sqrt{\mu}}],\partial_{\beta}^{\alpha}f)
\nonumber\\
&=(\partial^\alpha_\beta[E\cdot\nabla_{v}f],\partial_{\beta}^\alpha f)-(\partial^\alpha_\beta[\frac{v}{2}\cdot E f],\partial_{\beta}^\alpha f)
+(\partial^\alpha_\beta[v\times B\cdot\nabla_{v}f],\partial_{\beta}^\alpha f),
\end{align}
where we have used  $v\cdot(v\times B)=0$. The first term on the RHS of \eqref{5.70} is equivalent to
\begin{align*}
(E\cdot\nabla_{v}\partial^\alpha_\beta f,\partial_{\beta}^\alpha f)
+\sum_{1\leq \alpha_{1}\leq\alpha}C^{\alpha_{1}}_\alpha 
(\partial^{\alpha_{1}}E\cdot\nabla_{v}\partial^{\alpha-\alpha_{1}}_{\beta} f,\partial_{\beta}^\alpha f).
\end{align*}
Here if $|\alpha|=0$, the last term vanishes and if  $|\alpha|\geq 1$, the last term exists. 
Note that $|\alpha|\leq N-1$ since we only consider $|\alpha|+|\beta|\leq N$ and $|\beta|\geq 1$.
If $1\leq|\alpha_1|\leq N/2$, we use \eqref{3.3}, \eqref{4.13} and \eqref{4.15} to get
\begin{align*}
|(\partial^{\alpha_{1}}E\cdot\nabla_{v}\partial^{\alpha-\alpha_{1}}_{\beta} f,\partial_{\beta}^\alpha f)|
&\leq C\|\partial^{\alpha_{1}}E\|_{L^{\infty}}\|\nabla_{v}\partial^{\alpha-\alpha_{1}}_\beta f\|
\|\partial_{\beta}^\alpha f\|
\nonumber\\
&\leq C(\eta_{0}+\varepsilon^{\frac{1}{2}-a})\mathcal{D}_{N}(t).
\end{align*}
If $ N/2<|\alpha_1|\leq N-1$, it holds that
\begin{align*}
|(\partial^{\alpha_{1}}E\cdot\nabla_{v}\partial^{\alpha-\alpha_{1}}_{\beta} f,\partial_{\beta}^\alpha f)|
&\leq C\|\partial^{\alpha_{1}}E\|_{L^{3}}\||\nabla_{v}\partial^{\alpha-\alpha_{1}}_\beta f|_2\|_{L^{6}}
\|\partial_{\beta}^\alpha f\|
\\
&\leq C(\eta_{0}+\varepsilon^{\frac{1}{2}-a})\mathcal{D}_{N}(t).
\end{align*}
We thus deduce from the aforementioned two estimates that
\begin{align*}
\sum_{1\leq \alpha_{1}\leq\alpha}C^{\alpha_{1}}_\alpha|(\partial^{\alpha_{1}}E\cdot\nabla_{v}\partial^{\alpha-\alpha_{1}}_{\beta} f,\partial_{\beta}^\alpha f)|
\leq C(\eta_{0}+\varepsilon^{\frac{1}{2}-a})\mathcal{D}_{N}(t),
\end{align*}
which, together with the identity $(E\cdot\nabla_{v}\partial^\alpha_\beta f,\partial_{\beta}^\alpha f)=0$, yields that
\begin{align}
\label{5.71}
|(\partial^\alpha_\beta[E\cdot\nabla_{v}f],\partial_{\beta}^\alpha f)|\leq C(\eta_{0}+\varepsilon^{\frac{1}{2}-a})\mathcal{D}_{N}(t).
\end{align}
The last two terms on the RHS of \eqref{5.70} has the same bound as \eqref{5.71}.
%\begin{align*}
%|(\partial^\alpha_\beta[\frac{v}{2}\cdot E f],\partial_{\beta}^\alpha f)|
%+|(\partial^\alpha_\beta[v\times B\cdot\nabla_{v}f],\partial_{\beta}^\alpha f)|\leq C(\eta_{0}+\varepsilon^{\frac{1}{2}-a})\mathcal{D}_{N}(t).
%\end{align*}
Therefore, the desired estimate \eqref{5.69} follows. This ends the proof of Lemma \ref{lem5.13}.
\end{proof}
Similar arguments as \eqref{5.69} lead to the following estimate.
\begin{lemma}\label{lem5.14}
For $|\alpha|\leq N-1$, it holds that
\begin{align}
\label{5.72}
|(\frac{\partial^{\alpha}[(E+v\times B)\cdot\nabla_{v}(\sqrt{\mu}f)]}{\sqrt{\mu}},\partial^{\alpha}f)|
\leq C(\eta_{0}+\varepsilon^{\frac{1}{2}-a})\mathcal{D}_{N}(t).
\end{align}
\end{lemma}
\subsection{Estimates on fluid quantities}\label{seca.5.5} 
Now we are devoted to deriving  some estimates on the fluid quantities.
\begin{lemma}\label{lem4.17}
 For $|\alpha|\leq N-1$, one has
\begin{align}	
\label{4.86a}
&\varepsilon \frac{d}{dt}(\partial^{\alpha}\widetilde{u},\nabla_x\partial^{\alpha}\widetilde{\rho})
+c\varepsilon\{\|\nabla_x\partial^{\alpha}\widetilde{\rho}\|^2+\|\nabla_{x}\cdot \partial^{\alpha}\widetilde{E}\|^2\}
\nonumber\\
&\leq C\varepsilon \{\|(\nabla_x\partial^{\alpha}\widetilde{u},\nabla_x\partial^{\alpha}\widetilde{\theta})\|^2+\|\nabla_x\partial^{\alpha}f\|^2\}
+C\varepsilon\varepsilon^{2-2a}.
\end{align}
Moreover, for $|\alpha|\leq N-2$, one has
\begin{align}
\label{4.86b}
\varepsilon\int^{t}_0\{\|\nabla_{x}\partial^{\alpha}\widetilde{\rho}(s)\|^2+\|\nabla_{x}\cdot\partial^{\alpha}\widetilde{E}(s)\|^2\}\,ds
\leq C(\varepsilon+t\varepsilon)\varepsilon^{2-2a}.
\end{align}
\end{lemma}
\begin{proof}	
We subtract system \eqref{3.1} from system \eqref{2.13} to obtain
\begin{align}
\label{4.4}
\left\{
\begin{array}{rl}
&\partial_{t}\widetilde{\rho}+u\cdot\nabla_{x}\widetilde{\rho}+\bar{\rho}\nabla_{x}\cdot\widetilde{u}
+\widetilde{u}\cdot\nabla_{x}\bar{\rho}+\widetilde{\rho}\nabla_{x}\cdot u=0,
\\
&\partial_{t}\widetilde{u}+u\cdot\nabla_{x}\widetilde{u}+\frac{2\bar{\theta}}{3\bar{\rho}}\nabla_{x}\widetilde{\rho}
+\frac{2}{3}\nabla_{x}\widetilde{\theta}+\widetilde{u}\cdot\nabla_{x}\bar{u}+\frac{2}{3}(\frac{\theta}{\rho}-\frac{\bar{\theta}}{\bar{\rho}})
\nabla_{x}\rho
\\
&\hspace{2cm}+(\widetilde{E}+u\times \widetilde{B}+\widetilde{u}\times\bar{B})=-\frac{1}{\rho}\int_{\mathbb{R}^{3}} v\otimes v\cdot\nabla_{x} G\,dv,
\\
&\partial_{t}\widetilde{\theta}+u\cdot\nabla_{x}\widetilde{\theta}+\frac{2}{3}\bar{\theta}\nabla_{x}\cdot \widetilde{u}
+\widetilde{u}\cdot\nabla_{x}\bar{\theta}+\frac{2}{3}\widetilde{\theta}\nabla_{x}\cdot u
\\
&\hspace{2cm}=-\frac{1}{\rho}\int_{\mathbb{R}^{3}} \frac{1}{2}|v|^{2} v\cdot\nabla_{x} G\,dv+\frac{1}{\rho}u\cdot\int_{\mathbb{R}^{3}} v\otimes v\cdot\nabla_{x}G\,dv,
\end{array} 
\right.
\end{align}	
where the electromagnetic field $(\widetilde{E},\widetilde{B})$ satisfies \eqref{4.3}.
Applying  $\partial^{\alpha}$ with  $|\alpha|\leq N-1$ to the second equation of \eqref{4.4}, then multiplying the resulting identity by $\nabla_x\partial^{\alpha}\widetilde{\rho}$ and integrating {over} $\mathbb{R}^3$ by parts, one has
\begin{align}
\label{4.87a}	
&(\partial^{\alpha}\partial_{t}\widetilde{u},\nabla_x\partial^{\alpha}\widetilde{\rho})+(\frac{2\bar{\theta}}{3\bar{\rho}}\nabla_x\partial^{\alpha}\widetilde{\rho},\nabla_x\partial^{\alpha}\widetilde{\rho})+\sum_{1\leq \alpha_{1}\leq \alpha}C_{\alpha}^{\alpha_{1}}(\partial^{\alpha_1}(\frac{2\bar{\theta}}{3\bar{\rho}})\nabla_x\partial^{\alpha-\alpha_{1}}\widetilde{\rho},\nabla_x\partial^{\alpha}\widetilde{\rho})
\nonumber\\
=&-\big(\partial^{\alpha}(u\cdot\nabla_x\widetilde{u})+\frac{2}{3}\nabla_x\partial^{\alpha}\widetilde{\theta}+\partial^{\alpha}(\widetilde{u}\cdot\nabla_x\bar{u})+\partial^{\alpha}[\frac{2}{3}(\frac{\theta}{\rho}-\frac{\bar{\theta}}{\bar{\rho}})\nabla_x\rho],\nabla_x\partial^{\alpha}\widetilde{\rho}\big)
-(\partial^{\alpha}\widetilde{E},\nabla_{x}\partial^{\alpha}\widetilde{\rho})
\nonumber\\
&-(\partial^{\alpha}(u\times\widetilde{B}+\widetilde{u}\times\bar{B}),\nabla_{x}\partial^{\alpha}\widetilde{\rho})-(\partial^{\alpha}(\frac{1}{\rho}\int_{\mathbb{R}^{3}} v\otimes v\cdot\nabla_x G\,dv),\nabla_x\partial^{\alpha}\widetilde{\rho}).
\end{align}
Here if $|\alpha|=0$, the last term on the LHS of \eqref{4.87a} vanishes and if  $|\alpha|\geq 1$, this term exists.
By integration by parts and the first equation of \eqref{4.4}, one has
\begin{align*}
(\partial^{\alpha}\partial_{t}\widetilde{u},\nabla_{x}\partial^{\alpha}\widetilde{\rho})
=&\frac{d}{dt}(\partial^{\alpha}\widetilde{u},\nabla_{x}\partial^{\alpha}\widetilde{\rho})+(\nabla_{x}\partial^{\alpha}\widetilde{u},\partial^{\alpha}\partial_{t}\widetilde{\rho})
\\
=&\frac{d}{dt}(\partial^{\alpha}\widetilde{u},\nabla_{x}\partial^{\alpha}\widetilde{\rho})
-(\nabla_{x}\partial^{\alpha}\widetilde{u},\partial^{\alpha}[u\cdot\nabla_{x}\widetilde{\rho}+\bar{\rho}\nabla_{x}\cdot\widetilde{u}
+\widetilde{u}\cdot\nabla_{x}\bar{\rho}+\widetilde{\rho}\nabla_{x}\cdot u])
\\
\geq&\frac{d}{dt}(\partial^{\alpha}\widetilde{u},\nabla_{x}\partial^{\alpha}\widetilde{\rho})-C\eta\|\nabla_{x}\partial^{\alpha}\widetilde{\rho}\|^2
-C_\eta\|\nabla_{x}\partial^{\alpha}\widetilde{u}\|^2-C(\eta_{0}+\varepsilon^{\frac{1}{2}-a})\varepsilon^{2-2a},
\end{align*}
where in the last inequality  we used the Cauchy-Schwarz inequality, the embedding inequality, \eqref{3.3}, \eqref{4.13}, \eqref{4.27}
and \eqref{4.22a}.  From $\nabla_{x}\cdot \widetilde{E}=-\widetilde{\rho}$ in \eqref{4.3}, 
one gets by integration by parts that
\begin{align*}
-(\partial^{\alpha}\widetilde{E},\nabla_{x}\partial^{\alpha}\widetilde{\rho})=
(\nabla_{x}\cdot\partial^{\alpha}\widetilde{E},\partial^{\alpha}\widetilde{\rho})=-(\nabla_{x}\cdot\partial^{\alpha}\widetilde{E},\nabla_{x}\cdot \partial^{\alpha}\widetilde{E}).	
\end{align*}
%Using \eqref{3.3}, \eqref{4.13} and \eqref{4.22a} again, we get
%\begin{align*}|(\partial^{\alpha}(u\times\widetilde{B}+\widetilde{u}\times\bar{B}),\nabla_{x}\partial^{\alpha}\widetilde{\rho})|
%\leq \eta\|\nabla_{x}\partial^{\alpha}\widetilde{\rho}\|^2+C_\eta(\eta_{0}+\varepsilon^{\frac{1}{2}-a})\varepsilon^{2-2a}.
%\end{align*}
The other terms of \eqref{4.87a} can be controlled by
$$
(\eta+\varepsilon^{1-a})\|\nabla_{x}\partial^{\alpha}\widetilde{\rho}\|^2+
C_\eta\{\|(\nabla_x\partial^{\alpha}\widetilde{u},\nabla_x\partial^{\alpha}\widetilde{\theta})\|^2+\|\nabla_x\partial^{\alpha}f\|^2\}
+C_\eta(\eta_{0}+\varepsilon^{\frac{1}{2}-a})\varepsilon^{2-2a}.
$$
Hence, for $|\alpha|\leq N-1$, substituting the above all estimates into \eqref{4.87a}, one gets
\begin{align*}	
&\frac{d}{dt}(\partial^{\alpha}\widetilde{u},\nabla_x\partial^{\alpha}\widetilde{\rho})
+c(\|\nabla_x\partial^{\alpha}\widetilde{\rho}\|^2+\|\nabla_{x}\cdot \partial^{\alpha}\widetilde{E}\|^2)
\nonumber\\
&\leq C\{\|(\nabla_x\partial^{\alpha}\widetilde{u},\nabla_x\partial^{\alpha}\widetilde{\theta})\|^2+\|\nabla_x\partial^{\alpha}f\|^2\}
+C\varepsilon^{2-2a},
\end{align*}
by choosing  a small $\eta>0$ and using \eqref{4.22a}. This gives the desired estimate \eqref{4.86a}.
Integrating \eqref{4.86a} with respect to $t$ and using $\mathcal{E}_{N}(0)\leq C\eta^2_{0}\varepsilon^{2}$,
we can prove that \eqref{4.86b} holds. This ends the proof of Lemma \ref{lem4.17}.
%For $|\alpha|\leq N-2$, we easily see by \eqref{4.13} that
%$$\|(\nabla_x\partial^{\alpha}\widetilde{u},\nabla_x\partial^{\alpha}\widetilde{\theta})\|^2+\|\nabla_x\partial^{\alpha}f\|^2
%\leq C\varepsilon^{2-2a},$$
%and
%$$|(\partial^{\alpha}\widetilde{u},\nabla_x\partial^{\alpha}\widetilde{\rho})|
%\leq C\|\partial^{\alpha}\widetilde{u}\|\|\nabla_x\partial^{\alpha}\widetilde{\rho}\|\leq C\varepsilon^{2-2a}.
%$$
\end{proof}
Note that the norms in \eqref{4.14} do not include the temporal derivatives, we have to use the pure space
derivative norms to bound those estimates on the fluid quantities involving with the temporal derivatives.
\begin{lemma}\label{lem4.18}
 For $|\alpha|\leq N-1$, it holds that
\begin{align}
\label{4.86}
\|\partial^{\alpha}(\partial_{t}\widetilde{\rho},\partial_{t}\widetilde{u},\partial_{t}\widetilde{\theta})\|^{2}\leq &
C\{\|\partial^{\alpha}(\nabla_{x}\widetilde{\rho},\nabla_{x}\widetilde{u},\nabla_{x}\widetilde{\theta})\|^{2}
+\|\partial^{\alpha}\nabla_{x} f\|^{2}\}+C\varepsilon^{2-2a}.
\end{align}
In particular, for $|\alpha|\leq N-2$, one has
\begin{align}
\label{4.87}
\|\partial^{\alpha}(\partial_{t}\widetilde{\rho},\partial_{t}\widetilde{u},\partial_{t}\widetilde{\theta})\|^{2}
\leq &C\varepsilon^{2-2a}.
\end{align}
For $|\alpha|=N-1$, one has
\begin{align}
\label{5.47b}
\|\partial^{\alpha}(\partial_{t}\widetilde{\rho},\partial_{t}\widetilde{u},\partial_{t}\widetilde{\theta})\|^{2}\leq 
C\varepsilon^{-2a}.
\end{align}
\end{lemma}
\begin{proof}
Let $|\alpha|\leq N-1$, applying  $\partial^{\alpha}$ to the second equation of \eqref{4.4}, 
then multiplying the resulting identity by $\partial^{\alpha}\partial_{t}\widetilde{u}$ and integrating over $\mathbb{R}^3$ by parts, one has	
\begin{align}
\label{4.88}
&(\partial^{\alpha}\partial_{t}\widetilde{u},\partial^{\alpha}\partial_{t}\widetilde{u})
\nonumber\\
=&-(\partial^{\alpha}(u\cdot\nabla_{x}\widetilde{u}),\partial^{\alpha}\partial_{t}\widetilde{u})
-(\partial^{\alpha}[\frac{2\bar{\theta}}{3\bar{\rho}}\nabla_{x}\widetilde{\rho}+\frac{2}{3}\nabla_{x}\widetilde{\theta}+\widetilde{u}\cdot\nabla_{x}\bar{u}+\frac{2}{3}(\frac{\theta}{\rho}-\frac{\bar{\theta}}{\bar{\rho}})
\nabla_{x}\rho],\partial^{\alpha}\partial_{t}\widetilde{u})
\nonumber\\
&-(\partial^{\alpha}(\widetilde{E}+u\times\widetilde{B}+\widetilde{u}\times\bar{B}),\partial^{\alpha}\partial_{t}\widetilde{u})-(\partial^{\alpha}(\frac{1}{\rho}\int_{\mathbb{R}^{3}} v\otimes v\cdot\nabla_{x} G\,dv),\partial^{\alpha}\partial_{t}\widetilde{u}).
\end{align}
Making use of the Cauchy-Schwarz inequality, the embedding inequality, \eqref{3.3}, \eqref{4.13} and \eqref{4.22a}, one gets
\begin{align*}
|(\partial^{\alpha}(u\cdot\nabla_{x}\widetilde{u}),\partial^{\alpha}\partial_{t}\widetilde{u})|
&\leq  \eta\|\partial^{\alpha}\partial_{t}\widetilde{u}\|^{2}+C_\eta\|\partial^{\alpha}(\bar{u}\cdot\nabla_{x}\widetilde{u})\|^2
+C_\eta\|\partial^{\alpha}(\widetilde{u}\cdot\nabla_{x}\widetilde{u})\|^2
\\
&\leq \eta\|\partial^{\alpha}\partial_{t}\widetilde{u}\|^{2}+C_\eta\|\partial^{\alpha}\nabla_{x}\widetilde{u}\|^{2}
+C_\eta\varepsilon^{2-2a}.
\end{align*}
The other terms of \eqref{4.88} can be treated similarly, we thus arrive at
\begin{align}
\label{4.89}
\|\partial^{\alpha}\partial_{t}\widetilde{u}\|^2
\leq C\eta\|\partial^{\alpha}\partial_{t}\widetilde{u}\|^{2}
+C_\eta\{\|\partial^{\alpha}(\nabla_{x}\widetilde{\rho},\nabla_{x}\widetilde{u},\nabla_{x}\widetilde{\theta})\|^{2}+\|\partial^{\alpha}\nabla_{x} f\|^{2}\}
+C_\eta\varepsilon^{2-2a}.
\end{align}
Similar estimates also hold for $\partial^{\alpha}\partial_{t}\widetilde{\rho}$ and $\partial^{\alpha}\partial_{t}\widetilde{\theta}$. 
Hence, choosing $\eta>0$ small enough, we can prove that \eqref{4.86} holds true. 
From \eqref{4.13}, \eqref{4.13A} and \eqref{4.86}, we can obtain the desired estimates \eqref{4.87} and \eqref{5.47b}.
%we can see that for $|\alpha|\leq N-2$,
%\begin{align*}
%\|\partial^{\alpha}(\nabla_{x}\widetilde{\rho},\nabla_{x}\widetilde{u},\nabla_{x}\widetilde{\theta})\|^{2}+\|\partial^{\alpha}\nabla_{x} f\|^{2}\leq C\varepsilon^{2-2a},
%\end{align*}
Thus the proof of Lemma \ref{lem4.18} is complete.
\end{proof}
%\subsection{Estimates on solution}\label{seca.5.6}
At the end of this section, we give the characterization of low bound of the $N$-order energy norm
$\|\frac{\partial^{\alpha}F(t)}{\sqrt{\mu}}\|^{2}$ in terms of $\|\partial^{\alpha}(\widetilde{\rho},\widetilde{u},\widetilde{\theta})\|^{2}+\|\partial^{\alpha}f\|^{2}$.
\begin{lemma}\label{lem4.20}
For $|\alpha|=N$, there exists a constant $c_3>0$ such that
\begin{equation}\label{4.85A}
\|\frac{\partial^{\alpha}F(t)}{\sqrt{\mu}}\|^{2}
\geq c_3(\|\partial^{\alpha}(\widetilde{\rho},\widetilde{u},\widetilde{\theta})\|^{2}+\|\partial^{\alpha}f\|^{2})-C(\eta_{0}+\varepsilon^{\frac{1}{2}-a}).
\end{equation}
\end{lemma}
\begin{proof}
The proof is similar to the one in \cite[Lemma 5.10]{Duan-Yang-Yu}, we thus omit the details for {brevity}.	
\end{proof}

\section{A priori estimates}\label{seca.5}
In this section, we aim to establish the a priori energy estimates \eqref{3.8}
on the solution step by step in a series of lemmas in order to close a priori assumption \eqref{4.13}.
\subsection{Space derivative estimates up to $(N-1)$-order}\label{seca.6.2}
In this subsection, we are {devoted} to obtaining the space derivative estimates for both the fluid part and the non-fluid part  up to $(N-1)$-order.

\subsubsection{Space derivative estimates on fluid part up to $(N-1)$-order}\label{seca.6.2.1} 
We first make use of the dissipation mechanism for the Navier-Stokes-Maxwell-type equations \eqref{4.2}-\eqref{4.3}
to derive the space derivative estimates on the fluid part
$(\widetilde{\rho},\widetilde{u},\widetilde{\theta})$ and the electromagnetic field
$(\widetilde{E},\widetilde{B})$.
\begin{lemma}
\label{lem6.6}
Let $[F(t,x,v),E(t,x),B(t,x)]$ be the solution to the VMB system \eqref{1.1} and \eqref{3.6}.
Suppose that \eqref{3.3} and \eqref{4.13} hold, one has
\begin{align}
\label{6.34}
&\sum_{|\alpha|\leq N-1}\|\partial^{\alpha}(\widetilde{\rho},\widetilde{u},\widetilde{\theta},\widetilde{E},\widetilde{B})(t)\|^2
+c\varepsilon\sum_{1\leq |\alpha|\leq N}\int^t_0\|\partial^{\alpha}(\widetilde{\rho},\widetilde{u},\widetilde{\theta})(s)\|^{2}\,ds
\nonumber\\
&\hspace{0.5cm}\leq C\varepsilon^2\sum_{|\alpha|=N}
(\|\partial^{\alpha}(\widetilde{\rho},\widetilde{u},\widetilde{\theta})(t)\|^2+
\|\partial^{\alpha}f(t)\|^2)+ C\varepsilon\sum_{1\leq|\alpha|\leq N}\int^t_0\| \partial^{\alpha}f(s)\|^2\,ds
\nonumber\\
&\hspace{1cm}+C(\eta_{0}+\varepsilon^{\frac{1}{2}-a})\int^t_0\mathcal{D}_N(s)\,ds+C[\eta_{0}+\varepsilon+(\eta_{0}\varepsilon^a
+\varepsilon^{\frac{1}{2}-a})t]\varepsilon^{2-2a}.
\end{align}
\end{lemma}
\begin{proof}
We split the proof into four steps. 
%In the first three steps we make the direct energy estimates on $(\widetilde{\rho},\widetilde{u},\widetilde{\theta})$ and $(\widetilde{E},\widetilde{B})$
%in term of the system \eqref{4.2}-\eqref{4.3} and then combine those estimates to deduce the desired estimate \eqref{6.34} in the last step.

\medskip
\noindent{\it Step 1. Estimate on $\|\partial^{\alpha}\widetilde{\rho}\|^2$}. Let $|\alpha|\leq N-1$,
we apply $\partial^{\alpha}$ to the first equation  of \eqref{4.2} and take the inner product of the resulting equation with 
$\frac{2\bar{\theta}}{3\bar{\rho}^{2}}\partial^{\alpha}\widetilde{\rho}$ to get
\begin{align}
\label{6.35}
\frac{1}{2}&\frac{d}{dt}(\partial^{\alpha}\widetilde{\rho},\frac{2\bar{\theta}}{3\bar{\rho}^{2}}\partial^{\alpha}\widetilde{\rho})
-\frac{1}{2}(\partial^{\alpha}\widetilde{\rho},\partial_t(\frac{2\bar{\theta}}{3\bar{\rho}^{2}})\partial^{\alpha}\widetilde{\rho})
+(\bar{\rho}\nabla_x\cdot\partial^{\alpha}\widetilde{u},\frac{2\bar{\theta}}{3\bar{\rho}^{2}}\partial^{\alpha}\widetilde{\rho})
\nonumber\\
&+\sum_{1\leq \alpha_{1}\leq \alpha}C^{\alpha_{1}}_\alpha(\partial^{\alpha_1}\bar{\rho}\nabla_x\cdot\partial^{\alpha-\alpha_{1}}\widetilde{u},
\frac{2\bar{\theta}}{3\bar{\rho}^{2}}\partial^{\alpha}\widetilde{\rho})
+(\partial^{\alpha}(u\cdot\nabla_x\widetilde{\rho}),\frac{2\bar{\theta}}{3\bar{\rho}^{2}}\partial^{\alpha}\widetilde{\rho})
\nonumber\\
&=
-(\partial^{\alpha}(\widetilde{u}\cdot\nabla_x\bar{\rho}),\frac{2\bar{\theta}}{3\bar{\rho}^{2}}\partial^{\alpha}\widetilde{\rho})
-(\partial^{\alpha}(\widetilde{\rho}\nabla_x\cdot u),\frac{2\bar{\theta}}{3\bar{\rho}^{2}}\partial^{\alpha}\widetilde{\rho}).
\end{align}	
Here if $|\alpha|=0$, the fourth term on the LHS of \eqref{6.35} vanishes and if  $|\alpha|\geq 1$, this term exists. 

Let's estimate each term in \eqref{6.35}.
In view of \eqref{3.3}, \eqref{3.1} and \eqref{4.13}, we get
\begin{align}
\label{5.3a}
|(\partial^{\alpha}\widetilde{\rho},\partial_t(\frac{2\bar{\theta}}{3\bar{\rho}^{2}})\partial^{\alpha}\widetilde{\rho})|
\leq C\|(\partial_t\bar{\rho},\partial_t\bar{\theta})\|_{L^\infty}\|\partial^{\alpha}\widetilde{\rho}\|^2
\leq C\eta_0(1+t)^{-\vartheta}\varepsilon^{2-2a},
\end{align}
and
\begin{align}
\label{6.36}
\sum_{1\leq \alpha_{1}\leq \alpha}&C^{\alpha_{1}}_\alpha(\partial^{\alpha_1}\bar{\rho}\nabla_x\cdot\partial^{\alpha-\alpha_{1}}\widetilde{u},\frac{2\bar{\theta}}{3\bar{\rho}^{2}}\partial^{\alpha}\widetilde{\rho})
\nonumber\\
&\leq C\sum_{1\leq \alpha_{1}\leq \alpha}\|\partial^{\alpha_1}\bar{\rho}\|_{L^\infty}
\|\nabla_x\cdot\partial^{\alpha-\alpha_{1}}\widetilde{u}\|\|\partial^{\alpha}\widetilde{\rho}\|\leq C\eta_0(1+t)^{-\vartheta}\varepsilon^{2-2a}.
\end{align}
The last term on the LHS of \eqref{6.35} is equivalent to  
\begin{align}
\label{6.37}
&-
(u\cdot\nabla_x\partial^{\alpha}\widetilde{\rho},\frac{2\bar{\theta}}{3\bar{\rho}^{2}}\partial^{\alpha}\widetilde{\rho})
-\sum_{1\leq \alpha_{1}\leq \alpha}C^{\alpha_{1}}_\alpha(\partial^{\alpha_1}\bar{u}\cdot\nabla_x\partial^{\alpha-\alpha_1}\widetilde{\rho},\frac{2\bar{\theta}}{3\bar{\rho}^{2}}\partial^{\alpha}\widetilde{\rho})
\nonumber\\
&-\sum_{1\leq \alpha_{1}\leq \alpha}C^{\alpha_{1}}_\alpha(\partial^{\alpha_1}\widetilde{u}\cdot\nabla_x\partial^{\alpha-\alpha_1}\widetilde{\rho},
\frac{2\bar{\theta}}{3\bar{\rho}^{2}}\partial^{\alpha}\widetilde{\rho}),
\end{align}
where if $|\alpha|=0$, the last two terms of \eqref{6.37}
vanish and if  $|\alpha|\geq 1$, those terms exist. 
For the first term in \eqref{6.37}, we get by integration by parts, \eqref{3.3} and \eqref{4.13} that
\begin{align*}
|(u\cdot\nabla_x\partial^{\alpha}\widetilde{\rho},\frac{2\bar{\theta}}{3\bar{\rho}^{2}}\partial^{\alpha}\widetilde{\rho})|&=	|\frac{1}{2}(\partial^{\alpha}\widetilde{\rho},\nabla_x\cdot(u\frac{2\bar{\theta}}{3\bar{\rho}^{2}})\partial^{\alpha}\widetilde{\rho})|
\\	
&\leq C(\|\nabla_{x}(\bar{\rho},\bar{u},\bar{\theta})\|_{L^{\infty}}+\|\nabla_{x}\widetilde{u}\|_{L^{\infty}})
\|\partial^{\alpha}\widetilde{\rho}\|^2
\\
&\leq C[\eta_0(1+t)^{-\vartheta}+\varepsilon^{\frac{1}{2}-a}]\varepsilon^{2-2a},
\end{align*}
where we have used the fact that
\begin{align}
\label{5.39a}
\|\nabla_{x}\widetilde{u}\|_{L^{\infty}}\leq C\|\nabla^2_{x}\widetilde{u}\|^{\frac{1}{2}}
\|\nabla^3_{x}\widetilde{u}\|^{\frac{1}{2}}\leq C\varepsilon^{\frac{1}{2}-\frac{1}{2}a}\varepsilon^{-\frac{1}{2}a}
=C\varepsilon^{\frac{1}{2}-a}.
\end{align}
For the second term in \eqref{6.37}, we use similar arguments as \eqref{6.36} to get the same bound.
The third term in \eqref{6.37} can be controlled by
%we consider the two cases  $1\leq|\alpha_{1}|\leq N/2$ and $N/2<|\alpha_{1}|\leq N-1$. For the first case, we use the similar arguments as \eqref{5.39a} and \eqref{4.13} to get
%\begin{align*}
%&|(\partial^{\alpha_1}\widetilde{u}\cdot\nabla_x\partial^{\alpha-\alpha_1}\widetilde{\rho},\frac{2\bar{\theta}}{3\bar{\rho}^{2}}\partial^{\alpha}\widetilde{\rho})|
%\leq \|\partial^{\alpha_1}\widetilde{u}\|_{L^{\infty}}\|\nabla_x\partial^{\alpha-\alpha_1}\widetilde{\rho}\|\|\partial^{\alpha}\widetilde{\rho}\|
%\leq C\varepsilon^{\frac{1}{2}-a}\varepsilon^{2-2a}.
%\end{align*} The second case can be estimated in a similar way, namely
%\begin{align*}
%|(\partial^{\alpha_1}\widetilde{u}\cdot\nabla_x\partial^{\alpha-\alpha_1}\widetilde{\rho},\frac{2\bar{\theta}}{3\bar{\rho}^{2}}\partial^{\alpha}\widetilde{\rho})|
%\leq C\|\partial^{\alpha_1}\widetilde{u}\|\|\nabla_x\partial^{\alpha-\alpha_1}\widetilde{\rho}\|_{L^{\infty}}\|\partial^{\alpha}\widetilde{\rho}\|
%\leq C\varepsilon^{\frac{1}{2}-a}\varepsilon^{2-2a}.
%\end{align*}
\begin{align}
\label{5.40a}
\sum_{1\leq \alpha_{1}\leq \alpha}C^{\alpha_{1}}_\alpha|(\partial^{\alpha_1}\widetilde{u}\cdot\nabla_x\partial^{\alpha-\alpha_1}\widetilde{\rho},
\frac{2\bar{\theta}}{3\bar{\rho}^{2}}\partial^{\alpha}\widetilde{\rho})|\leq C\varepsilon^{\frac{1}{2}-a}\varepsilon^{2-2a}.
\end{align}
Hence, the last term on the LHS of \eqref{6.35} is bounded by
\begin{align}
\label{6.38}
|(\partial^{\alpha}(u\cdot\nabla_x\widetilde{\rho}),\frac{2\bar{\theta}}{3\bar{\rho}^{2}}\partial^{\alpha}\widetilde{\rho})|
\leq C[\eta_0(1+t)^{-\vartheta}+\varepsilon^{\frac{1}{2}-a}]\varepsilon^{2-2a}.
\end{align} 
For the first term on the RHS of \eqref{6.35}, we use \eqref{3.3} and \eqref{4.13} again to get
\begin{align}
\label{6.39}
|(\partial^{\alpha}(\widetilde{u}\cdot\nabla_x\bar{\rho}),\frac{2\bar{\theta}}{3\bar{\rho}^{2}}\partial^{\alpha}\widetilde{\rho})|
\leq C\sum_{\alpha_{1}\leq \alpha}
\|\nabla_x\partial^{\alpha-\alpha_{1}}\bar{\rho}\|_{L^{\infty}}\|\partial^{\alpha_1}\widetilde{u}\|\|\partial^{\alpha}\widetilde{\rho}\|
\leq C\eta_0(1+t)^{-\vartheta}\varepsilon^{2-2a}.
\end{align}
Similar to \eqref{6.37}, the last term on the RHS of \eqref{6.35} is equivalent to
\begin{align}
\label{6.40}
&-(\partial^{\alpha}(\widetilde{\rho}\nabla_x\cdot \bar{u}),\frac{2\bar{\theta}}{3\bar{\rho}^{2}}\partial^{\alpha}\widetilde{\rho})
-(\widetilde{\rho}\nabla_x\cdot \partial^{\alpha}\widetilde{u},\frac{2\bar{\theta}}{3\bar{\rho}^{2}}\partial^{\alpha}\widetilde{\rho})
\nonumber\\
&-\sum_{1\leq \alpha_{1}\leq \alpha}C^{\alpha_{1}}_\alpha
(\partial^{\alpha_1}\widetilde{\rho}\nabla_x\cdot\partial^{\alpha-\alpha_{1}}\widetilde{u},\frac{2\bar{\theta}}{3\bar{\rho}^{2}}\partial^{\alpha}\widetilde{\rho}).
\end{align}
Here if $|\alpha|=0$, the last term of \eqref{6.40} vanishes and if  $|\alpha|\geq 1$, this term exists. 
For this term, we can use the similar arguments as \eqref{5.40a} to get the same bound.
The first term in \eqref{6.40} has the same bound as \eqref{6.39}.
%\begin{align*}
%|(\partial^{\alpha}(\widetilde{\rho}\nabla_x\cdot \bar{u}),\frac{2\bar{\theta}}{3\bar{\rho}^{2}}\partial^{\alpha}\widetilde{\rho})|
%\leq C\eta_0(1+t)^{-\vartheta}\varepsilon^{2-2a}.
%\end{align*}
For the second term in \eqref{6.40}, using \eqref{4.13}, it is straightforward to get
\begin{align*}
|(\widetilde{\rho}\nabla_x\cdot \partial^{\alpha}\widetilde{u},\frac{2\bar{\theta}}{3\bar{\rho}^{2}}\partial^{\alpha}\widetilde{\rho})|
&\leq\eta\varepsilon\|\nabla_x\cdot\partial^{\alpha}\widetilde{u}\|^2
+C_\eta\frac{1}{\varepsilon}\|\widetilde{\rho}\|_{L^{\infty}}^2
\|\partial^{\alpha}\widetilde{\rho}\|^2
\\
&\leq \eta\varepsilon\|\nabla_x\cdot\partial^{\alpha}\widetilde{u}\|^2
+C_\eta\varepsilon^{1-2a}\varepsilon^{2-2a}.
\end{align*}
Hence, using $a\in[0,1/2)$, the last term on the RHS of \eqref{6.35} is bounded by
\begin{align}
\label{6.41}
|(\partial^{\alpha}(\widetilde{\rho}\nabla_x\cdot u),\frac{2\bar{\theta}}{3\bar{\rho}^{2}}\partial^{\alpha}\widetilde{\rho})|
\leq C\eta\varepsilon\|\nabla_x\cdot\partial^{\alpha}\widetilde{u}\|^2+C_\eta[\eta_0(1+t)^{-\vartheta}+\varepsilon^{\frac{1}{2}-a}]\varepsilon^{2-2a}.
\end{align}
In summary, for $|\alpha|\leq N-1$ and any small $\eta>0$, we
substitute \eqref{5.3a}, \eqref{6.36}, \eqref{6.38}, \eqref{6.39}, \eqref{6.41} into \eqref{6.35} to get
\begin{align}
\label{6.42}
\frac{1}{2}&\frac{d}{dt}\int_{\mathbb{R}^{3}} \frac{2\bar{\theta}}{3\bar{\rho}^{2}}|\partial^{\alpha}\widetilde{\rho}|^{2}\,dx
+(\nabla_x\cdot\partial^{\alpha}\widetilde{u},\frac{2\bar{\theta}}{3\bar{\rho}}\partial^{\alpha}\widetilde{\rho})
\nonumber\\
&\leq C\eta\varepsilon\|\nabla_x\cdot\partial^{\alpha}\widetilde{u}\|^2
+C_\eta[\eta_0(1+t)^{-\vartheta}+\varepsilon^{\frac{1}{2}-a}]\varepsilon^{2-2a}.
\end{align}

\medskip
\noindent{\it Step 2. Estimate on $\|\partial^{\alpha}\widetilde{u}\|^2$}.
Applying $\partial^{\alpha}$ with $|\alpha|\leq N-1$ to the second equation of \eqref{4.2} and taking the inner product of the resulting equation with 
$\partial^{\alpha}\widetilde{u}_i$, one gets
\begin{align}
\label{6.43}
&\frac{1}{2}\frac{d}{dt}\|\partial^{\alpha}\widetilde{u}_i\|^{2}+
(\frac{2\bar{\theta}}{3\bar{\rho}}\partial^{\alpha}\partial_{x_{i}}\widetilde{\rho},\partial^{\alpha}\widetilde{u}_i)
+(\frac{2}{3}\partial^{\alpha}\partial_{x_{i}}\widetilde{\theta},\partial^{\alpha}\widetilde{u}_i)
\nonumber\\
=&
-\sum_{1\leq \alpha_{1}\leq \alpha}C_{\alpha}^{\alpha_{1}}(\partial^{\alpha_1}(\frac{2\bar{\theta}}{3\bar{\rho}})
\partial^{\alpha-\alpha_1}\partial_{x_{i}}\widetilde{\rho},\partial^{\alpha}\widetilde{u}_i)
-(\partial^{\alpha}(u\cdot\nabla_{x}\widetilde{u}_{i}),\partial^{\alpha}\widetilde{u}_i)
-(\partial^{\alpha}(\widetilde{u}\cdot\nabla_{x}\bar{u}_{i}),\partial^{\alpha}\widetilde{u}_i)
\nonumber\\
&-(\partial^{\alpha}[\frac{2}{3}(\frac{\theta}{\rho}-\frac{\bar{\theta}}{\bar{\rho}})\partial_{x_{i}}\rho],\partial^{\alpha}\widetilde{u}_i)
-(\partial^{\alpha}\widetilde{E}_i,\partial^{\alpha}\widetilde{u}_i)
-(\partial^{\alpha}(u\times \widetilde{B}+\widetilde{u}\times\bar{B})_{i},\partial^{\alpha}\widetilde{u}_i)
\nonumber\\
&+\varepsilon\sum^{3}_{j=1}(\partial^{\alpha}(\frac{1}{\rho}\partial_{x_{j}}[\mu(\theta)D_{ij}]),\partial^{\alpha}\widetilde{u}_i)
-(\partial^{\alpha}(\frac{1}{\rho}\int_{\mathbb{R}^{3}} v_{i}v\cdot\nabla_{x}L^{-1}_{M}\Theta\,dv),\partial^{\alpha}\widetilde{u}_i).
\end{align}
Here if $|\alpha|=0$, the first term on the RHS of \eqref{6.43} vanishes and if  $|\alpha|\geq 1$, this term exists.

Let's estimate \eqref{6.43} term by term. Using the similar manner as \eqref{6.38}, \eqref{6.39}
and \eqref{6.41}, we can obtain that the first four terms on the RHS of \eqref{6.43} are bounded by
\begin{align}
\label{5.15A}
 C\eta\varepsilon\|\partial_{x_i}\partial^{\alpha}\widetilde{u}_i\|^2
+C_\eta[\eta_0(1+t)^{-\vartheta}+\varepsilon^{\frac{1}{2}-a}]\varepsilon^{2-2a}.
\end{align}
To bound the electric field term in \eqref{6.43}, we  note from the first equation of \eqref{4.3} that
$$
\partial_{t}\partial^{\alpha}\widetilde{E}-\nabla_{x}\times \partial^{\alpha}\widetilde{B}=\partial^{\alpha}(\rho\widetilde{u})
+\partial^{\alpha}(\widetilde{\rho}\bar{u}),
$$
which leads us to obtain
\begin{align}
\label{5.47a}
\sum^3_{i=1}(\partial^{\alpha}\widetilde{E}_i,\partial^{\alpha}\widetilde{u}_i)
=&(\frac{1}{\rho}\partial^{\alpha}\widetilde{E},\rho\partial^{\alpha}\widetilde{u})
=(\frac{1}{\rho}\partial^{\alpha}\widetilde{E},\partial_{t}\partial^{\alpha}\widetilde{E}
-\nabla_{x}\times \partial^{\alpha}\widetilde{B})
\nonumber\\
&-(\frac{1}{\rho}\partial^{\alpha}\widetilde{E},\partial^{\alpha}(\widetilde{\rho}\bar{u}))
-(\frac{1}{\rho}\partial^{\alpha}\widetilde{E},\sum_{1\leq \alpha_{1}\leq \alpha}C^{\alpha_{1}}_{\alpha}\partial^{\alpha_1}\rho\partial^{\alpha-\alpha_{1}}\widetilde{u}).
\end{align}
Here if $|\alpha|=0$, the last term of \eqref{5.47a} vanishes and if  $|\alpha|\geq 1$, this term exists. 
For the first term on the RHS of \eqref{5.47a}, we get by the similar arguments as \eqref{5.58} and \eqref{5.59} that
\begin{align*}
&(\frac{1}{\rho}\partial^{\alpha}\widetilde{E},\partial_{t}\partial^{\alpha}\widetilde{E}
-\nabla_{x}\times \partial^{\alpha}\widetilde{B})
\\
=&\frac{1}{2}\frac{d}{dt}(\frac{1}{\rho}\partial^{\alpha}\widetilde{E},\partial^{\alpha}\widetilde{E})
+\frac{1}{2}\frac{d}{dt}(\frac{1}{\rho}\partial^{\alpha}\widetilde{B},\partial^{\alpha}\widetilde{B})
\\
&-(\partial_{t}[\frac{1}{\rho}]\partial^{\alpha}\widetilde{E},\partial^{\alpha}\widetilde{E})
-(\partial_{t}[\frac{1}{\rho}]\partial^{\alpha}\widetilde{B},\partial^{\alpha}\widetilde{B})
-(\nabla_{x}[\frac{1}{\rho}]\times\partial^{\alpha}\widetilde{E}, \partial^{\alpha}\widetilde{B})
\\
\geq& \frac{1}{2}\frac{d}{dt}(\frac{1}{\rho}\partial^{\alpha}\widetilde{E},\partial^{\alpha}\widetilde{E})
+\frac{1}{2}\frac{d}{dt}(\frac{1}{\rho}\partial^{\alpha}\widetilde{B},\partial^{\alpha}\widetilde{B})
-C[\eta_0(1+t)^{-\vartheta}+\varepsilon^{\frac{1}{2}-a}]\varepsilon^{2-2a}.
\end{align*}
The last two terms of \eqref{5.47a} have the same bound as \eqref{6.38} and \eqref{6.39}, it follows that
\begin{align}
\label{5.17A}
&\sum^3_{i=1}(\partial^{\alpha}\widetilde{E}_i,\partial^{\alpha}\widetilde{u}_i)	
\nonumber\\
&\geq \frac{1}{2}\frac{d}{dt}(\frac{1}{\rho}\partial^{\alpha}\widetilde{E},\partial^{\alpha}\widetilde{E})
+\frac{1}{2}\frac{d}{dt}(\frac{1}{\rho}\partial^{\alpha}\widetilde{B},\partial^{\alpha}\widetilde{B})
-C[\eta_0(1+t)^{-\vartheta}+\varepsilon^{\frac{1}{2}-a}]\varepsilon^{2-2a}.
\end{align}
For the magnetic field term in \eqref{6.43}, we easily get by \eqref{3.3} and \eqref{4.13} that
\begin{align}
\label{5.18A}
|(\partial^{\alpha}(u\times \widetilde{B}+\widetilde{u}\times\bar{B})_{i},\partial^{\alpha}\widetilde{u}_i)|
\leq C[\eta_0(1+t)^{-\vartheta}+\varepsilon^{\frac{1}{2}-a}]\varepsilon^{2-2a}.
\end{align}
For the viscosity coefficient term in \eqref{6.43}. Recall the definition $D_{ij}$ in \eqref{2.14}, one has
\begin{align*}
\varepsilon&\sum^{3}_{j=1}(\partial^{\alpha}(\frac{1}{\rho}\partial_{x_{j}}[\mu(\theta)D_{ij}]),\partial^{\alpha}\widetilde{u}_i)
\\
=&\varepsilon\sum^{3}_{j=1}(\partial^{\alpha}(\frac{1}{\rho}\partial_{x_{j}}[\mu(\theta)(\partial_{x_{j}}\widetilde{u}_{i}+\partial_{x_{i}}\widetilde{u}_{j}-\frac{2}{3}\delta_{ij}\nabla_{x}\cdot\widetilde{u})]),\partial^{\alpha}\widetilde{u}_i)
\\
&+\varepsilon\sum^{3}_{j=1}(\partial^{\alpha}(\frac{1}{\rho}\partial_{x_{j}}[\mu(\theta)(\partial_{x_{j}}\bar{u}_{i}+\partial_{x_{i}}\bar{u}_{j}
-\frac{2}{3}\delta_{ij}\nabla_{x}\cdot \bar{u})]),\partial^{\alpha}\widetilde{u}_i)
\\
:=&I_{7}+I_{8}.
\end{align*}
%By integration by parts, we write $I_7=I^1_7+I^2_7+I^3_{7}$ with
%\begin{equation*}\left\{
%\begin{array}{rl}
%I^1_{7}&=-\varepsilon\sum^{3}_{j=1}(\frac{1}{\rho}\mu(\theta)(\partial^{\alpha}\partial_{x_{j}}\widetilde{u}_{i}+\partial^{\alpha}\partial_{x_{i}}\widetilde{u}_{j}-\frac{2}{3}\delta_{ij}\nabla_{x}\cdot \partial^{\alpha}\widetilde{u}),\partial^{\alpha}\partial_{x_{j}}\widetilde{u}_{i}),
%\\I^2_{7}&=-\varepsilon\sum^{3}_{j=1}\sum_{1\leq |\alpha_{1}|\leq|\alpha|}C_{\alpha}^{\alpha_{1}}(\partial^{\alpha_1}[\frac{1}{\rho}\mu(\theta)]
%\partial^{\alpha-\alpha_1}(\partial_{x_{j}}\widetilde{u}_{i}
%+\partial_{x_{i}}\widetilde{u}_{j}-\frac{2}{3}\delta_{ij}\nabla_{x}\cdot \widetilde{u}),\partial^{\alpha}\partial_{x_{j}}\widetilde{u}_{i}),
%\\I^3_{7}&=-\varepsilon\sum^{3}_{j=1}(\partial^{\alpha}[\partial_{x_{j}}(\frac{1}{\rho})\mu(\theta)(\partial_{x_{j}}\widetilde{u}_{i}
%+\partial_{x_{i}}\widetilde{u}_{j}-\frac{2}{3}\delta_{ij}\nabla_{x}\cdot\widetilde{u})],\partial^{\alpha}\widetilde{u}_{i}).
%\end{array} \right.
%\end{equation*}
%Here if $|\alpha|=0$, the term $I^2_7$ vanishes and if  $|\alpha|\geq 1$, this term exists.
Since both $\mu(\theta)$ and $\kappa(\theta)$ are smooth functions of $\theta$, there exists a constant $C_2>1$ 
such that $\mu(\theta),\kappa(\theta)\in[C^{-1}_2,C_2]$. Moreover, derivatives of  $\mu(\theta)$ and $\kappa(\theta)$ are also bounded.
Making use of this, the integration by parts, \eqref{3.3}, \eqref{4.13}, \eqref{4.13A}, \eqref{4.27} and \eqref{4.22a}, we arrive at
%\begin{align*}|I^2_{7}|+|I^3_{7}|\leq C[\eta_0(1+t)^{-\vartheta}+\varepsilon^{\frac{1}{2}-a}]\varepsilon^{2-2a}.
%\end{align*}
%With these, one gets
\begin{align*}
I_{7}
\leq&-\varepsilon\sum^{3}_{j=1}(\frac{1}{\rho}\mu(\theta)(\partial^{\alpha}\partial_{x_{j}}\widetilde{u}_{i}+\partial^{\alpha}\partial_{x_{i}}\widetilde{u}_{j}
-\frac{2}{3}\delta_{ij}\nabla_{x}\cdot \partial^{\alpha}\widetilde{u}),\partial^{\alpha}\partial_{x_{j}}\widetilde{u}_{i})
\\
&+C[\eta_0(1+t)^{-\vartheta}+\varepsilon^{\frac{1}{2}-a}]\varepsilon^{2-2a}.
\end{align*}
As for the term $I_{8}$, we have
\begin{align}
\label{5.15a}
I_{8}=&\varepsilon\sum^{3}_{j=1}(\frac{1}{\rho}\mu(\theta)(\partial^{\alpha}\partial^2_{x_{j}}\bar{u}_{i}+\partial^{\alpha}\partial_{x_{j}}\partial_{x_{i}}\bar{u}_{j}-\frac{2}{3}\delta_{ij}\nabla_{x}\cdot \partial^{\alpha}\partial_{x_{j}}\bar{u}),\partial^{\alpha}\widetilde{u}_i)
\nonumber\\
&+\varepsilon\sum^{3}_{j=1}\sum_{1\leq \alpha_{1}\leq\alpha}C_{\alpha}^{\alpha_{1}}(\partial^{\alpha_1}[\frac{1}{\rho}\mu(\theta)]\partial^{\alpha-\alpha_1}(\partial^2_{x_{j}}\bar{u}_{i}
+\partial_{x_{j}}\partial_{x_{i}}\bar{u}_{j}-\frac{2}{3}\delta_{ij}\nabla_{x}\cdot \partial_{x_{j}}\bar{u}),\partial^{\alpha}\widetilde{u}_i)
\nonumber\\
&+\varepsilon\sum^{3}_{j=1}(\partial^{\alpha}[\frac{1}{\rho}\partial_{x_{j}}\mu(\theta)(\partial_{x_{j}}\bar{u}_{i}+\partial_{x_{i}}\bar{u}_{j}
-\frac{2}{3}\delta_{ij}\nabla_{x}\cdot \bar{u})],\partial^{\alpha}\widetilde{u}_i)
\nonumber\\
\leq& C[\eta_0(1+t)^{-\vartheta}+\eta_0\varepsilon^a]\varepsilon^{2-2a}.
\end{align}
%Here if $|\alpha|=0$, the  term $I^2_8$ vanishes and if  $|\alpha|\geq 1$, this term exists.
%Using \eqref{3.3}, \eqref{4.13}, \eqref{4.13A}, \eqref{4.27} and \eqref{4.22a}, one gets
%\begin{align*}|I^2_8|\leq C\varepsilon\sum^{3}_{j=1}\sum_{1\leq |\alpha_{1}|\leq|\alpha|}&
%\|\partial^{\alpha-\alpha_1}(\partial^2_{x_{j}}\bar{u}_{i}+\partial_{x_{j}}\partial_{x_{i}}\bar{u}_{j}-\frac{2}{3}\delta_{ij}\nabla_{x}\cdot \partial_{x_{j}}\bar{u})\|_{L^\infty}
%\\&\times\|\partial^{\alpha_1}(\frac{1}{\rho}\mu(\theta))\|\|\partial^{\alpha}\widetilde{u}_i\|\leq C\eta_0(1+t)^{-\vartheta}\varepsilon^{2-2a}.
%\end{align*}
%The term $I^3_8$ has the same bound as $I^2_8$. For $I^1_8$, we get by\eqref{3.3} and \eqref{4.13} that
%\begin{align}\label{5.15a}
%|I^1_8|\leq C\varepsilon\|\partial^{\alpha}\nabla^2_{x}\bar{u}\|\|\partial^{\alpha}\widetilde{u}_i\|
%\leq C\eta_0\varepsilon^a\varepsilon^{2-2a},
%\end{align}which together with $I^2_8$ and $I^3_8$, gives
%$$
%|I_{8}|\leq C[\eta_0(1+t)^{-\vartheta}+\eta_0\varepsilon^a]\varepsilon^{2-2a}.
%$$
We thereupon  conclude from the estimates on  $I_{7}$ and  $I_{8}$ that
\begin{align}
\label{6.44}
&\varepsilon\sum^{3}_{j=1}(\partial^{\alpha}(\frac{1}{\rho}\partial_{x_{j}}[\mu(\theta)D_{ij}]),\partial^{\alpha}\widetilde{u}_i)
\nonumber\\
&\leq -\varepsilon\sum^{3}_{j=1}(\frac{1}{\rho}\mu(\theta)(\partial^{\alpha}\partial_{x_{j}}\widetilde{u}_{i}+\partial^{\alpha}\partial_{x_{i}}\widetilde{u}_{j}
-\frac{2}{3}\delta_{ij}\nabla_{x}\cdot \partial^{\alpha}\widetilde{u}),\partial^{\alpha}\partial_{x_{j}}\widetilde{u}_{i})
\nonumber\\
&\hspace{0.5cm}+C[\eta_0(1+t)^{-\vartheta}+\eta_0\varepsilon^a+\varepsilon^{\frac{1}{2}-a}]\varepsilon^{2-2a}.
\end{align}
It remains to compute the last term of \eqref{6.43}. We shall carefully deal with it. First of all, we have the following identity
\begin{align}
\label{6.22}
\int_{\mathbb{R}^{3}}v_{i}v_{j}L^{-1}_{M}\Theta \,dv=&{\int_{\mathbb{R}^{3}}\frac{v_{i}v_{j}ML^{-1}_{M}\Theta}{M} \,dv}=
\int_{\mathbb{R}^{3}} \frac{L^{-1}_{M}\{P_{1}(v_{i}v_{j}M)\}\Theta}{M}\,dv
\nonumber\\
=&\int_{\mathbb{R}^{3}} L^{-1}_{M}\{R\theta\hat{B}_{ij}(\frac{v-u}{\sqrt{R\theta}})M\}\frac{\Theta}{M}\,dv
=R\theta\int_{\mathbb{R}^{3}}B_{ij}(\frac{v-u}{\sqrt{R\theta}})\frac{\Theta}{M}\,dv,
\end{align}
for $i,j=1,2,3$, where we have used the self-adjoint property of $L^{-1}_{M}$, \eqref{2.5}, \eqref{2.17} and \eqref{2.18}.
We use \eqref{6.22} to write the last term of \eqref{6.43} as
\begin{align}
\label{6.45}
-&(\partial^{\alpha}(\frac{1}{\rho}\int_{\mathbb{R}^{3}} v_{i}v\cdot\nabla_{x}L^{-1}_{M}\Theta\,dv),\partial^{\alpha}\widetilde{u}_i)
\nonumber\\
=&-\sum^{3}_{j=1}(\partial^{\alpha}[\frac{1}{\rho}\partial_{x_j}(\int_{\mathbb{R}^{3}}R\theta B_{ij}(\frac{v-u}{\sqrt{R\theta}})\frac{\Theta}{M}\, dv)],\partial^{\alpha}\widetilde{u}_i)
\nonumber\\
=&\sum^{3}_{j=1}(\partial^{\alpha}[\frac{1}{\rho}\int_{\mathbb{R}^{3}}R\theta B_{ij}(\frac{v-u}{\sqrt{R\theta}})\frac{\Theta}{M}\, dv],\partial^{\alpha}\partial_{x_j}\widetilde{u}_i)
\nonumber\\
&+\sum^{3}_{j=1}(\partial^{\alpha}[\partial_{x_j}(\frac{1}{\rho})\int_{\mathbb{R}^{3}}R\theta B_{ij}(\frac{v-u}{\sqrt{R\theta}})\frac{\Theta}{M}\, dv],\partial^{\alpha}\widetilde{u}_i)
\nonumber\\
:=&I_{9}+I_{10},
\end{align}
where in the second identity we have used the integration by parts. 

Before computing \eqref{6.45}, we shall give the following desired estimate that
\begin{equation}
\label{6.25}
\int_{\mathbb{R}^{3}}\frac{|\langle v\rangle^{b}\sqrt{\mu}
\partial_{\beta}A_{i}(\frac{v-u}{\sqrt{R\theta}})|^{2}}{M^{2}}\,dv
+\int_{\mathbb{R}^{3}}\frac{|\langle v\rangle^{b}
\sqrt{\mu}\partial_{\beta}B_{ij}(\frac{v-u}{\sqrt{R\theta}})|^{2}}{M^{2}}\,dv\leq C,
\end{equation}
for any multi-index $\beta$ and $b\geq 0$, where \eqref{4.16} and \eqref{4.18} have used. Recall $\Theta$ in \eqref{2.12} given by
$$
\Theta=\varepsilon \partial_tG+\varepsilon P_{1}(v\cdot\nabla_{x}G)-\varepsilon(E+v\times B)\cdot\nabla_{v}G-Q(G,G).
$$ 
We easily see that
\begin{align*}
I_{9}=&\sum^{3}_{j=1}\int_{\mathbb{R}^{3}}\big\{\partial^{\alpha}[\int_{\mathbb{R}^{3}}\frac{1}{\rho}R\theta B_{ij}(\frac{v-u}{\sqrt{R\theta}})\frac{\varepsilon \partial_tG}{M}\, dv]\,\partial^{\alpha}\partial_{x_j}\widetilde{u}_i\big\}\,dx
\\
&+\sum^{3}_{j=1}\int_{\mathbb{R}^{3}}\big\{\partial^{\alpha}[\int_{\mathbb{R}^{3}}\frac{1}{\rho}R\theta B_{ij}(\frac{v-u}{\sqrt{R\theta}})\frac{\varepsilon P_{1}(v\cdot\nabla_{x}G)}{M}\, dv]\,\partial^{\alpha}\partial_{x_j}\widetilde{u}_i\big\}\,dx
\\
&-\sum^{3}_{j=1}\int_{\mathbb{R}^{3}}\big\{\partial^{\alpha}[\int_{\mathbb{R}^{3}}\frac{1}{\rho}R\theta B_{ij}(\frac{v-u}{\sqrt{R\theta}})
\frac{\varepsilon(E+v\times B)\cdot\nabla_{v}G}{M}\, dv]\,\partial^{\alpha}\partial_{x_j}\widetilde{u}_i\big\}\,dx
\\
&-\sum^{3}_{j=1}\int_{\mathbb{R}^{3}}\big\{\partial^{\alpha}[\int_{\mathbb{R}^{3}}\frac{1}{\rho}R\theta B_{ij}(\frac{v-u}{\sqrt{R\theta}})\frac{Q(G,G)}{M}\, dv]\,\partial^{\alpha}\partial_{x_j}\widetilde{u}_i\big\}\,dx
\\
:=&I^{1}_{9}+I^{2}_{9}+I^{3}_{9}+I^{4}_{9}.
\end{align*}
We compute each $I^{i}_{9}$ $(i=1,2,3,4)$ in the following way. We use $G=\overline{G}+\sqrt{\mu}f$ to write
\begin{align*}
I^{1}_{9}=&\sum^{3}_{j=1}\int_{\mathbb{R}^{3}}\big\{\partial^{\alpha}[
\int_{\mathbb{R}^{3}}\frac{1}{\rho}R\theta B_{ij}(\frac{v-u}{\sqrt{R\theta}})\frac{\varepsilon\partial_t\overline{G}}{M}\, dv]\,\partial^{\alpha}\partial_{x_j}\widetilde{u}_i\big\}\,dx
\\
&+\sum^{3}_{j=1}\int_{\mathbb{R}^{3}}\big\{\partial^{\alpha}[\int_{\mathbb{R}^{3}}\frac{1}{\rho}R\theta B_{ij}(\frac{v-u}{\sqrt{R\theta}})\frac{\varepsilon\sqrt{\mu}\partial_tf}{M}\, dv]\,\partial^{\alpha}\partial_{x_j}\widetilde{u}_i\big\}\,dx
\\
:=&I^{11}_{9}+I^{12}_{9}.
\end{align*}
To bound $I^{11}_9$, we note that $\partial_t\overline{G}$ has the similar expression as \eqref{4.23}, then for $|\alpha|\leq N-1$ and any $b\geq0$,
it holds by the similar arguments as \eqref{4.20} that
\begin{align}
\label{5.19a}
\|\langle v\rangle^{b}\frac{\partial^{\alpha}\partial_t\overline{G}}{\sqrt{\mu}}\|\leq C\eta_{0}\varepsilon^{1-a}.
\end{align}
Making use of this, \eqref{6.25}, \eqref{3.3}, \eqref{4.13}, \eqref{4.27} and \eqref{4.22a}, one gets
\begin{align*}
|I^{11}_9|&\leq \eta\varepsilon\sum^{3}_{j=1}\int_{\mathbb{R}^{3}}|\partial^{\alpha}\partial_{x_j}\widetilde{u}_i|^2\,dx+C_\eta\varepsilon
\sum^{3}_{j=1}\int_{\mathbb{R}^{3}}|
\int_{\mathbb{R}^{3}}\partial^{\alpha}[\frac{1}{\rho}R\theta B_{ij}(\frac{v-u}{\sqrt{R\theta}})\frac{\partial_t\overline{G}}{M}]\, dv|^2\,dx
\\
&\leq \eta\varepsilon\|\partial^{\alpha}\nabla_{x}\widetilde{u}\|^{2}+C_\eta[\eta_0(1+t)^{-\vartheta}+\varepsilon^{\frac{1}{2}-a}]\varepsilon^{2-2a}.
\end{align*}
For the term $I^{12}_{9}$, we use an integration by parts about $t$ to get
\begin{align}
\label{5.45a}
I^{12}_{9}=&\sum^{3}_{j=1}\frac{d}{dt}\int_{\mathbb{R}^{3}}\int_{\mathbb{R}^{3}}\partial^{\alpha}
\big\{[\frac{1}{\rho}R\theta B_{ij}(\frac{v-u}{\sqrt{R\theta}})\frac{\varepsilon\sqrt{\mu}}{M}]f\big\}\partial^{\alpha}\partial_{x_j}\widetilde{u}_i\,dv\,dx
\nonumber\\
&-\sum^{3}_{j=1}\int_{\mathbb{R}^{3}}\int_{\mathbb{R}^{3}}\partial^{\alpha}	\big\{\partial_t[\frac{1}{\rho}R\theta B_{ij}(\frac{v-u}{\sqrt{R\theta}})\frac{\varepsilon\sqrt{\mu}}{M}]f\big\}\partial^{\alpha}\partial_{x_j}\widetilde{u}_i\,dv\,dx
\nonumber\\
&-\sum^{3}_{j=1}\int_{\mathbb{R}^{3}}\int_{\mathbb{R}^{3}}\partial^{\alpha}
\big\{[\frac{1}{\rho}R\theta B_{ij}(\frac{v-u}{\sqrt{R\theta}})\frac{\varepsilon\sqrt{\mu}}{M}]f\big\}\partial^{\alpha}\partial_{x_j}\partial_t\widetilde{u}_{i}\,dv\,dx.
\end{align}
The last term of \eqref{5.45a} is bounded by
\begin{align*}
&\sum^{3}_{j=1}\int_{\mathbb{R}^{3}}\int_{\mathbb{R}^{3}}\partial^{\alpha}\partial_{x_j}
\big\{[\frac{1}{\rho}R\theta B_{ij}(\frac{v-u}{\sqrt{R\theta}})\frac{\varepsilon\sqrt{\mu}}{M}]f\big\}\partial^{\alpha}\partial_t\widetilde{u}_{i}\,dv\,dx
\\
\leq& \eta\varepsilon\|\partial^{\alpha}\partial_t\widetilde{u}_{i}\|^{2}+C_\eta\varepsilon
\sum^{3}_{j=1}\int_{\mathbb{R}^{3}}|\int_{\mathbb{R}^{3}}\partial^{\alpha}\partial_{x_j}
\big\{[\frac{1}{\rho}R\theta B_{ij}(\frac{v-u}{\sqrt{R\theta}})\frac{\sqrt{\mu}}{M}]f\big\}\,dv|^{2}\,dx
\\
\leq& C\eta\varepsilon\|\partial^{\alpha}
\nabla_x(\widetilde{\rho},\widetilde{u},\widetilde{\theta})\|^{2}+C_\eta\varepsilon\|\partial^{\alpha}\nabla_x f\|^{2}+
C_\eta[\eta_0(1+t)^{-\vartheta}+\varepsilon^{\frac{1}{2}-a}]\varepsilon^{2-2a}.
\end{align*}
where \eqref{6.25}, \eqref{4.86}, \eqref{3.3}, \eqref{4.13}, \eqref{4.13A}, \eqref{4.27} and \eqref{4.22a} have used in the last inequality.

The second term of \eqref{5.45a} can be handed similarly. We thereby obtain
%\begin{align*}
%I^{122}_{9}&\leq\eta\varepsilon\|\partial^{\alpha}\nabla_{x}\widetilde{u}_i\|^{2}+C_\eta\varepsilon\sum^{3}_{j=1}\int_{\mathbb{R}^{3}}|\int_{\mathbb{R}^{3}}\partial^{\alpha}
%\big\{\partial_t[\frac{1}{\rho}R\theta B_{ij}(\frac{v-u}{\sqrt{R\theta}})\frac{\sqrt{\mu}}{M}]f\big\}\,dv|^{2}\,dx
%\\
%&\leq C \eta\varepsilon\|\partial^{\alpha}\nabla_{x}\widetilde{u}_i\|^{2}+C_\eta[\eta_0(1+t)^{-\vartheta}+\varepsilon^{\frac{1}{2}-a}]\varepsilon^{2-2a}.
%\end{align*}
\begin{align*}
I^{12}_{9}\leq&\sum^{3}_{j=1}\frac{d}{dt}\int_{\mathbb{R}^{3}}\int_{\mathbb{R}^{3}}\partial^{\alpha}
\big\{[\frac{1}{\rho}R\theta B_{ij}(\frac{v-u}{\sqrt{R\theta}})\frac{\varepsilon\sqrt{\mu}}{M}]f\big\}\partial^{\alpha}\partial_{x_j}\widetilde{u}_i\,dv\,dx
\\
&+C\eta\varepsilon\|\partial^{\alpha}\nabla_x(\widetilde{\rho},\widetilde{u},\widetilde{\theta})\|^{2}
+C_\eta\varepsilon\|\partial^{\alpha}\nabla_x f\|^{2}
+C_\eta[\eta_0(1+t)^{-\vartheta}+\varepsilon^{\frac{1}{2}-a}]\varepsilon^{2-2a},
\end{align*}
which along  with the estimate $I^{11}_{9}$ shows that
\begin{align*}
I^1_{9}\leq&\sum^{3}_{j=1}\frac{d}{dt}\int_{\mathbb{R}^{3}}\int_{\mathbb{R}^{3}}
\big\{\partial^{\alpha}[\frac{1}{\rho}R\theta B_{ij}(\frac{v-u}{\sqrt{R\theta}})\frac{\varepsilon \sqrt{\mu}}{M}f]
\partial^{\alpha}\partial_{x_j}\widetilde{u}_i\big\}\,dv\,dx
\\
&+C\eta\varepsilon\|\partial^{\alpha}
\nabla_x(\widetilde{\rho},\widetilde{u},\widetilde{\theta})\|^{2}+C_\eta\varepsilon\|\partial^{\alpha}\nabla_x f\|^{2}+C_\eta
[\eta_0(1+t)^{-\vartheta}+\varepsilon^{\frac{1}{2}-a}]\varepsilon^{2-2a}.
\end{align*}
Using \eqref{6.25}, \eqref{4.86}, \eqref{3.3}, \eqref{4.13}, \eqref{4.13A} and \eqref{4.22a} again,
we can bound $I^{2}_{9}$ and $I^{3}_{9}$ by
\begin{align*}
|I^{2}_{9}|+|I^{3}_{9}|&\leq C\eta\varepsilon\|\partial^{\alpha}\nabla_{x}\widetilde{u}\|^2+ C_\eta\varepsilon\|\partial^{\alpha}\nabla_xf\|^2+C_\eta
[\eta_0(1+t)^{-\vartheta}+\varepsilon^{\frac{1}{2}-a}]\varepsilon^{2-2a}.
\end{align*}
As for the term $I^{4}_{9}$, we use \eqref{4.9} to write
\begin{align*}
I^{4}_{9}&=-\sum^{3}_{j=1}\int_{\mathbb{R}^{3}}\int_{\mathbb{R}^{3}}\partial^{\alpha}[\frac{1}{\rho}R\theta B_{ij}(\frac{v-u}{\sqrt{R\theta}})\frac{\sqrt{\mu}}{M}\Gamma(\frac{G}{\sqrt{\mu}},\frac{G}{\sqrt{\mu}})]\, dv\,\partial^{\alpha}\partial_{x_j}\widetilde{u}_i\,dx.
\end{align*}
If $|\alpha|=0$ in $I^{4}_{9}$, we get by \eqref{4.30}, \eqref{6.25}, \eqref{4.19}, \eqref{4.13}, \eqref{4.22a} and \eqref{4.15} that
\begin{align*}
|I^4_{9}|&\leq C\sum^{3}_{j=1}\int_{\mathbb{R}^{3}}|\frac{G}{\sqrt{\mu}}|_{2}|\frac{G}{\sqrt{\mu}}|_{\nu}|\partial_{x_j}\widetilde{u}_i|\,dx
\\
&\leq  \eta\varepsilon\|\nabla_{x}\widetilde{u}_i\|^2+
C_\eta\frac{1}{\varepsilon}\sup_{x\in\mathbb{R}^3}(|\frac{\overline{G}}{\sqrt{\mu}}|^2_{2}
+|f|^2_{2})(\|\frac{\overline{G}}{\sqrt{\mu}}\|^2_{\nu}+\|f\|^2_{\nu})
\\
&\leq  \eta\varepsilon\|\nabla_{x}\widetilde{u}\|^2
+C_\eta(\eta_{0}+\varepsilon^{\frac{1}{2}-a})\mathcal{D}_N(t)+C_\eta[\eta_0(1+t)^{-\vartheta}+\varepsilon^{\frac{1}{2}-a}]\varepsilon^{2-2a}.
\end{align*}
Similar estimate also holds for the cases $1\leq|\alpha|\leq N-1$, it follows that
\begin{align*}
|I^{4}_{9}|
&\leq  C\eta\varepsilon\|\partial^{\alpha}\nabla_{x}\widetilde{u}\|^2
+C_\eta(\eta_{0}+\varepsilon^{\frac{1}{2}-a})\mathcal{D}_N(t)+C_\eta[\eta_0(1+t)^{-\vartheta}+\varepsilon^{\frac{1}{2}-a}]\varepsilon^{2-2a}.
\end{align*}
Therefore, collecting the estimates form $I^{1}_{9}$ to $I^{4}_{9}$, one gets
\begin{align}
\label{6.46}
I_{9}\leq&-\sum^{3}_{j=1}\frac{d}{dt}\int_{\mathbb{R}^{3}}\int_{\mathbb{R}^{3}}
\big\{\partial^{\alpha}\partial_{x_j}[\frac{1}{\rho}R\theta B_{ij}(\frac{v-u}{\sqrt{R\theta}})\frac{\varepsilon \sqrt{\mu}}{M}f]
\partial^{\alpha}\widetilde{u}_i\big\}\,dv\,dx+C\eta\varepsilon\|\partial^{\alpha}\nabla_x(\widetilde{\rho},\widetilde{u},\widetilde{\theta})\|^{2}
\nonumber\\
&+C_\eta\varepsilon\|\partial^{\alpha}\nabla_x f\|^{2}
+C_\eta(\eta_{0}+\varepsilon^{\frac{1}{2}-a})\mathcal{D}_N(t)+C_\eta[\eta_0(1+t)^{-\vartheta}+\varepsilon^{\frac{1}{2}-a}]\varepsilon^{2-2a}.
\end{align}
The term $I_{10}$ in \eqref{6.45} can be treated in the similar way as $I_9$, so it holds that
\begin{align*}
I_{10}\leq& \sum^{3}_{j=1}\frac{d}{dt}\int_{\mathbb{R}^{3}}\int_{\mathbb{R}^{3}}\partial^{\alpha}
\big\{[\partial_{x_j}(\frac{1}{\rho})R\theta B_{ij}(\frac{v-u}{\sqrt{R\theta}})\frac{\varepsilon\sqrt{\mu}}{M}]f\big\}\partial^{\alpha}\widetilde{u}_i\,dv\,dx
+C\eta\varepsilon\|\partial^{\alpha}\nabla_x(\widetilde{\rho},\widetilde{u},\widetilde{\theta})\|^{2}
\nonumber\\
&+C_\eta\varepsilon\|\partial^{\alpha}\nabla_x f\|^{2}
+C(\eta_{0}+\varepsilon^{\frac{1}{2}-a})\mathcal{D}_N(t)+C_\eta[\eta_0(1+t)^{-\vartheta}+\varepsilon^{\frac{1}{2}-a}]\varepsilon^{2-2a}.
\end{align*}
Hence, substituting the estimates $I_{9}$ and $I_{10}$ into \eqref{6.45} yields
\begin{align}
\label{5.27A}	
-&(\partial^{\alpha}[\frac{1}{\rho}\int_{\mathbb{R}^{3}}v_{i}v\cdot\nabla_{x}L^{-1}_{M}\Theta\,dv],\partial^{\alpha}\widetilde{u}_i)
\nonumber\\
\leq&-\sum^{3}_{j=1}\frac{d}{dt}\int_{\mathbb{R}^{3}}\int_{\mathbb{R}^{3}}
\big\{\partial^{\alpha}[\frac{1}{\rho}\partial_{x_j}(R\theta B_{ij}(\frac{v-u}{\sqrt{R\theta}})\frac{\varepsilon \sqrt{\mu}}{M}f)]
\partial^{\alpha}\widetilde{u}_i\big\}\,dv\,dx+C\eta\varepsilon\|\partial^{\alpha}\nabla_x(\widetilde{\rho},\widetilde{u},\widetilde{\theta})\|^{2}
\nonumber\\
&+C_\eta\varepsilon\|\partial^{\alpha}\nabla_x f\|^{2}
+C_\eta(\eta_{0}+\varepsilon^{\frac{1}{2}-a})\mathcal{D}_N(t)+C_\eta[\eta_0(1+t)^{-\vartheta}+\varepsilon^{\frac{1}{2}-a}]\varepsilon^{2-2a}.
\end{align}
In summary, substituting \eqref{5.15A}, \eqref{5.17A}, \eqref{5.18A}, \eqref{6.44} and
\eqref{5.27A} into \eqref{6.43} and summing $i$ from 1 to 3, we have established that for $|\alpha|\leq N-1$ and  any small $\eta>0$,
\begin{align}
\label{6.47}
&\frac{1}{2}\frac{d}{dt}\|\partial^{\alpha}\widetilde{u}\|^{2}+\frac{1}{2}\frac{d}{dt}(\frac{1}{\rho}\partial^{\alpha}\widetilde{E},\partial^{\alpha}\widetilde{E})
+\frac{1}{2}\frac{d}{dt}(\frac{1}{\rho}\partial^{\alpha}\widetilde{B},\partial^{\alpha}\widetilde{B})
\nonumber\\
&+\sum^{3}_{i,j=1}\frac{d}{dt}\int_{\mathbb{R}^{3}}\int_{\mathbb{R}^{3}}
\big\{\partial^{\alpha}[\frac{1}{\rho}\partial_{x_j}(R\theta B_{ij}(\frac{v-u}{\sqrt{R\theta}})\frac{\varepsilon \sqrt{\mu}f}{M})]
\partial^{\alpha}\widetilde{u}_i\big\}\,dv\,dx
\nonumber\\
&+(\frac{2\bar{\theta}}{3\bar{\rho}}\partial^{\alpha}\nabla_{x}\widetilde{\rho},\partial^{\alpha}\widetilde{u})
+(\frac{2}{3}\partial^{\alpha}\nabla_{x}\widetilde{\theta},\partial^{\alpha}\widetilde{u})
+c\varepsilon\|\partial^{\alpha}\nabla_{x}\widetilde{u}\|^{2}
\nonumber\\
&\leq C\eta\varepsilon\|\partial^{\alpha}
\nabla_x(\widetilde{\rho},\widetilde{u},\widetilde{\theta})\|^{2}+C_\eta\varepsilon\|\partial^{\alpha}\nabla_x f\|^{2}
\nonumber\\
&\hspace{0.5cm}
+C_\eta(\eta_{0}+\varepsilon^{\frac{1}{2}-a})\mathcal{D}_N(t)+C_\eta[\eta_0(1+t)^{-\vartheta}+\eta_0\varepsilon^{a}+\varepsilon^{\frac{1}{2}-a}]\varepsilon^{2-2a}.
\end{align}
Here we have used the fact that
\begin{align*}
\varepsilon&\sum^{3}_{i,j=1}(\frac{1}{\rho}\mu(\theta)(\partial^{\alpha}\partial_{x_{j}}\widetilde{u}_{i}+\partial^{\alpha}\partial_{x_{i}}\widetilde{u}_{j}
-\frac{2}{3}\delta_{ij}\nabla_{x}\cdot \partial^{\alpha}\widetilde{u}),\partial^{\alpha}\partial_{x_{j}}\widetilde{u}_{i})
%\\
%=&\varepsilon\sum^{3}_{i,j=1}(\frac{\mu(\theta)}{\rho}\partial^{\alpha}\partial_{x_{j}}\widetilde{u}_{i},\partial^{\alpha}\partial_{x_{j}}\widetilde{u}_{i})
%+\frac{1}{3}\varepsilon\sum^{3}_{i,j=1}(\frac{\mu(\theta)}{\rho}\partial^{\alpha}\partial_{x_{j}}\widetilde{u}_{j},\partial^{\alpha}\partial_{x_{i}}\widetilde{u}_{i})
%\\
%&-\varepsilon\sum^{3}_{i,j=1}(\partial_{x_{j}}[\frac{\mu(\theta)}{\rho}]\partial^{\alpha}\partial_{x_{i}}\widetilde{u}_{j},\partial^{\alpha}\widetilde{u}_{i})
%+\varepsilon\sum^{3}_{i,j=1}(\partial_{x_{i}}[\frac{\mu(\theta)}{\rho}]\partial^{\alpha}\partial_{x_{j}}\widetilde{u}_{j},\partial^{\alpha}\widetilde{u}_{i})
\\
&\geq c\varepsilon\|\partial^{\alpha}\nabla_{x}\widetilde{u}\|^{2}-C[\eta_0(1+t)^{-\vartheta}+\varepsilon^{\frac{1}{2}-a}]\varepsilon^{2-2a}.
\end{align*}

\medskip
\noindent{\it Step 3. Estimate on $\|\partial^{\alpha}\widetilde{\theta}\|^2$}. As before, applying $\partial^{\alpha}$ with $|\alpha|\leq N-1$
to the third equation of \eqref{4.2} and taking the inner product of the resulting equation with 
$\frac{1}{\bar{\theta}}\partial^{\alpha}\widetilde{\theta}$, one gets
\begin{align}
\label{6.48}
(\partial_{t}&\partial^{\alpha}\widetilde{\theta},\frac{1}{\bar{\theta}}\partial^{\alpha}\widetilde{\theta})+(\frac{2}{3}\bar{\theta}\nabla_{x}\cdot \partial^{\alpha}\widetilde{u},\frac{1}{\bar{\theta}}\partial^{\alpha}\widetilde{\theta})+\sum_{1\leq \alpha_{1}\leq \alpha}C_{\alpha}^{\alpha_{1}}(\frac{2}{3}\partial^{\alpha_{1}}\bar{\theta}\nabla_{x}\cdot \partial^{\alpha-\alpha_{1}}\widetilde{u},\frac{1}{\bar{\theta}}\partial^{\alpha}\widetilde{\theta})
\nonumber\\
=&-(\partial^{\alpha}(u\cdot\nabla_{x}\widetilde{\theta}),\frac{1}{\bar{\theta}}\partial^{\alpha}\widetilde{\theta})-(\partial^{\alpha}(\widetilde{u}\cdot\nabla_{x}\bar{\theta}),\frac{1}{\bar{\theta}}\partial^{\alpha}\widetilde{\theta})
-(\frac{2}{3}\partial^{\alpha}(\widetilde{\theta}\nabla_{x}\cdot u),\frac{1}{\bar{\theta}}\partial^{\alpha}\widetilde{\theta})
\nonumber\\
&\quad
+\varepsilon\sum^{3}_{i,j=1}(\partial^{\alpha}[\frac{1}{\rho}\mu(\theta) \partial_{x_{j}}u_{i}D_{ij}],\frac{1}{\bar{\theta}}\partial^{\alpha}\widetilde{\theta})
+\varepsilon\sum^{3}_{j=1}(\partial^{\alpha}[\frac{1}{\rho}\partial_{x_{j}}(\kappa(\theta)\partial_{x_{j}}\theta)],\frac{1}{\bar{\theta}}\partial^{\alpha}\widetilde{\theta})
\nonumber\\
&\quad\quad\quad-(\partial^{\alpha}[\frac{1}{\rho}\int_{\mathbb{R}^{3}} \frac{1}{2}|v|^{2} v\cdot\nabla_{x}L^{-1}_{M}\Theta\,dv],\frac{1}{\bar{\theta}}\partial^{\alpha}\widetilde{\theta})
\nonumber\\
&\quad\quad\quad\quad\quad+(\partial^{\alpha}[\frac{1}{\rho}u\cdot\int_{\mathbb{R}^{3}} v\otimes v\cdot\nabla_{x}L^{-1}_{M}\Theta\,dv]
,\frac{1}{\bar{\theta}}\partial^{\alpha}\widetilde{\theta}).
\end{align}
Here if $|\alpha|=0$, the third term on the LHS of \eqref{6.48} vanishes and if  $|\alpha|\geq 1$, it exists.

We will estimate \eqref{6.48} term by term.
%For the first term on the LHS of \eqref{6.48}, one gets
%\begin{align*}
%(\partial_{t}\partial^{\alpha}\widetilde{\theta},\frac{1}{\bar{\theta}}\partial^{\alpha}\widetilde{\theta})
%&\geq\frac{1}{2}\frac{d}{dt}(\partial^{\alpha}\widetilde{\theta},\frac{1}{\bar{\theta}}\partial^{\alpha}\widetilde{\theta})-C\|\partial_{t}\bar{\theta}\|_{L^{\infty}}\|\partial^{\alpha}\widetilde{\theta}\|^2
%\\
%&\geq \frac{1}{2}\frac{d}{dt}(\partial^{\alpha}\widetilde{\theta},\frac{1}{\bar{\theta}}\partial^{\alpha}\widetilde{\theta})-C\eta_0(1+t)^{-\vartheta}\varepsilon^{2-2a},
%\end{align*}
%where \eqref{3.3} and  \eqref{4.13} have used in the last inequality.
The third term on the LHS of \eqref{6.48} has the same bound as \eqref{6.36}.
The first four terms on the RHS of \eqref{6.48} are bounded by
\begin{align*}
C\eta\varepsilon\|\partial^{\alpha}\nabla_{x}\cdot\widetilde{u}\|^2
+C_\eta[\eta_0(1+t)^{-\vartheta}+\varepsilon^{\frac{1}{2}-a}]\varepsilon^{2-2a}.
\end{align*}
The fifth term on the RHS of \eqref{6.48}  has the same structure as the viscosity coefficient term in \eqref{6.43}.
Therefore, we can perform the similar arguments as \eqref{6.44} to get
\begin{align*}
&\varepsilon\sum^{3}_{j=1}(\partial^{\alpha}[\frac{1}{\rho}\partial_{x_{j}}(\kappa(\theta)\partial_{x_{j}}\theta)],\frac{1}{\bar{\theta}}\partial^{\alpha}\widetilde{\theta})
\\
&\leq-\varepsilon\sum^{3}_{j=1}(\frac{1}{\rho}\kappa(\theta)\partial^{\alpha}\partial_{x_{j}}\widetilde{\theta},
\frac{1}{\bar{\theta}}\partial^{\alpha}\partial_{x_{j}}\widetilde{\theta})
+C[\eta_0(1+t)^{-\vartheta}+\eta_0\varepsilon^a+\varepsilon^{\frac{1}{2}-a}]\varepsilon^{2-2a}.
\end{align*}
%For the fifth term on the RHS of \eqref{6.48}.
%Recall the definition $D_{ij}$ in \eqref{2.14}, one has from \eqref{3.3} and \eqref{4.13} that
%\begin{align*}
%\varepsilon\sum^{3}_{i,j=1}(\partial^{\alpha}[\frac{1}{\rho}\mu(\theta) \partial_{x_{j}}u_{i}D_{ij}],\frac{1}{\bar{\theta}}\partial^{\alpha}\widetilde{\theta})
%\leq C[\eta_0(1+t)^{-\vartheta}+\varepsilon^{\frac{1}{2}-a}]\varepsilon^{2-2a}.
%\end{align*}
It remains to compute the last two terms in \eqref{6.48}. By the self-adjoint property of $L^{-1}_{M}$, \eqref{2.5}, \eqref{2.17} and \eqref{2.18}, one has
\begin{align}
\label{6.21}
\int_{\mathbb{R}^{3}}& (\frac{1}{2}v_{i}|v|^{2}-v_{i}u\cdot v)L^{-1}_{M}\Theta \,dv=
\int_{\mathbb{R}^{3}} L^{-1}_{M}\{P_{1}(\frac{1}{2}v_{i}|v|^{2}-v_{i}u\cdot v)M\}\frac{\Theta}{M}\,dv
\nonumber\\
=&\int_{\mathbb{R}^{3}} L^{-1}_{M}\{(R\theta)^{\frac{3}{2}}\hat{A}_{i}(\frac{v-u}{\sqrt{R\theta}})M\}\frac{\Theta}{M}\,dv
=(R\theta)^{\frac{3}{2}}\int_{\mathbb{R}^{3}}A_{i}(\frac{v-u}{\sqrt{R\theta}})\frac{\Theta}{M}\,dv.
\end{align}
In view of \eqref{6.22} and \eqref{6.21}, the last two terms of \eqref{6.48} are equivalent to
\begin{align*}
&{-\sum^{3}_{i=1}(\partial^{\alpha}[\frac{1}{\rho}\partial_{x_i}(\int_{\mathbb{R}^{3}}
(\frac{1}{2}|v|^{2}v_i-u\cdot vv_{i})L^{-1}_{M}\Theta\,dv)],\frac{1}{\bar{\theta}}\partial^{\alpha}\widetilde{\theta})}
\\
&{-\sum^{3}_{i=1}
(\partial^{\alpha}[\frac{1}{\rho}\int_{\mathbb{R}^{3}}
\partial_{x_i}u\cdot vv_{i}L^{-1}_{M}\Theta\,dv],\frac{1}{\bar{\theta}}\partial^{\alpha}\widetilde{\theta})}
\\
=&-\sum^{3}_{i=1}(\partial^{\alpha}[\frac{1}{\rho}\partial_{x_i}((R\theta)^{\frac{3}{2}}\int_{\mathbb{R}^{3}}A_{i}(\frac{v-u}{\sqrt{R\theta}})\frac{\Theta}{M}\,dv)],\frac{1}{\bar{\theta}}\partial^{\alpha}\widetilde{\theta})
\\
&-\sum^{3}_{i,j=1}	(\partial^{\alpha}[\frac{1}{\rho}
\partial_{x_i}u_j R\theta\int_{\mathbb{R}^{3}}B_{ij}(\frac{v-u}{\sqrt{R\theta}})\frac{\Theta}{M}\,dv],\frac{1}{\bar{\theta}}\partial^{\alpha}\widetilde{\theta})
\\
\leq	
&\frac{d}{dt}\widetilde{E}(t)+C\eta\varepsilon\|\partial^{\alpha}
\nabla_x(\widetilde{\rho},\widetilde{u},\widetilde{\theta})\|^{2}+C_\eta\varepsilon\|\partial^{\alpha}\nabla_x f\|^{2}
\\
&+C_\eta(\eta_{0}+\varepsilon^{\frac{1}{2}-a})\mathcal{D}_N(t)
+C_\eta[\eta_0(1+t)^{-\vartheta}+\varepsilon^{\frac{1}{2}-a}]\varepsilon^{2-2a}.
\end{align*}
where in the last inequality we used the similar calculations as \eqref{5.27A} and 
\begin{align}
	\label{5.30B}
\widetilde{E}(t)=&-\sum^{3}_{i=1}\int_{\mathbb{R}^{3}}\int_{\mathbb{R}^{3}}
\big\{\partial^{\alpha}[\frac{1}{\rho}\partial_{x_i}
((R\theta)^{\frac{3}{2}}A_{i}(\frac{v-u}{\sqrt{R\theta}})\frac{\varepsilon\sqrt{\mu}}{M}f)]\frac{1}{\bar{\theta}}\partial^{\alpha}\widetilde{\theta}\big\}\,dv\,dx
\nonumber\\	
&-\sum^{3}_{i,j=1}\int_{\mathbb{R}^{3}}\int_{\mathbb{R}^{3}}
\big\{\partial^{\alpha}[\frac{1}{\rho}
\partial_{x_i}u_j R\theta B_{ij}(\frac{v-u}{\sqrt{R\theta}})\frac{\varepsilon\sqrt{\mu}}{M}f]\frac{1}{\bar{\theta}}\partial^{\alpha}\widetilde{\theta}\big\}\,dv\,dx.
\end{align}
In summary, putting all the above estimates into \eqref{6.48}, we have established that for $|\alpha|\leq N-1$ and any small $\eta>0$,
\begin{align}
\label{6.49}
&\frac{1}{2}\frac{d}{dt}(\partial^{\alpha}\widetilde{\theta},\frac{1}{\bar{\theta}}\partial^{\alpha}\widetilde{\theta})+(\frac{2}{3}\nabla_{x}\cdot \partial^{\alpha}\widetilde{u},\partial^{\alpha}\widetilde{\theta})
+\varepsilon(\frac{1}{\rho}\kappa(\theta)\partial^{\alpha}\nabla_{x}\widetilde{\theta},\frac{1}{\bar{\theta}}\partial^{\alpha}\nabla_{x}\widetilde{\theta})
\nonumber\\	
&\leq \frac{d}{dt}\widetilde{E}(t)+C\eta\varepsilon\|\partial^{\alpha}
\nabla_x(\widetilde{\rho},\widetilde{u},\widetilde{\theta})\|^{2}+C_\eta\varepsilon\|\partial^{\alpha}\nabla_x f\|^{2}
\nonumber\\
&\hspace{0.5cm}+C_\eta(\eta_{0}+\varepsilon^{\frac{1}{2}-a})\mathcal{D}_N(t)
+C_\eta[\eta_0(1+t)^{-\vartheta}+\eta_0\varepsilon^{a}+\varepsilon^{\frac{1}{2}-a}]\varepsilon^{2-2a}.
\end{align}

\medskip
\noindent{\it Step 4. Estimate on $\|\partial^{\alpha}(\widetilde{\rho},\widetilde{u},\widetilde{\theta})\|^2$}. 
For $|\alpha|\leq N-1$, we get by integration by parts, \eqref{3.3} and \eqref{4.13} that
\begin{align*}
&|(\nabla_x\cdot\partial^{\alpha}\widetilde{u},\frac{2\bar{\theta}}{3\bar{\rho}}\partial^{\alpha}\widetilde{\rho})+	(\frac{2\bar{\theta}}{3\bar{\rho}}\partial^{\alpha}\nabla_{x}\widetilde{\rho},\partial^{\alpha}\widetilde{u})+(\frac{2}{3}\partial^{\alpha}\nabla_{x}\widetilde{\theta},\partial^{\alpha}\widetilde{u})
+(\frac{2}{3}\nabla_{x}\cdot \partial^{\alpha}\widetilde{u},\partial^{\alpha}\widetilde{\theta})|
\\
&=|(\partial^{\alpha}\widetilde{u},\nabla_x(\frac{2\bar{\theta}}{3\bar{\rho}})\partial^{\alpha}\widetilde{\rho})|\leq C\|\nabla_x(\bar{\rho},\bar{\theta})\|_{L^{\infty}}
\|\partial^{\alpha}\widetilde{u}\|\|\partial^{\alpha}\widetilde{\rho}\|\leq C\eta_0(1+t)^{-\vartheta}\varepsilon^{2-2a}.
\end{align*}
Adding \eqref{6.42}, \eqref{6.47} and \eqref{6.49} together, and using the aforementioned estimate,
then the summation of the resulting equation over $|\alpha|$ through a suitable linear combination gives
\begin{align}
\label{6.50}
\frac{1}{2}&\sum_{|\alpha|\leq N-1}\frac{d}{dt}\int_{\mathbb{R}^{3}}( \frac{2\bar{\theta}}{3\bar{\rho}^{2}}|\partial^{\alpha}\widetilde{\rho}|^{2}+|\partial^{\alpha}\widetilde{u}|^{2}
+\frac{1}{\bar{\theta}}|\partial^{\alpha}\widetilde{\theta}|^{2}+\frac{1}{\rho}|\partial^{\alpha}\widetilde{E}|^2
+\frac{1}{\rho}|\partial^{\alpha}\widetilde{B}|^2)\,dx
\nonumber\\
&+\frac{d}{dt}E_\alpha(t)+c\varepsilon\sum_{1\leq |\alpha|\leq N}\|\partial^{\alpha}(\widetilde{u},\widetilde{\theta})\|^{2}
\nonumber\\	
\leq&C\eta\varepsilon\sum_{1\leq |\alpha|\leq N}\|\partial^{\alpha}
(\widetilde{\rho},\widetilde{u},\widetilde{\theta})\|^{2}+C_\eta\varepsilon\sum_{1\leq |\alpha|\leq N}
\|\partial^{\alpha}f\|^{2}
\nonumber\\
&+C_\eta(\eta_{0}+\varepsilon^{\frac{1}{2}-a})\mathcal{D}_N(t)
+C_\eta[\eta_0(1+t)^{-\vartheta}+\eta_0\varepsilon^{a}+\varepsilon^{\frac{1}{2}-a}]\varepsilon^{2-2a}.
\end{align}
Here $E_\alpha(t)$ is given by
\begin{align}
\label{6.51}
E_\alpha(t)
=&\sum_{|\alpha|\leq N-1}(\sum^{3}_{i,j=1}\int_{\mathbb{R}^{3}}\int_{\mathbb{R}^{3}}
\big\{\partial^{\alpha}[\frac{1}{\rho}\partial_{x_j}(R\theta B_{ij}(\frac{v-u}{\sqrt{R\theta}})\frac{\varepsilon \sqrt{\mu}}{M}f)]
\partial^{\alpha}\widetilde{u}_i\big\}\,dv\,dx-\widetilde{E}(t)).
\end{align}
On the other hand, we can easily get by \eqref{4.86a} that 
\begin{align}	
\label{6.53}
&\varepsilon\sum_{|\alpha|\leq N-1} \frac{d}{dt}(\partial^{\alpha}\widetilde{u},\nabla_x\partial^{\alpha}\widetilde{\rho})
+c\varepsilon\sum_{1\leq |\alpha|\leq N}\|\partial^{\alpha}\widetilde{\rho}\|^2
\nonumber\\
&\leq C\varepsilon\sum_{1\leq |\alpha|\leq N}
(\|(\partial^{\alpha}\widetilde{u},\partial^{\alpha}\widetilde{\theta})\|^2+\|\partial^{\alpha}f\|^2)
+C\varepsilon\varepsilon^{2-2a}.
\end{align}
In summary, adding $\eqref{6.50}\times C_3$ to \eqref{6.53} and firstly choosing $C_3>1$ sufficiently large,
then taking $\eta>0$ small enough, one gets
\begin{align}
\label{6.54}
\frac{1}{2}&C_3\sum_{|\alpha|\leq N-1}\frac{d}{dt}\int_{\mathbb{R}^{3}}( \frac{2\bar{\theta}}{3\bar{\rho}^{2}}|\partial^{\alpha}\widetilde{\rho}|^{2}+|\partial^{\alpha}\widetilde{u}|^{2}
+\frac{1}{\bar{\theta}}|\partial^{\alpha}\widetilde{\theta}|^{2}+\frac{1}{\rho}|\partial^{\alpha}\widetilde{E}|^2+\frac{1}{\rho}|\partial^{\alpha}\widetilde{B}|^2)\,dx
\nonumber\\
&+C_3\frac{d}{dt}E_\alpha(t)+\varepsilon\sum_{|\alpha|\leq N-1} \frac{d}{dt}(\partial^{\alpha}\widetilde{u},\nabla_x\partial^{\alpha}\widetilde{\rho})
+c\varepsilon\sum_{1\leq |\alpha|\leq N}\|\partial^{\alpha}(\widetilde{\rho},\widetilde{u},\widetilde{\theta})\|^{2}
\nonumber\\	
\leq&C\varepsilon\sum_{1\leq |\alpha|\leq N}\|\partial^{\alpha}f\|^{2}
+C(\eta_{0}+\varepsilon^{\frac{1}{2}-a})\mathcal{D}_N(t)+
C[\eta_0(1+t)^{-\vartheta}+\eta_0\varepsilon^{a}+\varepsilon^{\frac{1}{2}-a}]\varepsilon^{2-2a}.
\end{align}
By the Cauchy-Schwarz inequality, \eqref{3.3}, \eqref{4.13} and \eqref{4.22a}, we get
\begin{align*}
	|E_\alpha(t)|\leq C\eta 
	\sum_{|\alpha|\leq N-1}\|\partial^{\alpha}(\widetilde{\rho},\widetilde{u},\widetilde{\theta})\|^2
	+C_{\eta}\varepsilon^2\sum_{|\alpha|=N}(\|\partial^{\alpha}(\widetilde{u},\widetilde{\theta})\|^2+\|\partial^{\alpha}f\|^2)+
	C_{\eta}\varepsilon\varepsilon^{2-2a}.
\end{align*}
Integrating \eqref{6.54} with respect to $t$ and using $\mathcal{E}_{N}(0)\leq C\eta^2_{0}\varepsilon^{2}$ and $\vartheta>1$
as well as the aforementioned estimate, we can prove that \eqref{6.34} holds true.
%\begin{align*}
%\varepsilon\sum_{|\alpha|\leq N-1}(\partial^{\alpha}\widetilde{u},\nabla_x\partial^{\alpha}\widetilde{\rho})
%&\leq  \sum_{|\alpha|\leq N-1}(\eta\|\partial^{\alpha}\widetilde{u}\|^2+C_{\eta}\varepsilon^2\|\nabla_x\partial^{\alpha}\widetilde{\rho}\|^2)
%\\
%&\leq \eta\sum_{|\alpha|\leq N-1}\|\partial^{\alpha}\widetilde{u}\|^2
%+C_{\eta}\varepsilon^2\sum_{|\alpha|=N}\|\partial^{\alpha}\widetilde{\rho}\|^2+
%C_{\eta}\varepsilon^2\varepsilon^{2-2a},
%\end{align*}
%and
This completes the proof of Lemma \ref{lem6.6}.
\end{proof}

\subsubsection{Space derivative estimates on non-fluid part up to $(N-1)$-order}\label{seca.6.2.2} 
Next we make use of the microscopic equation \eqref{4.11} to derive the 
space derivative estimates  up to $(N-1)$-order for the non-fluid part $f$.
We should emphasize that the fact that $f\in (\ker\mathcal{L})^{\perp}$ is crucial in the estimates.
\begin{lemma}
\label{lem6.7}
Under the conditions listed in Lemma \ref{lem6.6}, it holds that
\begin{align}
\label{6.55}
&\sum_{|\alpha|\leq N-1}\|\partial^{\alpha}f(t)\|^{2}+c\frac{1}{\varepsilon}\sum_{|\alpha|\leq N-1}\int^t_0\|\partial^{\alpha}f(s)\|_{\nu}^{2}\,ds
\nonumber\\
&\hspace{0.5cm}\leq C\varepsilon\sum_{1\leq|\alpha|\leq N}\int^t_0
\{\|\partial^{\alpha}(\widetilde{u},\widetilde{\theta})(s)\|^{2}+\|\partial^{\alpha}f(s)\|_\nu^{2}\}\,ds
\nonumber\\
&\hspace{1cm}+C(\eta_{0}+\varepsilon^{\frac{1}{2}-a})\int^t_0\mathcal{D}_N(s)\,ds
+C(\eta_{0}+\varepsilon^{\frac{1}{2}-a}t)\varepsilon^{2-2a}.
\end{align}
\end{lemma}
\begin{proof}
Applying $\partial^{\alpha}$ with $|\alpha|\leq N-1$ to \eqref{4.11} 
and then taking the inner product of the resulting identity with $\partial^{\alpha}f$, one gets
\begin{align}
\label{6.56}
\frac{1}{2}&\frac{d}{dt}\|\partial^{\alpha}f\|^{2}+(v\cdot\nabla_{x}\partial^{\alpha}f,\partial^{\alpha}f)
-\frac{1}{\varepsilon}(\mathcal{L}\partial^{\alpha}f,\partial^{\alpha}f)
\nonumber\\
=&(\frac{\partial^{\alpha}[(E+v\times B)\cdot\nabla_{v}\overline{G}]}{\sqrt{\mu}},\partial^{\alpha}f)
+(\frac{\partial^{\alpha}[(E+v\times B)\cdot\nabla_{v}(\sqrt{\mu}f)]}{\sqrt{\mu}},\partial^{\alpha}f)
\nonumber\\
&+\frac{1}{\varepsilon}(\partial^{\alpha}\Gamma(\frac{M-\mu}{\sqrt{\mu}},f)
+\partial^{\alpha}\Gamma(f,\frac{M-\mu}{\sqrt{\mu}}),\partial^{\alpha}f)
+\frac{1}{\varepsilon}(\partial^{\alpha}\Gamma(\frac{G}{\sqrt{\mu}},\frac{G}{\sqrt{\mu}}),\partial^{\alpha}f)
\nonumber\\
&+(\frac{\partial^{\alpha}P_{0}(v\sqrt{\mu}\cdot\nabla_{x}f)}{\sqrt{\mu}},\partial^{\alpha}f)-(\frac{\partial^{\alpha}P_{1}(v\cdot\nabla_{x}\overline{G})}{\sqrt{\mu}},\partial^{\alpha}f)
-(\frac{\partial^{\alpha}\partial_{t}\overline{G}}{\sqrt{\mu}},\partial^{\alpha}f)
\nonumber\\
&-(\frac{1}{\sqrt{\mu}}\partial^{\alpha}P_{1}\big\{v\cdot(\frac{|v-u|^{2}\nabla_{x}\widetilde{\theta}}{2R\theta^{2}}
+\frac{(v-u)\cdot\nabla_{x}\widetilde{u}}{R\theta})M\big\},\partial^{\alpha}f).
\end{align}
Let's now compute \eqref{6.56} term by term.
The second term on the LHS of \eqref{6.56} vanishes by integration by parts.
For the third term on the LHS of \eqref{6.56}, we have from \eqref{4.28} that
\begin{equation*}
-\frac{1}{\varepsilon}(\mathcal{L}\partial^{\alpha}f,\partial^{\alpha}f)\geq c_1\frac{1}{\varepsilon}\|\partial^{\alpha}f\|^{2}_{\nu}.
\end{equation*}
For the first term on the RHS of \eqref{6.56}, we get by \eqref{4.19}, \eqref{1.4}, \eqref{3.3} and \eqref{4.13} that
\begin{align}
\label{5.38A}
|(\frac{\partial^{\alpha}[(E+v\times B)\cdot\nabla_{v}\overline{G}]}{\sqrt{\mu}},\partial^{\alpha}f)|
&\leq \eta\frac{1}{\varepsilon}\|\partial^{\alpha}f\|^{2}+C_\eta\varepsilon\|\frac{\partial^{\alpha}[(E+v\times B)\cdot\nabla_{v}\overline{G}]}{\sqrt{\mu}}\|^2
\nonumber\\
&\leq C\eta\frac{1}{\varepsilon}\|\partial^{\alpha}f\|_{\nu}^{2}+C_{\eta}[\eta_0(1+t)^{-\vartheta}+\varepsilon^{\frac{1}{2}-a}]\varepsilon^{2-2a}.
\end{align}
By \eqref{5.72}, \eqref{5.27} and \eqref{5.46}, it is clear that
the second, third, fourth terms on the RHS of \eqref{6.56} are bounded by
\begin{equation*}
C(\eta_{0}+\varepsilon^{\frac{1}{2}-a})\mathcal{D}_{N}(t)+C\varepsilon\varepsilon^{2-2a}.
\end{equation*}
For the fifth term on the RHS of \eqref{6.56}, making use of \eqref{2.5}, \eqref{4.25},
the similar arguments as \eqref{4.36}, \eqref{1.4}, \eqref{3.3}, \eqref{4.13} and \eqref{4.22a}, one gets
\begin{align*}
|(\frac{\partial^{\alpha}P_{0}(v\sqrt{\mu}\cdot\nabla_{x}f)}{\sqrt{\mu}},\partial^{\alpha}f)|
&=|(\frac{1}{\sqrt{\mu}}\sum_{i=0}^{4}\partial^{\alpha}[\langle v\sqrt{\mu}\cdot\nabla_{x}f,\frac{\chi_{i}}{M}\rangle\chi_{i}],\partial^{\alpha}f)|
\\	
&\leq C\|\partial^{\alpha}f\|\|\frac{1}{\sqrt{\mu}}\sum_{i=0}^{4}\partial^{\alpha}[\langle v\sqrt{\mu}\cdot\nabla_{x}f,\frac{\chi_{i}}{M}\rangle\chi_{i}]\|
\\
&\leq  \eta\frac{1}{\varepsilon}\|\partial^{\alpha}f\|_{\nu}^{2}
+C_{\eta}\varepsilon\|\partial^{\alpha}\nabla_{x}f\|^{2}+C_{\eta}[\eta_0(1+t)^{-\vartheta}+\varepsilon^{\frac{1}{2}-a}]\varepsilon^{2-2a}.
\end{align*}
For the sixth and seventh terms on the RHS of \eqref{6.56}, in view of \eqref{2.5}, Lemma \ref{lem5.3}, \eqref{5.19a}, \eqref{1.4} and \eqref{4.22a}, one gets
\begin{align*}
&|(\frac{\partial^{\alpha}P_{1}(v\cdot\nabla_{x}\overline{G})}{\sqrt{\mu}},\partial^{\alpha}f)|+
|(\frac{\partial^{\alpha}\partial_{t}\overline{G}}{\sqrt{\mu}},\partial^{\alpha}f)|
\\
&\leq \eta\frac{1}{\varepsilon}\|\partial^{\alpha}f\|^{2}
+C_{\eta}\varepsilon(\|\frac{\partial^{\alpha}P_{1}(v\cdot\nabla_{x}\overline{G})}{\sqrt{\mu}}\|^{2}
+\|\frac{\partial^{\alpha}\partial_{t}\overline{G}}{\sqrt{\mu}}\|^{2})
\\
&\leq C\eta\frac{1}{\varepsilon}\|\partial^{\alpha}f\|_{\nu}^{2}+C_{\eta}[\eta_0(1+t)^{-\vartheta}+\varepsilon^{\frac{1}{2}-a}]\varepsilon^{2-2a}.
\end{align*}
For the last term on the RHS of \eqref{6.56}, we can bound it by
\begin{align*}	
&\eta\frac{1}{\varepsilon}\|\partial^{\alpha}f\|^{2}
+C_{\eta}\varepsilon\|\frac{1}{\sqrt{\mu}}\partial^{\alpha}P_{1}\big\{v\cdot(\frac{|v-u|^{2}
\nabla_{x}\widetilde{\theta}}{2R\theta^{2}}+\frac{(v-u)\cdot\nabla_{x}\widetilde{u}}{R\theta})M\big\}\|^{2}	
\\
&\leq C\eta\frac{1}{\varepsilon}\|\partial^{\alpha}f\|_{\nu}^{2}
+ C_{\eta}\varepsilon\|\partial^{\alpha}(\nabla_{x}\widetilde{u},\nabla_{x}\widetilde{\theta})\|^{2}
+C_{\eta}[\eta_0(1+t)^{-\vartheta}+\varepsilon^{\frac{1}{2}-a}]\varepsilon^{2-2a}.
\end{align*}
Consequently, substituting all the above estimates into \eqref{6.56} and choosing $\eta>0$ small enough, then
the summation of the resulting equation over $|\alpha|$ with $|\alpha|\leq N-1$, we have
\begin{align}
\label{6.57}
&\sum_{|\alpha|\leq N-1}(\frac{d}{dt}\|\partial^{\alpha}f\|^{2}+\frac{c_1}{2}\frac{1}{\varepsilon}\|\partial^{\alpha}f\|_{\nu}^{2})
\nonumber\\
\leq& C\varepsilon\sum_{|\alpha|\leq N-1}(\|\partial^{\alpha}(\nabla_{x}\widetilde{u},\nabla_{x}\widetilde{\theta})\|^{2}+\|\partial^{\alpha}\nabla_{x}f\|^{2})
\nonumber\\
&+C(\eta_{0}+\varepsilon^{\frac{1}{2}-a})\mathcal{D}_N(t)+C[\eta_0(1+t)^{-\vartheta}+\varepsilon^{\frac{1}{2}-a}]\varepsilon^{2-2a}.
\end{align}
Integrating \eqref{6.57} with respect to $t$ and using $\mathcal{E}_{N}(0)\leq C\eta^2_{0}\varepsilon^{2}$ as well as $\vartheta>1$,
we can obtain the desired estimate \eqref{6.55}. This completes the proof of Lemma \ref{lem6.7}.
\end{proof}
With the help of Lemma \ref{lem6.6} and  Lemma \ref{lem6.7}, we immediately obtain the space derivative estimates
for both the fluid part and the non-fluid part  up to $(N-1)$-order.
\begin{lemma}\label{lem6.9}
It holds that
\begin{align}
\label{6.61}
&\sum_{|\alpha|\leq N-1}\{\|\partial^{\alpha}(\widetilde{\rho},\widetilde{u},\widetilde{\theta},\widetilde{E},\widetilde{B})(t)\|^2+\|\partial^{\alpha}f(t)\|^{2}\}
\nonumber\\
&\hspace{0.5cm}+c\varepsilon\sum_{1\leq |\alpha|\leq N}\int^t_0\|\partial^{\alpha}(\widetilde{\rho},\widetilde{u},\widetilde{\theta})(s)\|^2\,ds
+c\frac{1}{\varepsilon}\sum_{|\alpha|\leq N-1}\int^t_0\|\partial^{\alpha}f(s)\|_{\nu}^{2}\,ds
\nonumber\\
\leq& 
C\varepsilon^2\sum_{|\alpha|=N}
(\|\partial^{\alpha}(\widetilde{\rho},\widetilde{u},\widetilde{\theta})(t)\|^2+
\|\partial^{\alpha}f(t)\|^2)+ C\varepsilon\sum_{|\alpha|=N}\int^t_0\| \partial^{\alpha}f(s)\|^2_\nu\,ds
\nonumber\\
&+C(\eta_{0}+\varepsilon^{\frac{1}{2}-a})\int^t_0\mathcal{D}_N(s)\,ds+C[\eta_{0}+\varepsilon+(\eta_{0}\varepsilon^a+\varepsilon^{\frac{1}{2}-a})t]\varepsilon^{2-2a}.
\end{align}
\end{lemma}
\begin{proof}
By multiplying \eqref{6.34} with a large positive constant $C$ and then adding the resultant inequality
to \eqref{6.55}, the estimate \eqref{6.61}  follows by letting $\varepsilon$ be small enough.	
We thus finish the proof of Lemma \ref{lem6.9}.
\end{proof}
\subsection{$N$-order space derivative estimates}\label{seca.6.3} 
To complete the estimates on space derivatives of all orders, we still deal with the highest $N$-order space derivatives. 
%For $|\alpha|=N$, we cannot directly use \eqref{4.11} to obtain the dissipation of $\|\partial^{\alpha}f\|^2_{\nu}$  since the estimates on the transport term $(\frac{1}{\sqrt{\mu}}\partial^{\alpha}P_{0}[v\cdot\nabla_{x}(\sqrt{\mu}f)],\partial^{\alpha}f)$ induces $N$+$1$-order
%derivatives so that the estimates cannot be closed. For this, we use the original equation \eqref{4.12} to
%treat the $N$-order space derivatives. In this case, the corresponding  transport  term becomes $(\frac{1}{\sqrt{\mu}}v\cdot\nabla_x\partial^{\alpha}F,\frac{1}{\sqrt{\mu}}\partial^{\alpha}F)$, which vanishes by integration by parts.
\begin{lemma}
\label{lem6.10}
Under the conditions listed in Lemma \ref{lem6.6}, one has
\begin{align}
\label{6.62}
\varepsilon^{2}&\sum_{|\alpha|=N}\{\|\partial^{\alpha}(\widetilde{\rho},\widetilde{u},\widetilde{\theta},\widetilde{E},\widetilde{B})(t)\|^{2}+\|\partial^{\alpha}f(t)\|^{2}\}
+\varepsilon\sum_{|\alpha|=N}\int^{t}_{0}\|\partial^{\alpha}f(s)\|_{\nu}^{2}\,ds
\nonumber\\
\leq&C(\eta_{0}+\varepsilon^{\frac{1}{2}-a})\int^{t}_{0}\mathcal{D}_N(s)\,ds+C[\eta_{0}+\varepsilon^{\frac{1}{2}}+(\eta_{0}\varepsilon^{2a}+\varepsilon^{\frac{1}{2}-a})t]\varepsilon^{2-2a}.
\end{align}
\end{lemma}
\begin{proof}
Applying $\partial^{\alpha}$ to \eqref{4.12} with $|\alpha|=N$ and then taking the inner product of the resulting equation with $\frac{\partial^{\alpha}F}{\sqrt{\mu}}$,
one gets
\begin{align}
\label{6.63}
\frac{1}{2}\frac{d}{dt}&\|\frac{\partial^{\alpha}F}{\sqrt{\mu}}\|^{2}
-(\frac{\partial^{\alpha}[(E+v\times B)\cdot\nabla_{v}F]}{\sqrt{\mu}},\frac{\partial^{\alpha}F}{\sqrt{\mu}})
-\frac{1}{\varepsilon}(\mathcal{L}\partial^{\alpha}f,\frac{\partial^{\alpha}F}{\sqrt{\mu}})
\nonumber\\
=&\frac{1}{\varepsilon}(\partial^{\alpha}\Gamma(\frac{M-\mu}{\sqrt{\mu}},f),\frac{\partial^{\alpha}F}{\sqrt{\mu}})
+\frac{1}{\varepsilon}(\partial^{\alpha}\Gamma(f,\frac{M-\mu}{\sqrt{\mu}}),\frac{\partial^{\alpha}F}{\sqrt{\mu}})
\nonumber\\
&+\frac{1}{\varepsilon}(\partial^{\alpha}\Gamma(\frac{G}{\sqrt{\mu}},\frac{G}{\sqrt{\mu}}),\frac{\partial^{\alpha}F}{\sqrt{\mu}})
+\frac{1}{\varepsilon}(\frac{\partial^{\alpha}L_{M}\overline{G}}{\sqrt{\mu}},\frac{\partial^{\alpha}F}{\sqrt{\mu}}).
\end{align}
We will compute each term for \eqref{6.63}. From \eqref{5.49}, it obviously holds that
\begin{align*}
-&(\frac{\partial^{\alpha}[(E+v\times B)\cdot\nabla_{v}F]}{\sqrt{\mu}},\frac{\partial^{\alpha}F}{\sqrt{\mu}})
\nonumber\\
\geq& \frac{1}{2}\frac{d}{dt}(\frac{1}{R\theta}\partial^{\alpha}\widetilde{E},\partial^{\alpha}\widetilde{E})
+\frac{1}{2}\frac{d}{dt}(\frac{1}{R\theta}\partial^{\alpha}\widetilde{B},\partial^{\alpha}\widetilde{B})
\nonumber\\
&-C(\eta_{0}+\varepsilon^{\frac{1}{2}-a})\frac{1}{\varepsilon^2}\mathcal{D}_{N}(t)
-C[\eta_{0}(1+t)^{-\vartheta}+\eta_{0}\varepsilon^{2a}+\varepsilon^{\frac{1}{2}-a}]\varepsilon^{-2a}.
\end{align*}	
In view of the decomposition $F=M+\overline{G}+\sqrt{\mu}f$, it is easy to see
\begin{align}
\label{6.64}
-\frac{1}{\varepsilon}(\mathcal{L}\partial^{\alpha}f,\frac{\partial^{\alpha}F}{\sqrt{\mu}})=-
\frac{1}{\varepsilon}(\mathcal{L}\partial^{\alpha}f,\frac{\partial^{\alpha}M}{\sqrt{\mu}})
-\frac{1}{\varepsilon}(\mathcal{L}\partial^{\alpha}f,\partial^{\alpha}f)
-\frac{1}{\varepsilon}(\mathcal{L}\partial^{\alpha}f,\frac{\partial^{\alpha}\overline{G}}{\sqrt{\mu}}).
\end{align}
Let's carefully treat \eqref{6.64}. For the first term  on the RHS of \eqref{6.64}, we first note that
$\partial^{\alpha}M=I_{1}+I_{2}+I_{3}$ in terms of \eqref{5.30}. The linear term $\frac{1}{\varepsilon}(\mathcal{L}\partial^{\alpha}f,\frac{I_{1}+I_{2}}{\sqrt{\mu}})$ presents significant difficulty so it  cannot be estimated directly. The key technique to handle this term is to use the properties  of the linearized operator  $\mathcal{L}$ and the smallness of 
$M-\mu$. For this, we denote
\begin{align*}
I_{1}+I_{2}=&\mu\big(\frac{\partial^{\alpha}\rho}{\rho}+\frac{(v-u)\cdot\partial^{\alpha}u}{R\theta}+(\frac{|v-u|^{2}}{2R\theta}-\frac{3}{2})\frac{\partial^{\alpha}\theta}{\theta}\big)
\\
&+(M-\mu)\big(\frac{\partial^{\alpha}\widetilde{\rho}}{\rho}
+\frac{(v-u)\cdot\partial^{\alpha}\widetilde{u}}{R\theta}+(\frac{|v-u|^{2}}{2R\theta}
-\frac{3}{2})\frac{\partial^{\alpha}\widetilde{\theta}}{\theta}\big)
\\
&+(M-\mu)\big(\frac{\partial^{\alpha}\bar{\rho}}{\rho}
+\frac{(v-u)\cdot\partial^{\alpha}\bar{u}}{R\theta}+(\frac{|v-u|^{2}}{2R\theta}
-\frac{3}{2})\frac{\partial^{\alpha}\bar{\theta}}{\theta}\big)
\\
:=&\widetilde{I}^{1}+\widetilde{I}^{2}+\widetilde{I}^{3}.
\end{align*}
Since $\frac{\widetilde{I}^{1}}{\sqrt{\mu}}\in\ker{\mathcal{L}}$, it follows that
$(\mathcal{L}f,\frac{\widetilde{I}^{1}}{\sqrt{\mu}})=0$. 
Recall $\mathcal{L}f=\Gamma(\sqrt{\mu},f)+\Gamma(f,\sqrt{\mu})$ given in \eqref{4.9}, we use the similar arguments of \eqref{4.44} to get
\begin{align}
\label{5.61b}
\frac{1}{\varepsilon}|(\mathcal{L}\partial^{\alpha}f,\frac{\widetilde{I}^{2}}{\sqrt{\mu}})|&\leq C(\eta_{0}+\varepsilon^{1-a})\frac{1}{\varepsilon}(\|\partial^{\alpha}f\|^{2}_{\nu}+\|\partial^{\alpha}(\widetilde{\rho},\widetilde{u},\widetilde{\theta})\|^{2})
\nonumber\\
&\leq C(\eta_{0}+\varepsilon^{1-a})\frac{1}{\varepsilon^{2}}\mathcal{D}_{N}(t).
\end{align}
The same arguments as \eqref{5.32} yield that
\begin{align*}
\frac{1}{\varepsilon}|(\mathcal{L}\partial^{\alpha}f,\frac{\widetilde{I}^{3}}{\sqrt{\mu}})|\leq C(\eta_{0}+\varepsilon^{1-a})\frac{1}{\varepsilon^{2}}\mathcal{D}_{N}(t)+C(\eta_{0}+\varepsilon^{1-a})\varepsilon.
\end{align*}
It follows from the  estimates $\widetilde{I}^{1}$, $\widetilde{I}^2$ and $\widetilde{I}^3$ that
\begin{align}
\label{5.64a}
\frac{1}{\varepsilon}|(\mathcal{L}\partial^{\alpha}f,\frac{I_{1}+I_{2}}{\sqrt{\mu}})|
\leq C(\eta_{0}+\varepsilon^{1-a})\frac{1}{\varepsilon^{2}}\mathcal{D}_{N}(t)+C(\eta_{0}+\varepsilon^{1-a})\varepsilon.
\end{align}
Recall $I_3$ given in \eqref{5.30}, we use the similar arguments as \eqref{5.34} to get
\begin{align*}
\frac{1}{\varepsilon}|(\mathcal{L}\partial^{\alpha}f,\frac{I_{3}}{\sqrt{\mu}})|\leq
C(\eta_{0}+\varepsilon^{\frac{1}{2}-a})\frac{1}{\varepsilon^{2}}\mathcal{D}_{N}(t)+C(\eta_{0}+\varepsilon^{1-a})\varepsilon,
\end{align*}
which along with \eqref{5.64a} and the fact $\partial^{\alpha}M=I_{1}+I_{2}+I_{3}$ shows that
\begin{align*}
\frac{1}{\varepsilon}|(\mathcal{L}\partial^{\alpha}f,\frac{\partial^{\alpha}M}{\sqrt{\mu}})|
\leq C(\eta_{0}+\varepsilon^{\frac{1}{2}-a})\frac{1}{\varepsilon^{2}}\mathcal{D}_N(t)+C(\eta_{0}+\varepsilon^{1-a})\varepsilon.
\end{align*}
For the second term on the RHS of \eqref{6.64}, since $f\in (\ker\mathcal{L})^{\perp}$, one gets by \eqref{4.28} that
$$
-\frac{1}{\varepsilon}(\mathcal{L}\partial^{\alpha}f,\partial^{\alpha}f)\geq c_1\frac{1}{\varepsilon}\|\partial^{\alpha}f\|_{\nu}^{2}.
$$
For the third term on the RHS of \eqref{6.64}, using $\mathcal{L}f=\Gamma(\sqrt{\mu},f)+\Gamma(f,\sqrt{\mu})$, \eqref{4.30} and \eqref{4.20},
one gets
\begin{align*}
\frac{1}{\varepsilon}|(\mathcal{L}\partial^{\alpha}f,\frac{\partial^{\alpha}\overline{G}}{\sqrt{\mu}})|
&\leq C\frac{1}{\varepsilon}\|\partial^{\alpha}f\|_{\nu}\|\frac{\partial^{\alpha}\overline{G}}{\sqrt{\mu}}\|_{\nu}\leq \eta_0{\frac{c_1}{2}}\frac{1}{\varepsilon}\|\partial^{\alpha}f\|^{2}_{\nu}+C\eta_0\varepsilon^{1-2a}.
\end{align*}
Hence, putting the aforementioned  three estimates into \eqref{6.64}, we obtain 
\begin{align}
\label{6.66}
-\frac{1}{\varepsilon}(\mathcal{L}\partial^{\alpha}f,\frac{\partial^{\alpha}F}{\sqrt{\mu}})
\geq& \frac{c_1}{2}\frac{1}{\varepsilon}\|\partial^{\alpha}f\|_{\nu}^{2}-C(\eta_{0}+\varepsilon^{\frac{1}{2}-a})\frac{1}{\varepsilon^{2}}\mathcal{D}_N(t)
-C(\eta_{0}+\varepsilon^{1-a})\varepsilon^{1-2a}.
\end{align}
In view of \eqref{5.28} and \eqref{5.47}, we can bound the first three terms on the RHS of \eqref{6.63}  by
$$
C(\eta_{0}+\varepsilon^{\frac{1}{2}-a})\frac{1}{\varepsilon^{2}}\mathcal{D}_{N}(t)+C(\eta_{0}+\varepsilon^{1-a})\varepsilon^{\frac{1}{2}-a}.
$$	
For the last term on the RHS of \eqref{6.63}, by \eqref{4.6}, \eqref{2.5}, \eqref{4.49} and \eqref{4.22a}, it holds that
\begin{align*}
\frac{1}{\varepsilon}(\frac{\partial^{\alpha}L_{M}\overline{G}}{\sqrt{\mu}},\frac{\partial^{\alpha}F}{\sqrt{\mu}})
&\leq C\|\frac{1}{\sqrt{\mu}}\partial^{\alpha}P_{1}\big\{v\cdot(\frac{|v-u|^{2}\nabla_x\overline{\theta}}{2R\theta^{2}}+\frac{(v-u)\cdot\nabla_x\bar{u}}{R\theta}) M\big\}\|\|\frac{\partial^{\alpha}F}{\sqrt{\mu}}\|
%\\
%&\leq C(\|\partial^{\alpha}f\|^{2}_{\nu}+\|\partial^{\alpha}(\widetilde{\rho},\widetilde{u},\widetilde{\theta})\|^{2}+\eta_{0}+\varepsilon^{1-a})
\\
&\leq C\frac{1}{\varepsilon}\mathcal{D}_N(t)+C(\eta_{0}+\varepsilon^{1-a}).
\end{align*}
In summary, substituting all the above estimates into \eqref{6.63} and using \eqref{4.22a}, we get
\begin{align}
\label{6.67}
\frac{1}{2}&\frac{d}{dt}\sum_{|\alpha|=N}\{\|\frac{\partial^{\alpha}F}{\sqrt{\mu}}\|^{2}
+(\frac{1}{R\theta}\partial^{\alpha}\widetilde{E},\partial^{\alpha}\widetilde{E})
+(\frac{1}{R\theta}\partial^{\alpha}\widetilde{B},\partial^{\alpha}\widetilde{B})\}
+c\frac{1}{\varepsilon}\sum_{|\alpha|=N}\|\partial^{\alpha}f\|_{\nu}^{2}
\nonumber\\
&\leq C(\eta_{0}+\varepsilon^{\frac{1}{2}-a})\frac{1}{\varepsilon^2}\mathcal{D}_{N}(t)
+C[\eta_{0}(1+t)^{-\vartheta}+\eta_{0}\varepsilon^{2a}+\varepsilon^{\frac{1}{2}-a}]\varepsilon^{-2a}.	
\end{align}
Multiplying \eqref{6.67} by $\varepsilon^{2}$ and integrating the resulting equation 
with respect to $t$, we can obtain the desired estimate \eqref{6.62} by using
%\begin{align}
%\label{5.47A}
%\varepsilon^{2}&\sum_{|\alpha|=N}(\|\frac{\partial^{\alpha}F(t)}{\sqrt{\mu}}\|^{2}+\|\partial^{\alpha}(\widetilde{E},\widetilde{B})(t)\|^{2})+\varepsilon\sum_{|\alpha|=N}\int^{t}_{0}\|\partial^{\alpha}f(s)\|_{\nu}^{2}\,ds
%\nonumber\\
%\leq& C\varepsilon^{2}\sum_{|\alpha|=N}(\|\frac{\partial^{\alpha}F(0)}{\sqrt{\mu}}\|^{2}+\|\partial^{\alpha}(\widetilde{E},\widetilde{B})(0)\|^{2})+
%C(\eta_{0}+\varepsilon^{\frac{1}{2}-a})\int^{t}_{0}\mathcal{D}_N(s)\,ds
%\nonumber\\
%&+C[\eta_{0}+(\eta_{0}\varepsilon^{2a}+\varepsilon^{\frac{1}{2}-a})t]\varepsilon^{2-2a}.
%\end{align}
\eqref{4.85A} and $\mathcal{E}_{N}(0)\leq C\eta^2_{0}\varepsilon^{2}$ as well as the following the fact
\begin{align*}
\varepsilon^{2}\sum_{|\alpha|=N}\|\frac{\partial^{\alpha}F(0)}{\sqrt{\mu}}\|^{2}
\leq C(\eta_{0}+\varepsilon^{\frac{1}{2}-a})\varepsilon^{2}.
\end{align*}
This completes the proof of Lemma \ref{lem6.10}.
\end{proof}
Combining those estimates in  Lemma \ref{lem6.9} and Lemma \ref{lem6.10}, we are able to 
conclude the energy estimate on solutions without velocity derivatives.
\begin{lemma}\label{lem6.12}
It holds that
\begin{align}
\label{6.70}
&\sum_{|\alpha|\leq N-1}\{\|\partial^{\alpha}(\widetilde{\rho},\widetilde{u},\widetilde{\theta},\widetilde{E},\widetilde{B})(t)\|^2+\|\partial^{\alpha}f(t)\|^{2}\}
+\varepsilon^{2}\sum_{|\alpha|=N}\{\|\partial^{\alpha}(\widetilde{\rho},\widetilde{u},\widetilde{\theta},\widetilde{E},\widetilde{B})(t)\|^2+\|\partial^{\alpha}f(t)\|^{2}\}
\nonumber\\
&+c\int^{t}_{0}\{\varepsilon\sum_{1\leq|\alpha|\leq N}\|\partial^{\alpha}(\widetilde{\rho},\widetilde{u},\widetilde{\theta})(s)\|^{2}
+\frac{1}{\varepsilon}\sum_{|\alpha|\leq N-1}\|\partial^{\alpha}f(s)\|_{\nu}^{2}
+\varepsilon\sum_{|\alpha|=N}\|\partial^{\alpha}f(s)\|_{\nu}^{2}\}\,ds
\nonumber\\
&\leq C(\eta_{0}+\varepsilon^{\frac{1}{2}-a})\int_0^t\mathcal{D}_{N}(s)\,ds
+C[\eta_{0}+\varepsilon^{\frac{1}{2}}+(\eta_{0}\varepsilon^a+\varepsilon^{\frac{1}{2}-a})t]\varepsilon^{2-2a}.
\end{align}
\end{lemma}

\subsection{Mixed derivative estimates}\label{sec6.4}
Notice that Lemma \ref{lem6.12} does not include the velocity derivative, we shall derive the mixed derivative estimates on solution
in this subsection. 
\begin{lemma}\label{lem6.13}
Under the conditions listed in Lemma \ref{lem6.6}, one has
\begin{align}
\label{6.71}
&\sum_{|\alpha|+|\beta|\leq N,|\beta|\geq1}\big\{\|\partial_{\beta}^{\alpha}f(t)\|^{2}
+c\frac{1}{\varepsilon}\int^{t}_{0}\|\partial_{\beta}^{\alpha}f(s)\|_{\nu}^{2}\,ds\big\}
\nonumber\\
\leq& C\int^{t}_{0}\{\varepsilon\sum_{1\leq|\alpha|\leq N}\|\partial^{\alpha}(\widetilde{u},\widetilde{\theta})(s)\|^{2}
+\frac{1}{\varepsilon}\sum_{|\alpha|\leq N-1}\|\partial^{\alpha}f(s)\|_{\nu}^{2}
+\varepsilon\sum_{|\alpha|=N}\|\partial^{\alpha}f(s)\|_{\nu}^{2}\}\,ds
\nonumber\\
&+C(\eta_{0}+\varepsilon^{\frac{1}{2}-a})\int^t_0\mathcal{D}_N(s)\,ds+C(\eta_{0}+\varepsilon^{1-a}t)\varepsilon^{2-2a}.
\end{align}
\end{lemma}

\begin{proof}
Let $|\alpha|+|\beta|\leq N$ and $|\beta|\geq1$, we apply $\partial_{\beta}^{\alpha}$ to \eqref{4.11} and take the inner product of the resulting equation with $\partial^{\alpha}_{\beta}f$ to obtain
\begin{align}
\label{6.72}
&\frac{1}{2}\frac{d}{dt}\|\partial_{\beta}^{\alpha}f\|^{2}+(v\cdot\nabla_{x}\partial^{\alpha}_{\beta}f,\partial^{\alpha}_{\beta}f)
+(C^{\beta-e_{i}}_{\beta}\delta^{e_{i}}_{\beta}\partial^{\alpha+e_{i}}_{\beta-e_{i}}f,\partial^{\alpha}_{\beta}f)-\frac{1}{\varepsilon}(\partial^{\alpha}_{\beta}\mathcal{L}f,\partial^{\alpha}_{\beta}f)
\nonumber\\
=&\frac{1}{\varepsilon}(\partial^{\alpha}_{\beta}\Gamma(\frac{M-\mu}{\sqrt{\mu}},f),\partial^{\alpha}_{\beta}f)
+\frac{1}{\varepsilon}(\partial^{\alpha}_{\beta}\Gamma(f,\frac{M-\mu}{\sqrt{\mu}}),\partial^{\alpha}_{\beta}f)
\nonumber\\
&+\frac{1}{\varepsilon}(\partial^{\alpha}_{\beta}\Gamma(\frac{G}{\sqrt{\mu}},\frac{G}{\sqrt{\mu}}),
\partial^{\alpha}_{\beta}f)+(\partial^{\alpha}_{\beta}[\frac{(E+v\times B)\cdot\nabla_{v}(\sqrt{\mu}f+\overline{G})}{\sqrt{\mu}}],\partial^{\alpha}_{\beta}f)
\nonumber\\
&-(\partial^{\alpha}_{\beta}[\frac{P_{1}(v\cdot\nabla_{x}\overline{G})}{\sqrt{\mu}}],\partial^{\alpha}_{\beta}f)
-(\partial^{\alpha}_{\beta}(\frac{\partial_{t}\overline{G}}{\sqrt{\mu}}),\partial^{\alpha}_{\beta}f)
+(\partial^{\alpha}_{\beta}[\frac{P_{0}(v\sqrt{\mu}\cdot\nabla_{x}f)}{\sqrt{\mu}}],\partial^{\alpha}_{\beta}f)
\nonumber\\
&-(\partial^{\alpha}_{\beta}\{\frac{1}{\sqrt{\mu}}P_{1}[v\cdot(\frac{|v-u|^{2}\nabla_{x}\widetilde{\theta}}{2R\theta^{2}}+\frac{(v-u)\cdot\nabla_{x}\widetilde{u}}{R\theta})M]\},\partial^{\alpha}_{\beta}f).
\end{align}
Here $\delta^{e_{i}}_{\beta}=1$ if $e_{i}\leq\beta$ or $\delta^{e_{i}}_{\beta}=0$ otherwise.
	
We will compute each term for \eqref{6.72}. The second term on the LHS of \eqref{6.72}
vanishes by integration by parts, while the third term on the LHS of \eqref{6.72} can be estimated as
\begin{align*}
|(C^{\beta-e_{i}}_{\beta}\delta^{e_{i}}_{\beta}\partial^{\alpha+e_{i}}_{\beta-e_{i}}f,\partial_{\beta}^{\alpha}f)|
\leq \eta\frac{1}{\varepsilon}\|\partial^{\alpha}_{\beta}f\|^{2}+C_{\eta}\varepsilon\|\partial^{\alpha+e_{i}}_{\beta-e_{i}}f\|^{2}.
\end{align*}
For the fourth term on the LHS of \eqref{6.72}, one has from \eqref{4.29} that
\begin{align*}
-\frac{1}{\varepsilon}(\partial_{\beta}^{\alpha}\mathcal{L}f,\partial_{\beta}^{\alpha}f)
\geq c_2\frac{1}{\varepsilon}\|\partial^{\alpha}_{\beta}f\|^{2}_{\nu}-C\frac{1}{\varepsilon}\|\partial^{\alpha}f\|_{\nu}^{2}.
\end{align*}
The first three terms on the RHS of \eqref{6.72} can be
controlled by $C(\eta_{0}+\varepsilon^{\frac{1}{2}-a})\mathcal{D}_{N}(t)+C\varepsilon\varepsilon^{2-2a}$
in terms of \eqref{5.20} and \eqref{5.41}.
The fourth, fifth and sixth terms on the RHS of \eqref{6.72} are bounded by
\begin{align*}
 C\eta\frac{1}{\varepsilon}\|\partial_{\beta}^{\alpha}f\|_{\nu}^{2}+C(\eta_{0}+\varepsilon^{\frac{1}{2}-a})\mathcal{D}_{N}(t)
+C_{\eta}[\eta_0(1+t)^{-\vartheta}+\varepsilon^{1-a}]\varepsilon^{2-2a}.
\end{align*}
% For the fifth term on the RHS of \eqref{6.72}, by \eqref{2.5}, \eqref{4.25}, \eqref{5.38},
%\eqref{3.3} and \eqref{4.13}, one has
%\begin{align*}
%|(\partial^{\alpha}_{\beta}[\frac{1}{\sqrt{\mu}}P_{0}(v\sqrt{\mu}\cdot\nabla_{x}f)],\partial^{\alpha}_{\beta}f)|
%&=\big|\sum_{j=0}^{4}(\partial_{\beta}^{\alpha}
%[\langle v\sqrt{\mu}\cdot\nabla_{x}f,\frac{\chi_{j}}{M}\rangle\frac{\chi_{j}}{\sqrt{\mu}}],\partial^{\alpha}_{\beta}f)\big|
%\\&\leq 
%\eta\frac{1}{\varepsilon}\|\partial^{\alpha}_{\beta}f\|^{2}_{\nu}+C_{\eta}\varepsilon
%\|\nabla_x\partial^{\alpha}f\|^{2}_{\nu}+C_{\eta}\varepsilon\varepsilon^{2-2a}.
%\end{align*}
%For the sixth and seventh terms on the RHS of \eqref{6.72}, using Lemma \ref{lem5.3}, \eqref{5.19a} and \eqref{4.22a}, one gets
%\begin{align*}
%|(\partial^{\alpha}_{\beta}[\frac{P_{1}(v\cdot\nabla_{x}\overline{G})}{\sqrt{\mu}}],\partial^{\alpha}_{\beta}f)|
%+|(\partial^{\alpha}_{\beta}(\frac{\partial_{t}\overline{G}}{\sqrt{\mu}}),\partial^{\alpha}_{\beta}f)|
%\leq C\eta\frac{1}{\varepsilon}\|\partial^{\alpha}_{\beta}f\|^{2}_{\nu}+C_{\eta}\varepsilon\varepsilon^{2-2a}.
%\end{align*}
For the last two terms of \eqref{6.72}, we can bound them by
\begin{align*}
C\eta\frac{1}{\varepsilon}\|\partial^{\alpha}_{\beta}f\|^{2}_{\nu}+C_{\eta}\varepsilon
\|(\nabla_{x}\partial^{\alpha}\widetilde{u},\nabla_{x}\partial^{\alpha}\widetilde{\theta})\|^{2}
+C_{\eta}\varepsilon
\|\nabla_x\partial^{\alpha}f\|^{2}_{\nu}
+C_{\eta}\varepsilon\varepsilon^{2-2a}.
\end{align*}
In summary, for all $|\alpha|+|\beta|\leq N$ and $|\beta|\geq1$,
substituting all the above estimates  into \eqref{6.72} and choosing a small $\eta>0$, then
%\begin{align}
%\label{6.73}
%&\frac{1}{2}\frac{d}{dt}\|\partial_{\beta}^{\alpha}f\|^{2}+c\frac{1}{\varepsilon}\|\partial^{\alpha}_{\beta}f\|^{2}_{\nu}
%\nonumber\\
%\leq& C\varepsilon\|\partial^{\alpha+e_{i}}_{\beta-e_{i}}f\|^{2}+C\frac{1}{\varepsilon}\|\partial^{\alpha}f\|_{\nu}^{2}+C\varepsilon
%(\|(\nabla_{x}\partial^{\alpha}\widetilde{u},\nabla_{x}\partial^{\alpha}\widetilde{\theta})\|^{2}+\|\nabla_{x}\partial^{\alpha}f\|_{\nu}^{2})
%\nonumber\\
%&+C(\eta_{0}+\varepsilon^{\frac{1}{2}-a})\mathcal{D}_N(t)+C[\eta_0(1+t)^{-\vartheta}+\varepsilon^{1-a}]\varepsilon^{2-2a}.
%\end{align}
the suitable linear combination of resulting equations, one gets
\begin{align}
\label{6.74}
&\sum_{|\alpha|+|\beta|\leq N,|\beta|\geq1}\{\frac{d}{dt}\|\partial_{\beta}^{\alpha}f\|^{2}
+c\frac{1}{\varepsilon}\|\partial_{\beta}^{\alpha}f\|_{\nu}^{2}\}
\nonumber\\
\leq&C\varepsilon\sum_{1\leq|\alpha|\leq N}(\|\partial^{\alpha}f\|_{\nu}^{2}+\|\partial^{\alpha}(\widetilde{u},\widetilde{\theta})\|^{2})+C\frac{1}{\varepsilon}\sum_{|\alpha|\leq N-1}\|\partial^{\alpha}f\|_{\nu}^{2}
\nonumber\\
&+C(\eta_{0}+\varepsilon^{\frac{1}{2}-a})\mathcal{D}_N(t)+C[\eta_0(1+t)^{-\vartheta}+\varepsilon^{1-a}]\varepsilon^{2-2a}.
\end{align}
Integrating \eqref{6.74} with respect to $t$ and using $\mathcal{E}_{N}(0)\leq C\eta^2_{0}\varepsilon^{2}$ as well as $\vartheta>1$,
 we can obtain the desired estimate \eqref{6.71}. This completes the proof of Lemma \ref{lem6.13}.
\end{proof}

\section{Proofs of the main theorem }\label{seca.6}
Based on those estimates in Lemma \ref{lem6.12} and Lemma \ref{lem6.13}, in this section, we are now in a position to
complete the

\medskip
\noindent{\it Proof of Theorem \ref{thm3.1}:}
By a suitable linear combination of \eqref{6.70} and \eqref{6.71}, there exists $C_1>0$ as stated in Theorem \ref{thm3.1} 
such that
\begin{align}
\label{6.1a}
\mathcal{E}_{N}(t)+\int^{t}_{0}\mathcal{D}_{N}(s)\,ds\leq& C_1(\eta_{0}+\varepsilon^{\frac{1}{2}-a})\int_0^t\mathcal{D}_{N}(s)\,ds
\nonumber\\
&+C_1[\eta_{0}+\varepsilon^{\frac{1}{2}}+(\eta_{0}\varepsilon^a+\varepsilon^{\frac{1}{2}-a})t]\varepsilon^{2-2a}.
\end{align}
Here $\mathcal{E}_N(t)$ and  $\mathcal{D}_N(t)$ are given by \eqref{4.14} and \eqref{4.15}, respectively. 
In order to close the a priori assumption \eqref{4.13}, we need to require that
\begin{align}
\label{6.2a}
C_1(\eta_{0}+\varepsilon^{\frac{1}{2}-a})\leq \frac{1}{2}, \quad C_1[\eta_{0}+\varepsilon^{\frac{1}{2}}+(\eta_{0}\varepsilon^a+\varepsilon^{\frac{1}{2}-a})t]
\leq \frac{1}{2}.
\end{align}
In fact, the desired estimate \eqref{6.2a} holds true by choosing both $\eta_{0}$ and $\varepsilon$ small enough and taking
\begin{align}
\label{6.3a}
a\in[0,\frac{1}{2}),\quad \mbox{and}\quad t\leq T_{max}=\frac{1}{4C_1}\frac{1}{\eta_{0}\varepsilon^a+\varepsilon^{\frac{1}{2}-a}}.
\end{align}
That is why one has to choose the range of the parameter $a$ and time $T_{max}$ in \eqref{4.13} which we start with.
Thanks to \eqref{6.1a} and \eqref{6.2a}, we immediately obtain
\begin{align}
\label{6.4a}
\mathcal{E}_{N}(t)+\frac{1}{2}\int^{t}_{0}\mathcal{D}_{N}(s)\,ds\leq \frac{1}{2}\varepsilon^{2-2a}.
\end{align}
Then \eqref{6.4a} implies that for $a\in[0,\frac{1}{2})$ and $T\in(0,T_{max}]$,  one has
\begin{align}
\label{6.5a}
\sup_{0\leq t\leq T}\mathcal{E}_{N}(t)\leq  \frac{1}{2}\varepsilon^{2-2a},
\end{align}
which is strictly stronger than \eqref{4.13}. Hence we have closed the a priori assumption \eqref{4.13}.

Since the local existence of the solutions to the VMB system  near a global Maxwellian is well known in \cite{Guo-2003}
for the torus and in \cite{Strain-2006} for the whole space,
by a straightforward modification of the arguments there one can construct a unique short time solution to the 
VMB system \eqref{1.1} and \eqref{3.6} under the assumptions in Theorem \ref{thm3.1}. The
details are omitted for simplicity of presentation. 
Therefore, by the uniform a priori estimates and the local existence of
the solution, the standard continuity argument, we can immediately derive the existence
and uniqueness of almost global-in-time smooth solutions to the  VMB system \eqref{1.1} and \eqref{3.6}.
Moreover, the desired estimate \eqref{3.8} holds true in terms of \eqref{6.4a}.

To finish the proof of Theorem \ref{thm3.1}, we are going to justify the a uniform convergence rate as in \eqref{3.7}. 
Making use of the Taylor expansion, \eqref{4.25}, \eqref{6.5a} and the embedding inequality, we obtain
\begin{align*}
\sup_{t\in[0,T_{max}]}\|\frac{(M_{[\rho,u,\theta]}-M_{[\bar{\rho},\bar{u},\bar{\theta}]})(t)}{\sqrt{\mu}}\|_{L_{x}^{\infty}L_{v}^{2}}
\leq C\sup_{t\in[0,T_{max}]}\|(\widetilde{\rho},\widetilde{u},\widetilde{\theta})(t)\|_{L_{x}^{\infty}}\leq 
C\varepsilon^{1-a}.
\end{align*}
Similarly, we get by Lemma \ref{lem5.3} and \eqref{6.5a}  that
\begin{equation*}
\sup_{t\in[0,T_{max}]}(\|\frac{\overline{G}(t)}{\sqrt{\mu}}\|_{L_{x}^{\infty}L_{v}^{2}}+\|f(t)\|_{L^{\infty}_{x}L^{2}_{v}})\leq C\varepsilon^{1-a}.
\end{equation*}
Combining the aforementioned two estimates and using $F=M+\overline{G}+\sqrt{\mu}f$, we get
\begin{align}
\label{6.6a}
\sup_{t\in[0,T_{max}]}\|\frac{(F-M_{[\bar{\rho},\bar{u},\bar{\theta}]})(t)}{\sqrt{\mu}}\|_{L_{x}^{\infty}L_{v}^{2}}
\leq C\varepsilon^{1-a}.
\end{align}
Similar estimate also holds for $\sup_{t\in[0,T_{max}]}\|\frac{(F-M_{[\bar{\rho},\bar{u},\bar{\theta}]})(t)}{\sqrt{\mu}}\|$.
Moreover, one has
\begin{align*}
\sup_{t\in[0,T_{max}]}\{\|(E-\bar{E},B-\bar{B})(t,x)\|+\|(E-\bar{E},B-\bar{B})(t,x)\|_{L_{x}^{\infty}}\}\leq C\varepsilon^{1-a},
\end{align*}
which together with \eqref{6.6a} and the fact that $(F,E,B)=(F^{\varepsilon},E^{\varepsilon},B^{\varepsilon})$,
immediately gives the desired estimate \eqref{3.7}.
And the proof of Theorem \ref{thm3.1} is complete. \qed

\medskip
\noindent {\bf Acknowledgment:}\,
The research of Renjun Duan was partially supported by the General Research Fund (Project No.~14301719) from RGC of Hong Kong and a Direct Grant (4053567) from CUHK. The research of Hongjun Yu was supported by the GDUPS 2017 and the NNSFC Grant 11371151. Dongcheng Yang would like to thank the Department of Mathematics, CUHK  for hosting his visit in the period 2020-23. The authors would like to thank the referee for all valuable and helpful comments on the manuscript.

\medskip

\noindent{\bf Conflict of Interest:} The authors declare that they have no conflict of interest.

\end{document}